\documentclass[reqno, 11pt, letterpaper, oneside]{article}

\usepackage{amsmath,amsthm,amssymb} %
\usepackage{geometry}
\usepackage[final]{graphicx}
\usepackage{setspace}

\newtheorem{theorem}{Theorem}[section]
\newtheorem{lemma}[theorem]{Lemma}
\newtheorem{corollary}[theorem]{Corollary}

\newcommand{\cA}{{\mathcal A}}
\newcommand{\cB}{{\mathcal B}}
\newcommand{\cC}{{\mathcal C}}
\newcommand{\cD}{{\mathcal D}}
\newcommand{\cE}{{\mathcal E}}
\newcommand{\cF}{{\mathcal F}}
\newcommand{\cG}{{\mathcal G}}
\newcommand{\cH}{{\mathcal H}}
\newcommand{\cI}{{\mathcal I}}
\newcommand{\cJ}{{\mathcal J}}
\newcommand{\cK}{{\mathcal K}}
\newcommand{\cL}{{\mathcal L}}
\newcommand{\cN}{{\mathcal N}}
\newcommand{\cM}{{\mathcal M}}
\newcommand{\cP}{{\mathcal P}}
\newcommand{\cQ}{{\mathcal Q}}
\newcommand{\cR}{{\mathcal R}}
\newcommand{\cS}{{\mathcal S}}
\newcommand{\cU}{{\mathcal U}}
\newcommand{\cT}{{\mathcal T}}
\newcommand{\cV}{{\mathcal V}}
\newcommand{\cW}{{\mathcal W}}
\newcommand{\cX}{{\mathcal X}}
\newcommand{\cY}{{\mathcal Y}}

\newcommand{\EE}{{\mathbb E}}
\newcommand{\PP}{{\mathbb P}}

\newcommand{\badE}[1][i]{\cB_{{#1}}(\Sigma)}
\newcommand{\badEC}[1][i]{\cB_{{#1}}(\Sigma^*)}
\newcommand{\badEL}[1][i]{\cB_{\leq {#1}}(\Sigma)}

\newcommand\marginal[1]{\marginpar{\raggedright\parindent=0pt\tiny #1}}
\renewcommand\marginal[1]{}

\renewcommand{\epsilon}{\varepsilon}
\newcommand\arxiv[1]{\texttt{\def~{{\tiny$\sim$}}{arXiv:#1}}} 

\begin{document}

\begin{center}

{\LARGE The $C_{\ell}$-free process}

\vspace{3mm}

{\large Lutz Warnke} 
\vspace{1mm}

{ Mathematical Institute,  University of Oxford\\
 24--29 St.~Giles', Oxford OX1 3LB, UK\\
 {\small\tt warnke@maths.ox.ac.uk}}

\vspace{6mm}

\small
\begin{minipage}{0.8\linewidth}
\textsc{Abstract.} 
The $C_{\ell}$-free process starts with the empty graph on $n$ vertices and adds edges chosen uniformly at random, one at a time, subject to the condition that no copy of $C_{\ell}$ is created. 
For every $\ell \geq 4$ we show that, with high probability as $n \to \infty$, the maximum degree is $O( (n \log n)^{1/(\ell-1)})$, which confirms a conjecture of Bohman and Keevash and improves on bounds of Osthus and Taraz.  
Combined with previous results this implies that the $C_{\ell}$-free process typically terminates with $\Theta(n^{\ell/(\ell-1)}(\log n)^{1/(\ell-1)})$ edges, which answers a question of Erd{\H o}s, Suen and Winkler. 
This is the first result that determines the final number of edges of the more general $H$-free process for a non-trivial \emph{class} of graphs $H$. 
We also verify a conjecture of Osthus and Taraz concerning the average degree, and obtain a new lower bound on the independence number. 
Our proof combines the differential equation method with a tool that might be of independent interest:  
we establish a rigorous way to `transfer' certain decreasing properties from the binomial random graph to the $H$-free process. 
\end{minipage}
\end{center} 
\normalsize

\section{Introduction} 
The \emph{random graph process} was introduced by Erd{\H{o}}s and R{\'e}nyi~\cite{ErdosRenyi1959} in 1959.  
It starts with the empty graph on $n$ vertices and adds new edges one by one, where each edge is chosen uniformly at random among all edges not yet present. 
Since then it has been studied extensively, and many tools and methods for investigating its typical properties have been developed, see e.g.~\cite{Bollobas2001RandomGraphs,Durrett2007,JLR2000RandomGraphs}. 
In this work we consider a natural variant of the above process which has very recently received a considerable amount of attention~\cite{Bohman2009K3, BohmanKeevash2010H, GerkeMakai2010K3, Picollelli2010K4Minus, Picollelli2010C4, Warnke2010H, Warnke2010K4, Wolfovitz2009H, Wolfovitz2010K4, Wolfovitz2009K3Subgraph}.

The \emph{$H$-free process} was suggested by Bollob{\'a}s and Erd{\H o}s~\cite{Bollobas2010PC} in 1990, as a way to generate an interesting probability distribution on the set of maximal $H$-free graphs with potential applications to Ramsey Theory. 
Given some fixed \mbox{graph $H$}, it is a modification of the classical random graph process, where each new edge is chosen uniformly at random subject to the condition that no copy of $H$ is formed. 
It was first described in print in 1995 by Erd{\H o}s, Suen and Winkler~\cite{ErdoesSuenWinkler1995}, who asked how many edges the final graph typically has (this also appears as a problem in~\cite{ChungGraham98}). 
The main difficulty when analysing this process  is that there is a complicated dependence among the edges; the order in which they are inserted is also relevant.

The first results addressed certain special graphs, determining the typical final number of edges up to logarithmic factors. 
The case $H=C_3$ was studied in 1995 by Erd{\H o}s, Suen and Winkler~\cite{ErdoesSuenWinkler1995}, and in 2000 Bollob{\'a}s and Riordan~\cite{BollobasRiordan2000} considered $H \in \{K_4,C_4\}$. 
In fact, a result of Ruci{\'n}ski and Wormald~\cite{RucinskiWormald1992} predates those mentioned above: 
in 1992 they considered the (much simpler) maximum degree $d$-process, which corresponds to the case $H=K_{1,d+1}$, and showed that whp\footnote{As usual, we say that an event holds \emph{with high probability}, or \emph{whp}, if it holds with probability $1-o(1)$ as $n\to\infty$.} it ends with $\lfloor nd/2\rfloor$ edges. 
The general $H$-free process was first analysed independently by Bollob{\'a}s and Riordan~\cite{BollobasRiordan2000} and Osthus and Taraz~\cite{OsthusTaraz2001} in 2000. 
In fact, they assumed that $H$ satisfies a certain density condition (strictly $2$-balanced), which holds for many  interesting graphs, including cycles and complete graphs. 
Osthus and Taraz determined the typical final number of edges up to logarithmic factors and conjectured that whp the average degree in the final graph of the $C_{\ell}$-free process is $\Theta((n \log n)^{1/(\ell-1)})$. 

The next improvements came about ten years later. 
In a breakthrough in 2009, Bohman~\cite{Bohman2009K3} obtained the first matching bounds: he proved that the $C_3$-free process ends whp with $\Theta(n^{3/2} \sqrt{\log n})$ edges, confirming a conjecture of Spencer~\cite{Spencer1995}. 
Next, Wolfovitz~\cite{Wolfovitz2009H} slightly improved the lower bound on the expected final number of edges for a range of graphs $H$. 
Very recently, for the class of strictly $2$-balanced \mbox{graphs $H$}, Bohman and Keevash~\cite{BohmanKeevash2010H} obtained new lower bounds that hold whp, which they conjectured to be tight up to the constants. In fact, their conjecture is for the maximum degree: for the $C_{\ell}$-free process they conjectured that the maximum degree is whp at most $D (n \log n)^{1/(\ell-1)}$ for some $D>0$.

As one can see, the typical final number of edges in the $H$-free process has attracted a lot of attention, and for a large class of \mbox{graphs $H$} interesting bounds are known. 
However, not much progress has been made in obtaining good upper bounds. 
After Bohman's result for $C_3$, the next case to be resolved was $H=K_4$, for which matching bounds have been obtained by the author~\cite{Warnke2010K4}, and, independently, by Wolfovitz~\cite{Wolfovitz2010K4}. 
During the preparation of this paper Picollelli~\cite{Picollelli2010K4Minus,Picollelli2010C4} also resolved the cases $H \in \{C_4,K_{4}^{-}\}$. 
But despite this progress, since the upper bound for the maximum degree $d$ process in~\cite{RucinskiWormald1992} is immediate, one can argue that non-trivial matching upper bounds have not been determined for any \emph{class} of graphs.

The $H$-free process is nowadays considered a model of independent interest as well. 
For strictly $2$-balanced $H$, the early evolution of various graph parameters, including the degree and the number of small subgraphs, has been investigated in~\cite{BohmanKeevash2010H,Wolfovitz2009K3Subgraph}. 
These results suggest that, perhaps surprisingly, during this initial phase the graph produced by the $H$-free process is very similar to the uniform random graph with the same number of edges, although it contains no copy of $H$. 
Studying the typical structural properties, e.g.\ the degree, in the later evolution of the $H$-free process is an intriguing problem, and so far only some preliminary results are known, cf.~\cite{GerkeMakai2010K3,Warnke2010H}.

Motivation for studying the $H$-free process also comes from extremal combinatorics, where its analysis has produced several new results. 
For example, improved lower bounds on the Tur{\'a}n numbers of certain bipartite graphs and Ramsey numbers $R(s,t)$ with $s \geq 4$ have been established in~\cite{Bohman2009K3,BohmanKeevash2010H,Wolfovitz2009H}, and Bohman~\cite{Bohman2009K3} reproved the famous lower bound for $R(3,t)$ obtained by Kim~\cite{Kim95}. 
One of the key ingredients for these results is an upper bound on the independence number of the $H$-free process, cf.~\cite{Bohman2009K3,BohmanKeevash2010H}. 
So far only for the special cases $H \in \{C_3,C_4\}$ are these estimates known to be best possible, and it would be interesting to obtain good lower bounds for other graphs.

\subsection{Main result}
In this paper we prove a new upper bound on the final number of edges of the $C_{\ell}$-free process. 
In fact, we give a new upper bound for the maximum degree, which confirms a conjecture of Bohman and Keevash~\cite{BohmanKeevash2010H} and improves previous upper bounds by Osthus and Taraz~\cite{OsthusTaraz2001}. 
\begin{theorem}%
\label{thm:main_result}
For every $\ell \geq 4$ there exists $D>0$ such that whp the maximum degree in the final graph of the $C_{\ell}$-free process is at most $D (n \log n)^{1/(\ell-1)}$. 
\end{theorem}
Up to the constant our upper bound is best possible, since the results of Bohman and Keevash~\cite{BohmanKeevash2010H} imply that for some $c > 0$, whp the minimum degree is at least $c (n \log n)^{1/(\ell-1)}$. 
The special case $\ell=4$ was proved independently by Picollelli~\cite{Picollelli2010C4}; since this manuscript was submitted Picollelli~\cite{Picollelli2011Cl} has independently also proved the case $\ell \geq 4$. 
So, combining our findings with~\cite{BohmanKeevash2010H}, we not only verify the mentioned conjecture of Osthus and Taraz~\cite{OsthusTaraz2001}, but establish the following stronger result. 
\begin{corollary}%
\label{cor:main_result:degree_edge}
For every $\ell \geq 4$ there exist $c,D > 0$ such that in the final graph of the $C_{\ell}$-free process whp the number of edges is between $cn^{\ell/(\ell-1)} (\log n)^{1/(\ell-1)}$ and $Dn^{\ell/(\ell-1)} (\log n)^{1/(\ell-1)}$, and whp the degree of every vertex is between $c (n \log n)^{1/(\ell-1)}$ and $D (n \log n)^{1/(\ell-1)}$. \qed 
\end{corollary}
This is a natural extension of the main  result of Bohman~\cite{Bohman2009K3} for the $C_{3}$-free process, 
and answers a question of Erd{\H o}s, Suen and Winkler for the $C_{\ell}$-free process (see~\cite{ChungGraham98,ErdoesSuenWinkler1995}): whp the final graph has $\Theta(n^{\ell/(\ell-1)} (\log n)^{1/(\ell-1)})$ edges. 
Since this question was asked for the $H$-free process in 1995, this is the first result that determines (up to constants) the final number of edges for a \emph{class} of graphs.

We also obtain a new lower bound on the independence number of the $C_{\ell}$-free process. 
Indeed, as pointed out to us by Picollelli, using Corollary~$2.4$ of Alon, Krivelevich and Sudakov~\cite{AlonKrivelevichSudakov1999}, Corollary~\ref{cor:main_result:degree_edge} implies the following bound conjectured in an earlier version of this paper (together with a proof of a weaker bound). 
\begin{corollary}%
\label{cor:main_result:independence:number}
For every $\ell \geq 4$ there exists $c > 0$ such that whp the independence number in the final graph of the $C_{\ell}$-free process is at least  $c (n \log n)^{(\ell-2)/(\ell-1)}$. \qed 
\end{corollary}
Up to the constant this matches the upper bound established by Bohman and Keevash~\cite{BohmanKeevash2010H}. We infer that whp the independence number in the final graph of the $C_{\ell}$-free process is $\Theta(n \log n)^{(\ell-2)/(\ell-1)})$.

\subsection{Comparison with previous work} 
The basic idea of the proof is similar to~\cite{OsthusTaraz2001}: we show that, after a certain number of steps, every pair $(\tilde{v},U)$ with $\tilde{v} \notin U$ and $|U|=D (n \log n)^{1/(\ell-1)}$ has some property that prevents $U \subseteq \Gamma(\tilde{v})$ in the final graph of the $C_{\ell}$-free process. 
Osthus and Taraz~\cite{OsthusTaraz2001} establish their $O(n^{1/(\ell-1)} \log n)$ bound for the maximum degree using a `static' point of view: they couple the $C_{\ell}$-free process (or more generally the $H$-free process) with the classical random graph process and then show that even after deleting all edges contained in a copy of $C_{\ell}$, every $(\tilde{v},U)$ has the desired property. 
By contrast, we obtain the better $O((n \log n)^{1/(\ell-1)})$ bound by tracking the step-by-step effects of each edge added in the $C_{\ell}$-free process, and our main tool is the differential equation method used in~\cite{Warnke2010K4}.

Our argument relates to the proof of Bohman for the $C_{3}$-free process as follows. 
In~\cite{Bohman2009K3} it is shown that every large set of vertices contains at least one edge, which implies a bound on the maximum degree, since the neighbourhood of each vertex is an independent set. 
In other words, the upper bound follows from a bound on the independence number. 
For the $C_{\ell}$-free process, $\ell \geq 4$, the maximum degree is a separate question. 
In particular, we need to consider a more involved event, and thus must study the combinatorial structure of large sets more precisely.

To this end we track several random variables for every $(\tilde{v},U)$. 
But, when applying the differential equation method, there are significant technical difficulties, and a simple refinement of the approach used in~\cite{Warnke2010K4} for the $K_4$-free process does not suffice to overcome them. 
Here one crucial ingredient is a new connection between the $H$-free process and the  Erd{\H{o}}s--R{\'e}nyi random graph, which might be of independent interest. 
More precisely, we develop a `transfer theorem', which enables us to prove certain results for the $H$-free process using the \emph{much} simpler binomial random graph model. 
This is a key tool for establishing properties of the $C_{\ell}$-free process which otherwise seem difficult to derive. 
We believe that it will also aid in proving new upper bounds for the $H$-free process.

\subsection{Organization of the paper}
We start by collecting the relevant properties of the $C_{\ell}$-free process in Section~\ref{sec:CL-free}. 
In Section~\ref{sec:tools} we then introduce several probabilistic tools and the differential equation method. 
Section~\ref{sec:bounding_max_degree} is devoted to the proof of Theorem~\ref{thm:main_result}. 
Our argument relies on two key statements, whose proofs are deferred to Sections~\ref{sec:trajectory_verification} and~\ref{sec:good_configurations_exist}.  
We apply the differential equation method in Section~\ref{sec:trajectory_verification}, and introduce the `transfer theorem' in Section~\ref{sec:transfer}. 
Next, in Section~\ref{sec:binomial_results} we collect properties of the binomial random graph, which are then used to complete the proof in Section~\ref{sec:good_configurations_exist}.

\section{The $C_{\ell}$-free process: preliminaries and notation}
\label{sec:CL-free}
In this section we introduce some notation and briefly review properties of the $C_{\ell}$-free process needed in our argument. 
We closely follow~\cite{BohmanKeevash2010H} and the reader familiar with the results of Bohman and Keevash may wish to skip this section.

\subsection{Terminology and notation}\label{sec:notation} 
Let $G(i)$ denote the graph with vertex set $[n]=\{1,\ldots, n\}$ after $i$ steps of the $C_{\ell}$-free process. 
Its edge set $E(i)$ contains $i$ edges; we partition the remaining non-edges $\binom{[n]}{2} \setminus E(i)$ into two sets, $O(i)$ and $C(i)$, which we call \emph{open} and \emph{closed} pairs, respectively. 
We say that a pair $uv$ of vertices is \emph{open} in $G(i)$ if $G(i) \cup \{uv\}$ contains no copy of $C_{\ell}$. 
So, the $C_{\ell}$-free process always chooses the next edge $e_{i+1}$ uniformly at random from $O(i)$. 
In addition, for $uv \in O(i) \cup C(i)$ we write $C_{uv}(i)$ for the set of pairs $xy \in O(i)$ such that adding $uv$ and $xy$ to $G(i)$ creates a copy of $C_{\ell}$ containing both $uv$ and $xy$. 
Note that $uv \in O(i)$ would become closed, i.e., belong to $C(i+1)$, if $e_{i+1} \in C_{uv}(i)$.

With a given graph in mind, we denote the \emph{neighbourhood} of a vertex $v$ by $\Gamma(v)$, where, as usual, $\Gamma(v)$ does not include $v$. For $S \subseteq [n]$ we define $\Gamma(S) = \bigcup_{v \in S} \Gamma(v)$. 
Furthermore, for $A,B \subseteq [n]$, let $e(A,B)$ denote the number of edges that have one endpoint in $A$ and the other in $B$, where an edge with both ends in $A \cap B$ is counted once.  
If the graph under consideration is $G(i)$ we simply write $\Gamma_i(\cdot)$, but usually we omit the subscript if the corresponding $i$ is clear from the context. 
Given a set $S$ and an integer $k \geq 0$, we write $\binom{S}{k}$ for the set of all $k$-element subsets of $S$.

We use the symbol $\pm$ in two different ways, following~\cite{Bohman2009K3,BohmanKeevash2010H}. 
First, we denote by $a \pm b$ the interval $\{ a + x b : -1 \leq x \leq 1\}$. Multiple occurrences are treated independently; for example,  $\sum_{i \in [j]}(a_i \pm b_i)$ and $\prod_{i \in [j]}(a_i \pm b_i)$ mean $\{ \sum_{i \in [j]}(a_i + x_i b_i) : -1 \leq x_1, \ldots, x_{j} \leq 1\}$ and $\{ \prod_{i \in [j]}(a_i + x_i b_i) : -1 \leq x_1, \ldots, x_{j} \leq 1\}$, respectively. 
For brevity we also use the convention that $x=a \pm b$ means $x \in a \pm b$.  
Second, when considering pairs of random variables and functions, e.g.\ $Y^{+}$, $Y^{-}$ and $y^{+}$, $y^{-}$, we use the superscript $\pm$ to denote two different statements: one with~$\pm$ replaced by~$+$, and the other with~$\pm$ replaced by~$-$. For example,  $Y^{\pm}(i)=y^{\pm}(t)$ means $Y^{+}(i)=y^{+}(t)$ and $Y^{-}(i)=y^{-}(t)$. 
Finally, combinations of both ways are treated independently; for example, $Y^{\pm}(i)=y^{\pm}(t) \pm b$ means $Y^{+}(i)=y^{+}(t) \pm b$ and $Y^{-}(i)=y^{-}(t)\pm b$.

\subsection{Parameters, functions and constants} 
In the remainder of this paper we fix $\ell \geq 4$. 
Following~\cite{BohmanKeevash2010H}, we introduce constants $\epsilon$, $\mu$ and $W$. We choose $W$ sufficiently large and afterwards $\epsilon$ and $\mu$ small enough such that, in addition to the constraints implicit in~\cite{BohmanKeevash2010H} for $H=C_{\ell}$, we have
\begin{equation}
\label{eq:Cl-constants:Wepsmu}
W \geq \ell^{2} 2^{\ell+1} \geq 50
   , \qquad 
\epsilon \leq 1/\big(2^{15}\ell^3\big)
  \qquad \text{ and } \qquad 
 2W\mu^{\ell-1} \leq \epsilon  .
\end{equation}
Since the additional constraints in~\cite{BohmanKeevash2010H} only depend on $H=C_{\ell}$, we deduce that $\mu$ is an absolute constant (depending only on $\ell$). 
Next, similar as in~\cite{BohmanKeevash2010H} we set 
\begin{equation}
\label{eq:Cl-parameters}
p = n^{-1+1/(\ell-1)}  ,  \quad  t_{\max} = \mu (\log n)^{1/(\ell-1)}  \quad  \text{and}  \quad  m =  n^2 p t_{\max} = \mu n^{\ell/(\ell-1)} (\log n)^{1/(\ell-1)} .  
\end{equation}
Formally, $m$ (a number of steps) should be defined as $\lfloor n^2 p t_{\max} \rfloor$, say, but, as usual, we will henceforth ignore the irrelevant rounding to integers. 
For every step $i$ we define $t = t(i) = i/(n^2p)$, where, for the sake of brevity, we simply write $t$ if the corresponding $i$ is clear from the context.  
Next we introduce the functions 
\begin{equation}
\label{eq:Cl-functions}
q(t) = e^{-(2t)^{\ell-1}} \qquad \text{ and } \qquad  f(t) = e^{(t^{\ell-1} + t)W}   .
\end{equation} 
Now, using \eqref{eq:Cl-constants:Wepsmu}, for every $0 \leq t \leq t_{\max}$, for $n$ large enough we readily obtain 
\begin{equation}
\label{eq:Cl-functions-estimates}
1 \geq q(t) \geq n^{-\epsilon/4} \qquad \text{ and } \qquad	1 \leq f(t) q(t)^{\ell} \leq f(t) \leq n^{\epsilon}   .
\end{equation}

\subsection{Previous results for the $C_{\ell}$-free process} 
The results of Bohman and Keevash~\cite{BohmanKeevash2010H} imply that a wide range of random variables are dynamically concentrated throughout the \mbox{first $m$ steps} of the $C_{\ell}$-free process. 
For our argument the key properties are estimates on the number of open pairs as well as bounds for the degree and certain closed pairs. 
So, for the reader's convenience we state their results here in a simplified form. 
\begin{theorem}%
\label{thm:BohmanKeevash2010H}%
{\normalfont\cite{BohmanKeevash2010H}} 
Set $s_e = n^{1/(2\ell)-\epsilon}$. 
Let $\cT_j$ denote the event that for every $0 \leq i \leq j$, we have $|O(i)| > 0$ as well as 
\begin{align}
\label{eq:open-estimate}
|O(i)| &= \left(1 \pm 3f(t)/s_e\right) q(t)n^2/2 && \text{and}\\
\label{eq:degree-estimate}
|\Gamma_i(v)| & \leq 3 np t_{\max} && \text{for all vertices $v \in [n]$.}  
\end{align}
Let $\cJ_j$ denote the event that for every $0 \leq i \leq j$ we have 
\begin{align}
\label{eq:closed-estimate}
 &|C_{uv}(i)| = \left((\ell-1) (2t)^{\ell-2} q(t) \pm 7\ell f(t)/s_e\right) p^{-1} && \text{for all $uv \in O(i) \cup C(i)$ and }  \\
\label{eq:closed-intersection-estimate}
 & |C_{u'v'}(i)  \cap C_{u''v''}(i)| \leq  n^{-1/\ell} p^{-1} &&  \text{for all distinct $u'v',u''v'' \in O(i)$.} 
\end{align}
Then $\cJ_m \cap \cT_m$ holds whp in the $C_{\ell}$-free process. \qed 
\end{theorem}
After some simple estimates, both \eqref{eq:open-estimate} and \eqref{eq:degree-estimate} follow directly from Theorem~$1.4$ in~\cite{BohmanKeevash2010H}. \marginal{CHECK!}
Now, using $\mathrm{aut}(C_{\ell}) = 2\ell$ and $(2t)^{\ell-2} q(t) \leq 1$, which follow from elementary considerations, Corollary~$6.2$ and Lemma~$8.4$ in~\cite{BohmanKeevash2010H} imply \eqref{eq:closed-estimate} and \eqref{eq:closed-intersection-estimate}. 
(Because the `high probability events' of~\cite{BohmanKeevash2010H} in fact hold with probability at least $1-n^{-\omega(1)}$, we may take the union bound over all steps and pairs.) 
We remark that there is a factor of $2$ difference in \eqref{eq:closed-estimate} since we use unordered instead of ordered pairs.

In our argument we use two additional properties of the $C_{\ell}$-free process. 
The next lemma follows from Lemmas~$4.2$ and~$4.3$ in~\cite{Warnke2010K4}, which in turn are based on Lemmas $4.1$--$4.3$ in~\cite{BohmanKeevash2010H}. 
\begin{lemma}%
\label{lem:edges_bounded_and_large_degree_bounded}%
\normalfont{\cite{Warnke2010K4}} 
Let $\cK_i$ denote the event that for all $a,b \geq 1$ and every $A, B \subseteq [n]$ with $|A|=a$ and $|B|=b$, in $G(i)$ we have $e(A,B) < \max\{4\epsilon^{-1}(a+b),pabn^{2\epsilon}\}$. 
Let $\cL_i$ denote the event that for all $a \geq 1$ and $d \geq \max\{16\epsilon^{-1}, 2apn^{2\epsilon}\}$, for every $A \subseteq [n]$ with $|A| = a$ we have $|D_{A,d}(i)| < 16\epsilon^{-1}d^{-1}a$, where $D_{A,d}(i) \subseteq [n]$ contains all vertices $v \in [n]$ with $|\Gamma(v) \cap A| \geq d$ in $G(i)$. 
Then the probability that $\cT_m$ holds and $\cK_m \cap \cL_m$ fails is $o(1)$. \qed 
\end{lemma}

\section{Probabilistic tools}
\label{sec:tools}
In this section we introduce several probabilistic tools that we will use in our argument.

\subsection{Concentration inequalities} 
The following Chernoff bounds, see e.g.\ Section~$2.1$ of~\cite{JLR2000RandomGraphs}, provide estimates for the probability that a sum of independent indicator variables deviates substantially from its expected value. 
\begin{lemma}[`Chernoff bounds']%
\label{lem:chernoff}%
Let $X = \sum_{i\in[n]} X_i$, where the $X_i$'s are independent Bernoulli-distributed random variables.  Set  $\mu = \EE[X]$. 
Then for all $t \geq 0$ we have
\begin{equation}
\label{eq:chernoff:lower}
\PP[X \leq \mu - t] \leq e^{-t^2/(2\mu)}   . 
\end{equation}
Furthermore, for all $t \geq 7 \mu$ we have
\begin{equation}
\label{eq:chernoff:upper:simple}
\PP[X \geq t] \leq e^{-t}   . 
\end{equation} 
\end{lemma}
In our argument we need to estimate the probability that in $G_{n,p}$ some subset contains `too many' copies of a certain graph. 
R{\"o}dl and Ruci\'nski~\cite{RoedlRucinski1995} showed that exponential upper-tail bounds can be obtained if we allow for deleting a few edges; this is usually referred to as the Deletion Lemma~\cite{JansonRucinski2004Deletion}. 
\begin{lemma}[`Deletion Lemma']%
\label{lem:deletion_lemma}%
Suppose $0 < p < 1$ and that $\cS$ is a family of subsets from $\binom{[n]}{2}$. 
We say that a graph $G$ \emph{contains} $\alpha \in \cS$ if all the edges of $\alpha$ are present in $G$. 
Let $\mu$ denote the expected number of elements in $\cS$ that are contained in $G_{n,p}$. 
Let $\cD\cL(b,k,\cS)$ denote the event that there exists $\cI_0 \subseteq \cS$ with $|\cI_0| \leq b$ such that, setting $E_0 = \bigcup_{\alpha \in \cI_0} \alpha$, $G(n,p) \setminus E_0$ contains at most $\mu+k$ elements from $\cS$. 
Then for every $b,k>0$ the probability that $\cD\cL(b,k,\cS)$ fails is at most 
\begin{equation*}
\label{eq:lem:deletion_lemma}
\left(1+\frac{k}{\mu}\right)^{-b} \leq \exp\left\{-\frac{b k}{\mu+k}\right\}   . 
\end{equation*} 
\end{lemma}
In~\cite{Warnke2010K4} a slightly weaker variant of the above lemma was proven for the $H$-free process, where $H$ is strictly $2$-balanced. 
The results of Section~\ref{sec:transfer} will shed some light on this intriguing phenomenon.

\subsection{Differential equation method}
\label{sec:dem}
A crucial ingredient of our analysis is the differential equation method, which was developed by Wormald~\cite{Wormald1995DEM,Wormald1999DEM} to show that in certain discrete stochastic processes a collection $\cV$ of random variables is whp approximated by the solution of a suitably defined system of differential equations. 
Developing ideas of Bohman and Keevash~\cite{BohmanKeevash2010H}, the following variant was introduced in~\cite{Warnke2010K4}. 
It will be an important tool for showing that certain random variables are dynamically concentrated throughout the evolution of the $C_{\ell}$-free process. 
\begin{lemma}[`Differential Equation Method' \normalfont{\cite[Lemma 5.3]{Warnke2010K4}}]%
\label{lem:dem}%
Suppose that $m=m(n)$ and $s=s(n)$ are positive parameters. 
Let $\cC=\cC(n)$ and $\cV=\cV(n)$ be sets. 
For every $0 \leq i \leq m$ set $t=t(i)=i/s$. 
Suppose we have a filtration $\cF_0 \subseteq \cF_1 \subseteq \cdots$ and random variables $X_{\sigma}(i)$ and $Y^{\pm}_{\sigma}(i)$ which satisfy the following conditions. 
Assume that for all $\sigma \in \cC \times \cV$ the random variables $X_{\sigma}(i)$ are non-negative and $\cF_i$-measurable for all $0 \leq i \leq m$, and that for all $0 \leq i < m$ the random variables $Y^{\pm}_{\sigma}(i)$ are non-negative, $\cF_{i+1}$-measurable and satisfy
\begin{equation}
\label{eq:lem:dem:rv_relation}
X_{\sigma}(i+1) - X_{\sigma}(i) =  Y^{+}_{\sigma}(i) - Y^{-}_{\sigma}(i)   .
\end{equation} 
Furthermore, suppose that for all $0 \leq i \leq m$ and $\Sigma \in \cC$ we have an event $\badE \in \cF_i$.
Then, for all $0 \leq i \leq m$ we define $\badEL = \bigcup_{0 \leq j \leq i} \badE[j]$. 
In addition, suppose that for each $\sigma \in \cC \times \cV$ we have positive parameters $u_{\sigma}=u_{\sigma}(n)$, $\lambda_{\sigma}=\lambda_{\sigma}(n)$, $\beta_{\sigma}=\beta_{\sigma}(n)$, $\tau_{\sigma}=\tau_{\sigma}(n)$, $s_{\sigma} = s_{\sigma}(n)$ and $S_{\sigma}=S_{\sigma}(n)$, as well as functions $x_{\sigma}(t)$ and $f_{\sigma}(t)$ that are smooth and non-negative for $t \geq 0$. 
For all $0 \leq i^* \leq m$ and $\Sigma \in \cC$, let $\cG_{i^*}(\Sigma)$ denote the event that for every $0 \leq i \leq i^*$ and $\sigma=(\Sigma,j)$ with $j \in \cV$ we have 
\begin{equation}
\label{eq:lem:dem:parameter_trajectory}
X_{\sigma}(i) = \left(x_{\sigma}(t) \pm \frac{f_{\sigma}(t)}{s_{\sigma}} \right) S_{\sigma}   .
\end{equation}
Next, for all $0 \leq i^* \leq m$ let $\cE_{i^*}$ denote the event that for every $0 \leq i \leq i^*$ and $\Sigma \in \cC$ the event $\badEL[i-1] \cup \cG_{i}(\Sigma)$ holds. 
Moreover, assume that we have an event $\cH_i \in \cF_i$ for all $0 \leq i \leq m$ with $\cH_{i+1} \subseteq \cH_{i}$ for all $0 \leq i < m$. 
Finally, suppose that the following conditions hold:  
\begin{enumerate}
\item(Trend hypothesis)
For all $0 \leq i < m$ and $\sigma = (\Sigma,j) \in \cC \times \cV$, whenever $\cE_i \cap \neg\badEL \cap \cH_i$ holds we have 
\begin{equation}
\label{eq:lem:dem:parameter_martingale_property}
\EE\big[Y^{\pm}_{\sigma}(i) \mid \cF_i \big] = \left( y^{\pm}_{\sigma}(t) \pm  \frac{h_{\sigma}(t)}{s_{\sigma}} \right)  \frac{S_{\sigma}}{s}   , 
\end{equation}
where $y_{\sigma}^{\pm}(t)$ and $h_{\sigma}(t)$ are smooth non-negative functions such that
\begin{equation}
\label{eq:lem:dem:derivative}
x'_{\sigma}(t) = y^+_{\sigma}(t) - y^-_{\sigma}(t) \qquad \text{ and } \qquad f_{\sigma}(t) \geq 2\int_{0}^{t} h_{\sigma}(\tau) \ d\tau + \beta_{\sigma}   .
\end{equation}

\item(Boundedness hypothesis) 
For all $0 \leq i < m$ and $\sigma = (\Sigma,j) \in \cC \times \cV$, whenever $\cE_i \cap \neg\badEL \cap \cH_i$ holds we have
\begin{equation}
\label{eq:lem:dem:parameter_max_change}
Y^{\pm}_{\sigma}(i)  \leq \frac{\beta_{\sigma}^2}{s_{\sigma}^2 \lambda_{\sigma} \tau_{\sigma}} \cdot  \frac{S_{\sigma}}{u_{\sigma}}   . 
\end{equation}

\item(Initial conditions) 
For all $\sigma \in \cC \times \cV$ we have 
\begin{equation}
\label{eq:lem:dem:initial_condition}
X_{\sigma}(0) = \left(x_{\sigma}(0) \pm \frac{\beta_{\sigma}}{3s_{\sigma}} \right) S_{\sigma}   .
\end{equation}

\item(Bounded number of configurations and variables) We have 
\begin{equation}
\label{eq:lem:dem:bounded_parameters}
\max\left\{|\cC|,|\cV|\right\} \leq \min_{\sigma \in \cC \times \cV} e^{u_{\sigma}}   .
\end{equation}
 
\item(Additional technical assumptions) 
For all $\sigma \in \cC \times \cV$ we have 
\begin{gather}
\label{eq:lem:dem:technical_assumptions:sm}
s \geq \max\{15 u_{\sigma} \tau_{\sigma} (s_{\sigma}  \lambda_{\sigma}/\beta_{\sigma})^2, 9s_{\sigma} \lambda_{\sigma}/\beta_{\sigma}\}   , \qquad s/(18 s_{\sigma} \lambda_{\sigma}/\beta_{\sigma}) < m \leq s \cdot \tau_{\sigma}/1944   ,\\
\label{eq:lem:dem:technical_assumptions:xy}
\sup_{0 \leq t \leq m/s} y^{\pm }_{\sigma}(t) \leq \lambda_{\sigma}   , \qquad 
\int_0^{m/s} |x''_{\sigma}(t)| \ dt \leq \lambda_{\sigma}   , \\ 
\label{eq:lem:dem:technical_assumptions:h}
h_{\sigma}(0) \leq s_{\sigma} \lambda_{\sigma}  \qquad \text{ and }  \qquad  \int_0^{m/s} |h'_{\sigma}(t)| \ dt \leq s_{\sigma} \lambda_{\sigma}   .
\end{gather}
\end{enumerate}
Then we have 
\begin{equation*}
\label{eq:lem:dem}
\mathbb{P}[\neg\cE_m \cap \cH_m] \leq 4 \max_{\sigma \in \cC \times \cV}e^{-u_{\sigma}}   .
\end{equation*} 
\end{lemma} 
An important feature of Lemma~\ref{lem:dem} is that the variables in $\cV$ are tracked for every \emph{configuration $\Sigma \in \cC$}. 
However, it only gives approximation guarantees for the variables that `belong' to $\Sigma$ as long as the `local' \emph{bad event} $\badEL$ fails. 
For more details we refer to Section~$5.3$ and Appendix~A.$1$ in~\cite{Warnke2010K4}. 
Here we just remark that if the above conditions $1$--$5$ are satisfied for $n$ large enough, $\cH_m$ holds whp and $u_{\sigma} = \omega(1)$ for all $\sigma \in \cC \times \cV$, then Lemma~\ref{lem:dem} implies that $\cE_m$ holds whp.

\section{Bounding the maximum degree} 
\label{sec:bounding_max_degree}
In this section we prove our main result, namely that whp the maximum degree in the final graph of the $C_{\ell}$-free process is $O((n \log n)^{1/(\ell-1)})$.  
In Sections~\ref{sec:motivation} and~\ref{sec:formal_setup} we first discuss the main proof ideas and introduce the formal setup used. 
Section~\ref{sec:finishing_the_proof} is then devoted to the proof of Theorem~\ref{thm:main_result}, which in turn relies on two involved statements that are proved in subsequent sections.

\subsection{Sketch of the proof} 
\label{sec:motivation}
The following definition plays a crucial role in our proof.  
Given $(\tilde{v},U)$, where $\tilde{v} \in [n]$ and $U \subseteq [n] \setminus \{\tilde{v}\}$, a \emph{$C_{\ell}$-extension for $(\tilde{v},U)$} is a path on $\ell-1$ vertices whose end vertices are in $U$ and whose remaining vertices are disjoint from $U \cup \{\tilde{v}\}$. 
Clearly, for every vertex $\tilde{v} \in [n]$, in the final graph of the $C_{\ell}$-free process $(\tilde{v},\Gamma(\tilde{v}))$ must not have a $C_{\ell}$-extension. 
Set 
\begin{equation}
\label{eq:Cl-parameters2}
\delta = \frac{1}{60^2 \ell! \ell^{\ell}}   , \qquad  \gamma = \max\left\{\frac{3^{\ell+1}}{\delta \mu^{\ell-1}},180\right\}  \qquad \text{and} \qquad u = \gamma n p t_{\max} = \gamma \mu (n\log n)^{1/(\ell-1)}   , 
\end{equation}
again ignoring the irrelevant rounding to integers in the definition of $u$.  
In order to bound the maximum degree by $u = D (n\log n)^{1/(\ell-1)}$, where $D=\gamma \mu$, it is enough to prove that whp every $(\tilde{v},U) \in [n] \times \binom{[n]}{u}$ with $\tilde{v} \notin U$ has at least one $C_{\ell}$-extension after the first $m$ steps. 
The same basic idea was used in~\cite{OsthusTaraz2001}, but our proof takes a different route, inspired by~\cite{Warnke2010K4}. 
After $i$ steps, we denote by $O_{\tilde{v},U}(i)$ the set of open pairs which would complete a $C_{\ell}$-extension for $(\tilde{v},U)$ if chosen as the next edge. 
It seems plausible that it in order prove Theorem~\ref{thm:main_result}, it suffices to show that, after some initial number of steps, $|O_{\tilde{v},U}(i)|$ is always not too small. 
Indeed,  this implies a reasonable probability of completing such an extension in each step, which in turn suggests that the probability of avoiding a $C_{\ell}$-extension in all of the first $m$ steps is very small.

We now illustrate our approach for establishing a good lower bound on $|O_{\tilde{v},U}(i)|$ for the case when $\ell=5$. 
For ease of exposition, we ignore $n^{\epsilon}$ factors whenever these are not crucial and also assume that the number of steps $i$ is large. 
So, in our rough calculations we will e.g.\ ignore whether an edge is open or not, since $|O(i)| = \omega(n^{2-\epsilon})$ by \eqref{eq:Cl-functions-estimates} and \eqref{eq:open-estimate}. 
Note that in this case we have  $p=n^{-3/4}$, $m \approx n^{5/4}$, $|C_{xy}(i)| \approx p^{-1}$ and $|U| \approx np = n^{1/4}$  by \eqref{eq:Cl-parameters}, \eqref{eq:closed-estimate} and \eqref{eq:Cl-parameters2}.

\subsubsection{The random variables used} 
We define $O'_{\tilde{v},U}(i)$ as the set of pairs $xy \in  O_{\tilde{v},U}(i)$ with $x \in U$ and $y \notin U \cup \{\tilde{v}\}$. 
Observe that for every  $xy \in O'_{\tilde{v},U}(i)$ there exists a path $v_0v_1v_2=y$ with $v_{0} \in U \setminus \{x\}$ and $v_1 \notin U \cup \{\tilde{v},x,y\}$, cf.\ Figure~\ref{fig:open:sketch}. 
The `last' edge completing a $C_{5}$-extension for $(\tilde{v},U)$ could be any one of the edges of the path, so we expect that $O'_{\tilde{v},U}(i)$ contains constant proportion of $O_{\tilde{v},U}(i)$.

\begin{figure}[h]
\centering
  \setlength{\unitlength}{1bp}%
  \begin{picture}(83.20, 54.69)(0,0)
  \put(0,0){\includegraphics{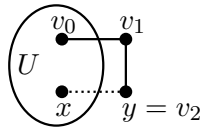}}
  \put(28.72,42.12){\rotatebox{0.00}{\fontsize{11.38}{13.66}\selectfont \smash{\makebox[0pt]{$v_0$}}}}
  \put(15.43,24.94){\fontsize{11.38}{13.66}\selectfont \makebox[0pt]{$U$}}
  \put(28.72,9.52){\rotatebox{0.00}{\fontsize{11.38}{13.66}\selectfont \smash{\makebox[0pt]{$x$}}}}
  \put(54.65,42.12){\rotatebox{0.00}{\fontsize{11.38}{13.66}\selectfont \smash{\makebox[0pt]{$v_1$}}}}
  \put(65.74,9.52){\rotatebox{0.00}{\fontsize{11.38}{13.66}\selectfont \smash{\makebox[0pt]{$y=v_2$}}}}
  \end{picture}\vspace{-0.5em}%
 \caption{\label{fig:open:sketch}A pair $xy \in O'_{\tilde{v},U}(i)$. Solid lines represent edges and dotted lines open pairs.} 
\end{figure}%

Let $Z_{\tilde{v},U}(i)$ contain all quadruples $(v_0,v_1,v_2,v_3) \in U \times [n]^2 \times U$ with $\{v_0v_1,v_1v_2\} \subseteq E(i)$, $v_2v_3 \in O(i)$ and $\{v_1,v_2\} \cap (U \cup \{\tilde{v}\}) = \emptyset$. 
Using random graphs as a guide, we expect that $G(i)$ shares many properties with the binomial random graph $G_{n,p}$, since its edge density is roughly $2tp \approx n^{-3/4} = p$. 
So, given $y$, the expected number of $v_0 \in U$ for which there exists a path $v_0v_1v_{2}=y$ should be roughly $n|U|p^2 = o(1)$. 
Hence on average $xy \in O'_{\tilde{v},U}(i)$ is contained in only one such path ending in $U$, which suggests that up to constants $|Z_{\tilde{v},U}(i)| \approx |O'_{\tilde{v},U}(i)|$. 
To sum up, our discussion indicates that a reasonable lower bound for $|Z_{\tilde{v},U}(i)|$ suffices to prove that $|O_{\tilde{v},U}(i)|$ is large. 
For this we intend to use the differential equation method and so we introduce additional variables in order to control the one-step changes of $|Z_{\tilde{v},U}(i)|$. 
To this end let $Y_{\tilde{v},U}(i)$ be the set of all $(v_0,v_1,v_2,v_3) \in  U \times [n]^2 \times U$ with $\{v_1,v_2\} \cap (U \cup \{\tilde{v}\}) = \emptyset$ that satisfy $v_0v_1 \in E(i)$, $\{v_1v_2,v_2v_3\} \subseteq O(i)$, and, similarly, let $X_{\tilde{v},U}(i)$ contain all such quadruples with $\{v_0v_1,v_1v_2,v_2v_3\} \subseteq O(i)$.

\subsubsection{Technical difficulties}\label{sec:difficulties}
One of the main problems with the approach described above is the bound on the one-step changes. 
It can happen that in one step up to $p^{-1}$ quadruples are removed from $Z_{\tilde{v},U}(i)$, which turns out to be too large for applying the differential equation method directly. 
Indeed, pick $\tilde{v},U$ such that $\{v_0\} \cup \Gamma_i(w) \subseteq U$, $|\Gamma_i(w)| \approx |U|$ and $\tilde{v}  \notin \{w\} \cup U \cup \Gamma_i(U)$; taking the random graph $G_{n,p}$ as a guide, for $e_{i+1}=wv_0$ it is easy to see that about $(np)^2|U| \approx p^{-1}$ quadruples $(v_0,v_1,v_2,v_3)$ with $v_3 \in \Gamma_i(w)$ are removed from $Z_{\tilde{v},U}(i)$. 
For the $C_{4}$-free process this can be resolved using ad-hoc arguments (e.g.\ exploiting that every $v \neq \tilde{v}$ satisfies $|\Gamma_i(v) \cap U| \leq 1$ if no $C_{4}$-extension for $(\tilde{v},U)$ exists), but for larger cycles the situation is more delicate. 
To overcome this issue, we consider a different random variable $T_{\tilde{v},U}(i)$, which is an approximation of $Z_{\tilde{v},U}(i)$ and is defined in such a way that the one-step changes are automatically not too large. Roughly speaking, this can be achieved by `ignoring' the steps where the one-step changes would be too large; similar ideas have been used e.g.\ in~\cite{Bohman2009K3,BohmanKeevash2010H,KimVu2004,Warnke2010K4}. 
Clearly, this introduces a new difficulty: we need to ensure that we do not ignore `too much', so that on the one hand the expected one-step changes are still `correct', and on the other hand $|Z_{\tilde{v},U}(i)| \approx |T_{\tilde{v},U}(i)|$ holds. 
Consequently,  we refine the tracked variables and use more sophisticated rules for ignoring tuples.

There is another significant obstacle when applying the differential equation method: adding $e_{i+1}=v_1v_2$ to $(v_0,v_1,v_2,v_3) \in Y_{\tilde{v},U}(i)$ does \emph{not} always result in an element of $Z_{\tilde{v},U}(i+1)$, since $e_{i+1}=v_1v_2$ closes $v_2v_3$ whenever $v_2v_3 \in C_{v_1v_2}(i)$ holds. 
This is an important difference to the $C_{\ell}$-free process with $\ell \leq 4$, where this does not cause any problems when bounding the maximum degree. 
For example, whenever this happens for $\ell = 4$, it is not difficult to deduce that at least one $C_{4}$-extension for $(\tilde{v},U)$ already exists. 
Returning to the case $\ell=5$, using our random graph intuition we expect that $|Y_{\tilde{v},U}(i)| \approx |U|^2n^2p \approx n^{7/4}$. 
Similar calculations suggest that the expected number of quadruples in $Y_{\tilde{v},U}(i)$ with $v_2v_3 \in C_{v_1v_2}(i)$ should be negligible compared to $|Y_{\tilde{v},U}(i)|$. 
However, if we pick $U$ such that $\Gamma_i(w) \subseteq U$ and $|\Gamma_i(w)| \approx |U|$, for $\tilde{v} \notin \{w\} \cup U \cup \Gamma_i(U)$, it certainly can happen that there are $|U|^2 \cdot np \cdot n  \approx |Y_{\tilde{v},U}(i)|$ quadruples in $Y_{\tilde{v},U}(i)$ with $v_2v_3 \in C_{v_1v_2}(i)$. 
In other words, it is simply \emph{not} true that for all $(\tilde{v},U)$ the effect of these `bad' quadruples is negligible. This is a new difficulty in comparison to the variables tracked in the analysis of the $H$-free process~\cite{BohmanKeevash2010H}. 
To deal with this issue, we substantially refine the tracked random variables, developing ideas used in~\cite{Warnke2010K4}. 
Intuitively, we show that for every $(\tilde{v},U)$ there exists a slightly altered set of random variables where the above extreme example (and other difficulties) can be avoided. 
Here the new `transfer theorem' (Theorem~\ref{thm:transfer:binomial}) is an important ingredient, which allows us to use the \emph{much} more tractable binomial random graph model for certain calculations (see Section~\ref{sec:binomial_results}).

\subsection{Formal setup}
\label{sec:formal_setup}
We now introduce the formal setup used in our argument. 
In the following it is useful to keep in mind that we intend to apply the differential equation method (Lemma~\ref{lem:dem}).

\subsubsection{Preliminaries: neighbourhoods and partitions} 
\label{sec:formal_setup:preliminaries}
Recall that by \eqref{eq:Cl-parameters2} we have $u = \gamma n p t_{\max} = \gamma \mu (n\log n)^{1/(\ell-1)}$. We set
\begin{equation}
\label{eq:Cl-parameters3}
k=u/60 = \gamma/60 \cdot n p t_{\max} = \gamma \mu /60 \cdot  (n\log n)^{1/(\ell-1)}  \qquad \text{ and } \qquad r=\lfloor n/(\ell-3)\rfloor   .
\end{equation} 
Given $X \subseteq [n]$, we partition $\{1, \ldots, (\ell-3)r \} \setminus X$ as follows: for every $1 \leq j \leq \ell-3$ we set 
\begin{equation}
\label{eq:def:Vj}
V_{j}=V_{j}(X) = \{ v \in [n] \setminus X \;:\; (j-1)r < v  \leq j r\} .
\end{equation} 
With a given graph in mind, which will later be $G(i)$ or the binomial random graph, for every $S \subseteq [n]$ we define its \emph{neighbourhoods wrt.\ $X$} as  
\begin{equation*}
\label{eq:def:Nj}
N^{(0)}(S,X) = S \qquad \text{ and } \qquad N^{(j+1)}(S,X) = \Gamma\big(N^{(j)}(S,X)\big) \cap V_{j+1}(X)   ,
\end{equation*}
see also Figure~\ref{fig:neighbourhoods:Lambda}. 
\begin{figure}[t]
\centering
  \setlength{\unitlength}{1bp}%
  \begin{picture}(236.02, 100.79)(0,0)
  \put(0,0){\includegraphics{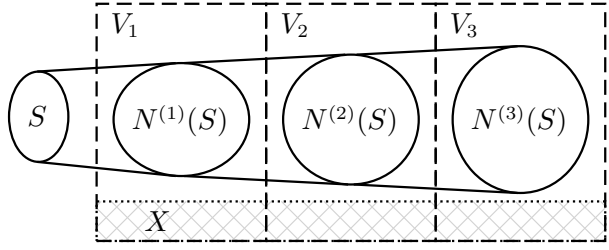}}
  \put(44.12,84.34){\rotatebox{0.00}{\fontsize{10.69}{12.83}\selectfont \smash{\makebox[0pt][l]{$V_1$}}}}
  \put(16.24,48.97){\rotatebox{0.00}{\fontsize{10.69}{12.83}\selectfont \smash{\makebox[0pt]{$S$}}}}
  \put(70.26,47.90){\rotatebox{0.00}{\fontsize{10.69}{12.83}\selectfont \smash{\makebox[0pt]{$N^{(1)}(S)$}}}}
  \put(61.74,9.57){\rotatebox{0.00}{\fontsize{10.69}{12.83}\selectfont \smash{\makebox[0pt]{$X$}}}}
  \put(108.01,84.34){\rotatebox{0.00}{\fontsize{10.69}{12.83}\selectfont \smash{\makebox[0pt][l]{$V_2$}}}}
  \put(134.16,47.90){\rotatebox{0.00}{\fontsize{10.69}{12.83}\selectfont \smash{\makebox[0pt]{$N^{(2)}(S)$}}}}
  \put(171.90,84.34){\rotatebox{0.00}{\fontsize{10.69}{12.83}\selectfont \smash{\makebox[0pt][l]{$V_3$}}}}
  \put(198.05,47.90){\rotatebox{0.00}{\fontsize{10.69}{12.83}\selectfont \smash{\makebox[0pt]{$N^{(3)}(S)$}}}}
  \end{picture}%
  \caption{\label{fig:neighbourhoods:Lambda}The neighbourhoods $N^{(j)}(S) = N^{(j)}(S,X)$ for $j \in [3]$, where $S$ may also intersect with $X$ and the vertex classes, i.e., with  $X \cup V_1 \cup V_2 \cup V_3$. 
Furthermore, $S \cap  N^{(j)}(S) \neq \emptyset$ is also possible.} 
\end{figure}%
Observe that all $N^{(j)}(S,X)$ are disjoint if $S \subseteq X$.  
Furthermore, $X \subseteq Y$ implies 
\begin{equation}
\label{eq:neighbourhood:monotone}
V_{j}(Y) \subseteq V_{j}(X) \qquad \text{ and } \qquad N^{(j)}(S,Y) \subseteq N^{(j)}(S,X)   .
\end{equation}
Finally, for the sake of brevity we define $N^{(\leq j)}(S,X) = \bigcup_{0 \leq j' \leq j} N^{(j')}(S,X)$.

\subsubsection{Configurations} 
\label{sec:main-proof:configs}
We define the set $\cC$ of configurations to be the set of all $\Sigma=(\tilde{v},U,A,B,R)$ with $\tilde{v} \in [n]$, $U \in \binom{[n] \setminus \{\tilde{v}\}}{u}$, disjoint $A,B \in \binom{U}{k}$, and $R \subseteq [n]$ with $\{\tilde{v}\} \cup U \subseteq R$ and $|R| \leq kn^{10\ell\epsilon}$. 
Given $\Sigma \in \cC$, we then set $T_{\Sigma}=A \times V_1 \times \cdots \times V_{\ell-3} \times B$, where each  $V_{j}=V_{j}(R)$ is given by \eqref{eq:def:Vj}. 

Given $\Sigma \in \cC$, distinct $x,y \in [n]$ and $j \in [\ell-1]$,  
let $C_{x,y,\Sigma}(i,j)$ contain all pairs $bw \in B \times N^{(\ell-3)}(A,R)$ for which there exist disjoint paths $b=w_1 \cdots w_{j}=x$ and $y=w_{j+1} \cdots w_{\ell}=w$ in $G(i)$. 
Note that adding $xy$ and $bw$ completes a copy of $C_{\ell}$ containing both $xy$ and $bw$. 
Furthermore, observe that $C_{x,y,\Sigma}(i,j)$ and $C_{y,x,\Sigma}(i,j)$ may differ. 
So, for all $xy \in O(i) \cup C(i)$ we see that the intersection of $C_{xy}(i)$ with $B \times N^{(\ell-3)}(A,R)$ is contained in $\bigcup_{j \in [\ell-1]} \big[ C_{x,y,\Sigma}(i,j) \cup C_{y,x,\Sigma}(i,j) \big]$. 
Finally, note that by monotonicity we have $C_{x,y,\Sigma}(i,j) \subseteq C_{x,y,\Sigma}(i+1,j)$.

\subsubsection{Random variables}
\label{sec:main-proof:variables} 
For every $\Sigma \in \cC$ we track the sizes of several sets throughout the evolution of the $C_{\ell}$-free process. 
For brevity, given $(v_0, \ldots, v_{\ell-2}) \in T_{\Sigma}$, we set $f_{j} = v_{j-1}v_{j}$ for all $1 \leq j \leq \ell-2$. 
For every $0 \leq j \leq \ell-3$ we introduce sets $T_{\Sigma,j}(i)$, which for $0 \leq j < \ell-3$ will satisfy 
\begin{equation}
\label{eq:X:j:inclusion}
T_{\Sigma,j}(i) \subseteq \big\{ (v_0, \ldots, v_{\ell-2}) \in T_{\Sigma} \;:\; \{f_1, \ldots, f_j\} \subseteq E(i) \;\wedge\; \{f_{j+1}, \ldots, f_{\ell-2}\} \subseteq O(i) \big\}   , 
\end{equation}
and for the special case $j = \ell-3$ we will have  
\begin{equation}
\label{eq:X:l3:inclusion}
T_{\Sigma,\ell-3}(i) \subseteq \big\{ (v_0, \ldots, v_{\ell-2}) \in T_{\Sigma} \;:\; \{f_1, \ldots, f_{\ell-3}\} \subseteq E(i) \;\wedge\; f_{\ell-2} \in O(i) \cup C(i) \big\}   ,
\end{equation}
see also Figure~\ref{fig:tuples:T}. Note that $f_{\ell-2}$ can be in $O(i)$ or $C(i)$ for $T_{\Sigma,\ell-3}(i)$, but we will see later that the number of tuples with pairs in $C(i)$ is negligible. 
In the following we define the $T_{\Sigma,j}(i)$ inductively, starting with $T_{\Sigma,j}(0) = \emptyset$ for $j > 0$ and $T_{\Sigma,0}(0) = T_{\Sigma}$. 
Now suppose the process chooses $e_{i+1}=xy \in O(i)$ as the next edge in step $i+1$. 
For $j > 0$ a tuple $(v_0, \ldots, v_{\ell-2}) \in T_{\Sigma,j-1}(i)$ is \emph{added} to $T_{\Sigma,j}(i+1)$, i.e., is in $T_{\Sigma,j}(i+1)$, if $f_j = e_{i+1}$, $\{f_{j+1}, \ldots, f_{\ell-2}\} \cap C_{f_{j}}(i) = \emptyset$, and in $G(i)$ there is no path $w_0 \cdots w_j=v_j$ with $w_0 \in A$.  
Furthermore, for $j < \ell-3$ a tuple $(v_0, \ldots, v_{\ell-2}) \in T_{\Sigma,j}(i)$ is \emph{removed}, i.e., not in $T_{\Sigma,j}(i+1)$, if $e_{i+1} \in \{f_{j+1}, \ldots, f_{\ell-2}\}$ or 
$e_{i+1} \in C_{f_{j+1}}(i) \cup \cdots \cup C_{f_{\ell-2}}(i)$. 
\begin{figure}[t]
	\centering
  \setlength{\unitlength}{1bp}%
  \begin{picture}(114.66, 92.51)(0,0)
  \put(0,0){\includegraphics{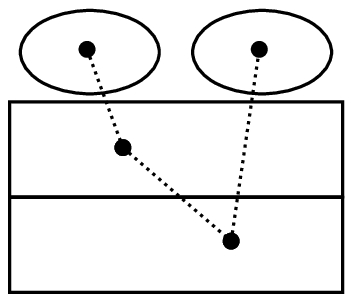}}
  \put(42.35,71.87){\fontsize{11.38}{13.66}\selectfont \makebox[0pt]{$A$}}
  \put(93.56,50.52){\fontsize{11.38}{13.66}\selectfont \makebox[0pt]{$V_1$}}
  \put(20.14,74.04){\fontsize{11.38}{13.66}\selectfont \makebox[0pt]{$v_0$}}
  \put(69.74,74.04){\fontsize{11.38}{13.66}\selectfont \makebox[0pt]{$v_3$}}
  \put(31.44,42.90){\fontsize{11.38}{13.66}\selectfont \makebox[0pt]{$v_1$}}
  \put(62.05,15.41){\fontsize{11.38}{13.66}\selectfont \makebox[0pt]{$v_2$}}
  \put(91.39,71.87){\fontsize{11.38}{13.66}\selectfont \makebox[0pt]{$B$}}
  \put(93.56,22.74){\fontsize{11.38}{13.66}\selectfont \makebox[0pt]{$V_2$}}
  \end{picture}%
\hspace{2.5em}
  \setlength{\unitlength}{1bp}%
  \begin{picture}(114.66, 92.51)(0,0)
  \put(0,0){\includegraphics{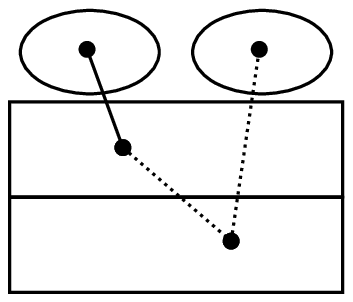}}
  \put(42.35,71.87){\fontsize{11.38}{13.66}\selectfont \makebox[0pt]{$A$}}
  \put(93.56,50.52){\fontsize{11.38}{13.66}\selectfont \makebox[0pt]{$V_1$}}
  \put(20.14,74.04){\fontsize{11.38}{13.66}\selectfont \makebox[0pt]{$v_0$}}
  \put(69.74,74.04){\fontsize{11.38}{13.66}\selectfont \makebox[0pt]{$v_3$}}
  \put(31.44,42.90){\fontsize{11.38}{13.66}\selectfont \makebox[0pt]{$v_1$}}
  \put(62.05,15.41){\fontsize{11.38}{13.66}\selectfont \makebox[0pt]{$v_2$}}
  \put(91.39,71.87){\fontsize{11.38}{13.66}\selectfont \makebox[0pt]{$B$}}
  \put(93.56,22.74){\fontsize{11.38}{13.66}\selectfont \makebox[0pt]{$V_2$}}
  \end{picture}%
\hspace{2.5em}
  \setlength{\unitlength}{1bp}%
  \begin{picture}(114.66, 92.51)(0,0)
  \put(0,0){\includegraphics{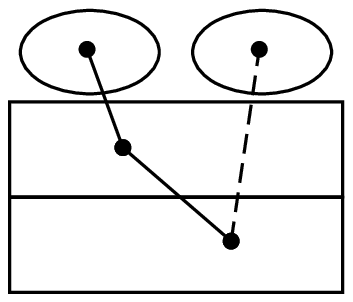}}
  \put(42.35,71.87){\fontsize{11.38}{13.66}\selectfont \makebox[0pt]{$A$}}
  \put(93.56,50.52){\fontsize{11.38}{13.66}\selectfont \makebox[0pt]{$V_1$}}
  \put(20.14,74.04){\fontsize{11.38}{13.66}\selectfont \makebox[0pt]{$v_0$}}
  \put(69.74,74.04){\fontsize{11.38}{13.66}\selectfont \makebox[0pt]{$v_3$}}
  \put(31.44,42.90){\fontsize{11.38}{13.66}\selectfont \makebox[0pt]{$v_1$}}
  \put(62.05,15.41){\fontsize{11.38}{13.66}\selectfont \makebox[0pt]{$v_2$}}
  \put(91.39,71.87){\fontsize{11.38}{13.66}\selectfont \makebox[0pt]{$B$}}
  \put(93.56,22.74){\fontsize{11.38}{13.66}\selectfont \makebox[0pt]{$V_2$}}
  \end{picture}%
	\caption{\label{fig:tuples:T}Tuples $(v_0, v_1,v_2, v_{3})$ in $T_{\Sigma,0}(i)$, $T_{\Sigma,1}(i)$ and $T_{\Sigma,2}(i)$ for $\ell=5$, where $\Sigma=(\tilde{v},U,A,B,R)$. Solid lines represent edges, dotted lines open pairs and dashed lines pairs that are open or closed. For the other pairs there is no restriction, i.e., they may be open, closed or an edge.}
\end{figure}
For the special case $j=\ell-3$, a tuple $(v_0, \ldots, v_{\ell-2}) \in T_{\Sigma,\ell-3}(i)$ is \emph{removed}, i.e., not in $T_{\Sigma,\ell-3}(i+1)$, or \emph{ignored}, i.e., remains in $T_{\Sigma,\ell-3}(i+1)$, according to the following rules:   
\begin{center}
\begin{minipage}[b]{0.95\linewidth}
\begin{description}
\item[Case 1.] If $f_{\ell-2} = e_{i+1}$, then the tuple $(v_0, \ldots, v_{\ell-2})$  is removed,\vspace{-0.2em}
\item[Case 2.] If $e_{i+1} \in C_{f_{\ell-2}}(i)$, then the tuple $(v_0, \ldots, v_{\ell-2})$ is\vspace{-0.5em}
	\begin{enumerate}
		\item[(R2)] removed if there exists $j \in [\ell-1]$ and $x,y \in [n]$ such that $e_{i+1}=xy$, $f_{\ell-2} \in C_{x,y,\Sigma}(i,j)$ and $|C_{x,y,\Sigma}(i,j)| \leq p^{-1} n^{-30\ell\epsilon}$, and\vspace{-0.2em} 
		\item[(I2)] ignored otherwise.\vspace{-0.2em}
	\end{enumerate} 
\end{description}
\end{minipage}
\end{center} 
The above definition clearly satisfies \eqref{eq:X:j:inclusion} and \eqref{eq:X:l3:inclusion}. 
Intuitively, the rules for removing tuples from $T_{\Sigma,\ell-3}(i)$ ensure that the one-step changes are `by definition' not too large. 
Furthermore, the way in which the tuples are added yields the following \emph{extension property $\cU_{T}$}. 
\begin{lemma}
\label{lemma:extension:property}
Given $i \geq 0$, let $\cU_T(i)$ denote the property that for all $\Sigma \in \cC$ and $1 \leq j \leq \ell-3$, for every $(v_{j}, \ldots, v_{\ell-2}) \in V_j \times \cdots V_{\ell-3} \times B$ there exists at most one $(v_0, \ldots, v_{j-1}) \in A \times V_1 \times \cdots \times V_{j-1}$ such that $(v_0, \ldots, v_{\ell-2}) \in \bigcup_{i' \leq i} T_{\Sigma,j}(i')$. 
Then $\cU_T = \cU_T(i)$ holds for every $i \geq 0$. \qed
\end{lemma}
The proof proceeds by induction on $i$ and $j$; we leave the straightforward details to the reader (it is helpful to observe that after $(v_0, \ldots, v_{\ell-2}) \in T_{\Sigma,j-1}(i)$ is added to $T_{\Sigma,j}(i+1)$, no further tuples containing $v_{j}$ can be added due to the $v_0 \cdots v_{j}$ path). 
Note that by $\cU_T$ every $(v_{j}, \ldots, v_{\ell-2}) \in V_j \times \cdots V_{\ell-3} \times B$ is contained in at most one tuple in $\bigcup_{i' \leq i} T_{\Sigma,j}(i')$. 
This is an important ingredient of our argument, and we remark that a simpler variant of this property has previously been used in~\cite{Warnke2010K4}.

Recall that our goal is to show that there are many open pairs whose addition would complete a $C_{\ell}$-extension for $(\tilde{v},U)$. 
Given $\Sigma=(\tilde{v},U,A,B,R)$, note that for every $(v_0, \ldots, v_{\ell-2}) \in T_{\Sigma,\ell-3}(i)$, if $f_{\ell-2} \in O(i)$, then adding $f_{\ell-2}$ to $G(i)$ would complete such a $C_{\ell}$-extension. 
Now, since $\cU_T$ implies that every pair $f_{\ell-2}=xy$ with $x \in V_{\ell-3}$ and $y \in B$ is contained in at most one such tuple in $T_{\Sigma,\ell-3}(i)$, our aim is to obtain a lower bound on the size of 
\begin{equation}
\label{eq:def:partial_open}
Z_{\Sigma,\ell-3}(i) = \big\{ (v_0, \ldots, v_{\ell-2}) \in T_{\Sigma,\ell-3}(i) \;:\; f_{\ell-2} \in O(i) \big\}   .
\end{equation}

\subsubsection{Bad events } 
\label{sec:main-proof:bad_high_probability_events}
The following bad event $\cB_{i}(\Sigma)$ is crucial for our argument: it addresses the two main technical difficulties outlined in Section~\ref{sec:difficulties}. 
For all  $0 \leq i \leq m$ and $\Sigma \in \cC$ we define $\cB_{i}(\Sigma) = \cB_{1,i}(\Sigma) \cup \cB_{2,i}(\Sigma)$, where 
\begin{flushright}
\begin{minipage}[c]{0.95\linewidth}
\begin{center}
\begin{minipage}[b]{0.95\linewidth}
\begin{enumerate}
\item[$\cB_{1,i}(\Sigma)$ =] in $G(i)$ there are more than $k^2(np)^{\ell-4}n^{-9\epsilon}$ pairs $(b,w) \in B \times N^{(\ell-4)}(A,R)$ for which there exists a path $b=w_0 \cdots w_{\ell-2}=w$, and  \vspace{-0.4em} 
\item[$\cB_{2,i}(\Sigma)$ =] in $G(i)$ we have $|L_{\Sigma}(i)| \geq p^{-1}n^{-1/(2\ell)}$, where $L_{\Sigma}(i)$ contains all $xy \in \binom{[n]}{2}$ with $\max_{j \in [\ell-1]}\{|C_{x,y,\Sigma}(i,j)|,|C_{y,x,\Sigma}(i,j)|\} \geq p^{-1}n^{-30\ell\epsilon}$. 
\end{enumerate} 
\end{minipage}
\end{center}
\end{minipage}
\end{flushright}
Clearly, $\cB_{i}(\Sigma)$ depends only on the first $i$ steps and is increasing, i.e., $\cB_{i}(\Sigma) \subseteq \cB_{i+1}(\Sigma)$ holds. 

We now briefly give some intuition for $\cB_{1,i}(\Sigma)$ and $\cB_{2,i}(\Sigma)$, which are important ingredients for estimating the number of tuples added to $T_{\Sigma,\ell-3}(i+1)$ and removed from $T_{\Sigma,\ell-3}(i)$. 
First, recall that $(v_0, \ldots, v_{\ell-2}) \in T_{\Sigma,\ell-4}(i)$ can not be added to $T_{\Sigma,\ell-3}(i+1)$ if $f_{\ell-2} \in  C_{f_{\ell-3}}(i)$. 
For such `useless' tuples there exists a path $v_{\ell-2}=w_0 \cdots w_{\ell-2}=v_{\ell-4}$ with $(v_{\ell-2},v_{\ell-4}) \in B \times N^{(\ell-4)}(A,R)$ in $G(i)$, and whenever $\neg\cB_{1,i}(\Sigma)$ holds there can not be `too many' such pairs. 
As we shall see, from this we can deduce (using the extension property $\cU_{T}$) that the number of `useless' tuples is small compared to $|T_{\Sigma,\ell-4}(i)|$.  
Second, recall that not all tuples $(v_0, \ldots, v_{\ell-2}) \in T_{\Sigma,\ell-3}(i)$ are removed if $e_{i+1} \in C_{f_{\ell-2}}(i)$: some are are ignored. 
Here the key point is that $e_{i+1} \in C_{f_{\ell-2}}(i) \setminus L_{\Sigma}(i)$ is a sufficient condition for being removed, and, with \eqref{eq:closed-estimate} in mind, that $\neg\cB_{2,i}(\Sigma)$ essentially implies that $|L_{\Sigma}(i)|$ is small compared to $|C_{f_{\ell-2}}(i)|$. 
Intuitively, this will allow us to show that the ignored tuples have negligible impact, i.e., that $|Z_{\Sigma,\ell-3}(i)| \approx |T_{\Sigma,\ell-3}(i)|$.

\subsection{Proof of Theorem~\ref{thm:main_result}} 
\label{sec:finishing_the_proof}
In this section we prove Theorem~\ref{thm:main_result} assuming the following two statements. 
Intuitively, the first lemma ensures that for `good' configurations $\Sigma$ the variables $|T_{\Sigma,j}(i)|$ are dynamically concentrated, and the second lemma essentially guarantees that for every $(\tilde{v},U)$ there exists a good  $\Sigma^*=(\tilde{v},U,A,B,R)$ for which $|T_{\Sigma^*,\ell-3}(i)| \approx |Z_{\Sigma^*,\ell-3}(i)|$. 
Now we give some intuition for the trajectories our variables follow. 
Using \eqref{eq:open-estimate}, we see that the proportion of pairs which are open or an edge in $G(i)$ roughly equals $q(t)$ or $2tp$, respectively, where $t=i/(n^2p)$.  
So, using random graphs as a guide, it seems plausible to expect $|T_{\Sigma,j}(i)| \approx c_j (2tp)^{j}{q(t)}^{\ell-2-j} k^2 r^{\ell-3}$, where the factor $c_j=1/j!$ takes into account that we only count tuples created in a certain order. 
In the following results the functions $q(t)$, $f(t)$ and parameters  $k$, $m$, $p$, $r$, $u$ are defined by  \eqref{eq:Cl-parameters}, \eqref{eq:Cl-functions}, \eqref{eq:Cl-parameters2} and \eqref{eq:Cl-parameters3}. 
\begin{lemma}
\label{lem:dem:trajectories}
For all $0 \leq i^* \leq m$ and $\Sigma \in \cC$, let $\cG_{i^*}(\Sigma)$ denote the event that for every $0 \leq i \leq i^*$ and all $0 \leq j \leq \ell-3$ we have 
\begin{equation}
\label{eq:lem:dem:trajectories:T}
|T_{\Sigma,j}(i)| = \left( (2t)^{j}{q(t)}^{\ell-2-j}/j! \; \pm \; f(t) {q(t)}^{\ell-3-j}/n^{2\epsilon} \right) k^2 r^{\ell-3}p^{j}   ,
\end{equation} 
and let $\cE_{j}$ denote the event that for all $0 \leq i \leq j$ and $\Sigma \in \cC$ the event $\cB_{i-1}(\Sigma) \cup \cG_{i}(\Sigma)$ holds. 
Then $\cE_m$ holds whp in the $C_{\ell}$-free process. 
\end{lemma}
\begin{lemma}
\label{lem:dem:config}
Let $\cR_{j}$ denote the event that for all $0 \leq i \leq j$, for every $(\tilde{v},U) \in [n] \times \binom{[n]}{u}$ with $\tilde{v} \notin U$ there exists $\Sigma^*=(\tilde{v},U,A,B,R) \in \cC$ such that $\neg\cB_{i-1}(\Sigma^*)$ holds and 
\begin{equation}
\label{eq:lem:dem:config:ignored}
|T_{\Sigma^*,\ell-3}(i) \setminus Z_{\Sigma^*,\ell-3}(i)| \leq k^2 (rp)^{\ell-3}n^{-9\epsilon}   . 
\end{equation} 
Then $\cR_m$ holds whp in the $C_{\ell}$-free process. 
\end{lemma}
The proofs of these lemmas are rather involved and therefore deferred to Sections~\ref{sec:trajectory_verification} and~\ref{sec:good_configurations_exist}. 
With these results in hand, we are now ready to establish our main result. 
\begin{proof}[Proof of Theorem~\ref{thm:main_result}] 
For the sake of concreteness, we prove the theorem with $D=\gamma \mu$. 
Given $\tilde{v} \in [n]$, $U \subseteq [n] \setminus \{\tilde{v}\}$ and $i \leq m$, let $\cX_{\tilde{v},U,i}$ denote the event that up to \mbox{step $i$}, there is no $C_{\ell}$-extension for $(\tilde{v},U)$ in the $C_{\ell}$-free process. 
By $\cX_m$ we denote the event that there exists $(\tilde{v},U) \in [n] \times \binom{[n]}{u}$ with $\tilde{v} \notin U$ for which $\cX_{\tilde{v},U,m}$ holds.  
Furthermore, for every $i \leq m$ we set $\cA_i = \cE_i \cap \cR_i \cap \cT_i$, where $\cT_i$ is defined as in Theorem~\ref{thm:BohmanKeevash2010H} and $\cE_i$, $\cR_i$ as in Lemmas~\ref{lem:dem:trajectories} and~\ref{lem:dem:config}. 
If $\cX_m$ fails, then, as discussed in Section~\ref{sec:motivation}, the $C_{\ell}$-free process has maximum degree at most $u = D (n\log n)^{1/(\ell-1)}$. 
So, since $\cA_m$ holds whp by Theorem~\ref{thm:BohmanKeevash2010H} and Lemmas~\ref{lem:dem:trajectories}  and~\ref{lem:dem:config}, to complete the proof it suffices to show 
\begin{equation}
\label{eq:thm:main_result:prob_no_extension_small}
\PP[\cX_m \cap \cA_m] = o(1)   .
\end{equation}
Suppose that for $m/2 \leq i \leq m$ the event $\cA_i = \cE_i \cap \cR_i \cap \cT_i$ holds. 
Observe that $\cE_i \cap \neg\badEC[i-1]$ implies $\cG_{i}(\Sigma)$, which is defined as in Lemma~\ref{lem:dem:trajectories}. Using \eqref{eq:Cl-parameters} we see that $m/2 \leq i \leq m$ implies $t=i/(n^2p)=\omega(1)$, so for $j=\ell-3$ the main term in the brackets of \eqref{eq:lem:dem:trajectories:T} is $(2t)^{\ell-3}{q(t)}/(\ell-3)!$ since $f(t)/[n^{2\epsilon}q(t)]=o(1)$ by \eqref{eq:Cl-functions-estimates}.  
Thus, whenever $\cE_i \cap \cR_i$ holds, using \eqref{eq:lem:dem:trajectories:T}, \eqref{eq:lem:dem:config:ignored} and $q(t) \geq n^{-\epsilon/4}$, it follows that for every $(\tilde{v},U)$ with $U \in \binom{[n]\setminus\{\tilde{v}\}}{u}$ there exists $\Sigma^*=(\tilde{v},U,A,B,R) \in \cC$ satisfying 
\begin{equation*}
|T_{\Sigma^*,\ell-3}(i)| \geq k^2(2tpr)^{\ell-3}q(t)/(\ell-1)! 
\quad \text{ and } \quad |T_{\Sigma^*,\ell-3}(i) \setminus Z_{\Sigma^*,\ell-3}(i)| \leq k^2 (2tpr)^{\ell-3}q(t)n^{-7\epsilon}   . 
\end{equation*} 
Note that $\cT_i$ gives $q(t) \geq |O(i)|/n^2$ by \eqref{eq:Cl-functions-estimates} and \eqref{eq:open-estimate}. 
So, combining our findings with $Z_{\Sigma^*,\ell-3}(i) \subseteq T_{\Sigma^*,\ell-3}(i)$, using $k = u/60$, $r \geq n/\ell$, \eqref{eq:Cl-parameters2} and $t=i/(n^2p)$ we see that for such $\Sigma^*$ we crudely have 
\begin{equation}
\label{eq:thm:main_result:many_open:tuples}
\begin{split}
|Z_{\Sigma^*,\ell-3}(i)| &=  |T_{\Sigma^*,\ell-3}(i)| - |T_{\Sigma^*,\ell-3}(i) \setminus Z_{\Sigma^*,\ell-3}(i)| \geq k^2(2tpr)^{\ell-3}q(t) / \ell! \\
&\geq \delta u^2(tpn)^{\ell-3} q(t) = \delta \frac{u^{2} i^{\ell-3}}{n^{\ell-3}} q(t) \geq \delta  \frac{u^{2} i^{\ell-3}}{n^{\ell-1}} |O(i)|   . 
\end{split}
\end{equation} 
Recall that $O_{\tilde{v},U}(i) \subseteq O(i)$ denotes the set of open pairs which would complete a $C_{\ell}$-extension for $(\tilde{v},U)$ if chosen as the next edge $e_{i+1}$. 
Let $O_{\Sigma^*}(i)$ be the set of all $xy \in O(i)$ for which there exists $(v_{0}, \ldots, v_{\ell-2}) \in Z_{\Sigma^*,\ell-3}(i)$ with $f_{\ell-2}=xy$. 
As already discussed in Section~\ref{sec:main-proof:variables}, by construction we have $O_{\Sigma^*}(i) \subseteq O_{\tilde{v},U}(i)$, and $\cU_T$ implies $|O_{\Sigma^*}(i)| = |Z_{\Sigma^*,\ell-3}(i)|$. 
Together with \eqref{eq:thm:main_result:many_open:tuples} this establishes 
\begin{equation}
\label{eq:thm:main_result:many_open}
|O_{\tilde{v},U}(i)| \geq \delta  \frac{u^{2} i^{\ell-3}}{n^{\ell-1}} |O(i)|   . 
\end{equation}
Using this estimate, we now prove \eqref{eq:thm:main_result:prob_no_extension_small}. 
To this end fix $(\tilde{v},U) \in [n] \times \binom{[n]}{u}$ with $\tilde{v} \notin U$. We see that 
\begin{equation}
\label{eq:thm:main_result:prob_no_extension_prod}
\begin{split}
\PP[\cX_{\tilde{v},U,m} \cap \cA_m] &= \PP[\cX_{\tilde{v},U,m/2} \cap \cA_{m/2}] \prod_{m/2 \leq i \leq m-1} \PP[\cX_{\tilde{v},U,i+1} \cap \cA_{i+1} \mid \cX_{\tilde{v},U,i} \cap \cA_i]\\
&\leq  \prod_{m/2 \leq i \leq m-1} \PP[e_{i+1} \notin O_{\tilde{v},U}(i) \mid \cX_{\tilde{v},U,i} \cap \cA_i]   .
\end{split}
\end{equation}
Note that $\cX_{\tilde{v},U,i} \cap \cA_i$ depends only on the first $i$ steps of the process, so given this, the process fails to choose $e_{i+1}$ from $O_{\tilde{v},U}(i)$ with probability $1-|O_{\tilde{v},U}(i)|/|O(i)|$. 
Now from \eqref{eq:thm:main_result:many_open} and \eqref{eq:thm:main_result:prob_no_extension_prod} as well as the inequality $1-x \leq e^{-x}$ we deduce, with room to spare, 
\begin{equation}
\label{eq:prob_no_extension_closed}
\PP[\cX_{\tilde{v},U,m} \cap \cA_m] \leq \exp\left\{- \delta  \frac{u^{2} }{n^{\ell-1}} \sum_{m/2 \leq i \leq m-1} i^{\ell-3}\right\} \leq \exp\left\{- \frac{\delta}{2^{\ell}} \frac{u^{2} m^{\ell-2}}{n^{\ell-1}}\right\}   .
\end{equation}
Substituting the definitions of $m$, $u$, $p$ and $t_{\max}$ into \eqref{eq:prob_no_extension_closed} we obtain 
\[
\PP[\cX_{\tilde{v},U,m} \cap \cA_m] \leq \exp\left\{- \frac{\delta \gamma }{2^{\ell}} n^{\ell-2} p^{\ell-1} t_{\max}^{\ell-1}  u \right\} = \exp\left\{- \gamma  \frac{\delta \mu^{\ell-1}}{2^{\ell}} u \log n\right\} \leq n^{- 2 u}   , 
\]
where the last inequality follows from \eqref{eq:Cl-parameters2}, i.e., the definition of $\gamma$. 
Finally, taking the union bound over all choices of $(\tilde{v},U)$ implies \eqref{eq:thm:main_result:prob_no_extension_small}, which, as explained, completes the proof. 
\end{proof}

\section{Trajectory verification} 
\label{sec:trajectory_verification}
This section is devoted to the proof of Lemma~\ref{lem:dem:trajectories}. 
Henceforth we work with the `natural' filtration given by the $C_{\ell}$-free process, where $\cF_i$ corresponds to the first $i$ steps, and  tacitly assume that $n$ is sufficiently large whenever necessary. 
For every $0 \leq i \leq m$ we set $\cH_i = \cJ_i \cap \cT_i$, where $\cJ_i$, $\cT_i$ are defined as in Theorem~\ref{thm:BohmanKeevash2010H}. 
Clearly, $\cH_m$ holds whp. Furthermore $\cH_{i+1} \subseteq \cH_i$ and $\cH_i \in \cF_i$, since $\cH_i$ is monotone decreasing and depends only on the first $i$ steps. 
We set $s=n^2p$ and apply the differential equation method (Lemma~\ref{lem:dem}) with $\cV=\{0, \ldots, \ell-3\}$.  
Recalling that $\badE$ is monotone increasing, we see that $\badE = \badEL$. 
For all $\sigma \in \cC \times \cV$ we define
\begin{equation}
\label{def:dem:parameters}
u_{\sigma} = k n^{15\ell\epsilon} = \omega(1), \qquad \lambda_{\sigma} = \tau_{\sigma} = n^{\epsilon}, \qquad \beta_{\sigma}=1, \qquad \text{ and } \qquad s_{\sigma} = s_{o} = n^{2\epsilon} .
\end{equation}
Formally, for all $\sigma = (\Sigma,j) \in \cC \times \cV$ we set $X_{\sigma}(i) = |T_{\Sigma,j}(i)|$ and $Y^{\pm}_{\sigma}(i) = |T^{\pm}_{\Sigma,j}(i)|$, where $T^+_{\Sigma,j}(i) = T_{\Sigma,j}(i+1) \setminus T_{\Sigma,j}(i)$ and $T^-_{\Sigma,j}(i) = T_{\Sigma,j}(i) \setminus T_{\Sigma,j}(i+1)$. 
But, for the sake of clarity,  we will henceforth just use $|T_{\Sigma,j}(i)|$ and $|T^{\pm}_{\Sigma,j}(i)|$. 
Now, for every $\sigma = (\Sigma,j) \in \cC \times \cV$ we set $x_{\sigma}(t) = x_j(t)$, $y_{\sigma}^{\pm}(t) = x^{\pm}_j(t)$, $S_{\sigma} = S_j$, $f_{\sigma}(t) = f_j(t)$ and $h_{\sigma}(t) = h_j(t)$, where  
\begin{align}
\label{def:xj_Sj}
x_j(t) &= 1/j! \cdot (2t)^{j}{q(t)}^{\ell-2-j}, &  
S_j &= k^2 r^{\ell-3}p^j ,  \\
\label{def:xpj_fj}
x^{+}_j(t) &= 2j/j! \cdot (2t)^{j-1}{q(t)}^{\ell-2-j}, &  
f_j(t) &= f(t) {q(t)}^{\ell-3-j} , \\ 
\label{def:xmj_hj}
x^{-}_j(t) &= 2 (\ell-2-j) (\ell-1) (2t)^{\ell-2}x_j(t),  
& h_j(t) &= f'_j(t)/2   . 
\end{align}
The definition of $ x^{+}_j(t)$ might seem overly complicated, but it conveniently ensures $x^{+}_0(t)=0$ and $x^{+}_{j}(t) = 2 x_{j-1}(t)/q(t)$ for $j > 0$. 
With the above parametrization we can restate \eqref{eq:lem:dem:trajectories:T} as
\begin{equation}
\label{eq:lem:dem:trajectories:T:simplified}
|T_{\Sigma,j}(i)| = \left( x_j(t) \pm f_j(t)/s_{o} \right) k^2 r^{\ell-3}p^{j}   . 
\end{equation}
The remainder of this section is organized as follows. 
First, in Section~\ref{sec:proof:verification:trend} we verify the trend hypothesis of Lemma~\ref{lem:dem}, and, next, the boundedness hypothesis in Section~\ref{sec:proof:verification:bound}. 
Finally, in Section~\ref{sec:proof:verification:end} we check the remaining conditions of  the differential equation method.

\subsection{Trend hypothesis} 
\label{sec:proof:verification:trend}
In order to establish~\eqref{eq:lem:dem:parameter_martingale_property}, whenever $\cE_i \cap \neg\badE \cap \cH_i$ holds, for every $j \in \cV$ we have to prove 
\begin{equation}
\label{eq:Tj:th_goal}
\EE[|T^{\pm}_{\Sigma,j}(i)| \mid \cF_i ] = \left(x^{\pm}_j(t) \pm \frac{h_j(t)}{s_{o}} \right) \frac{k^2 r^{\ell-3}p^{j}}{n^2p}  .
\end{equation}

\subsubsection{Basic estimates}  
\label{sec:pm_inequalities}
The following inequalities were given in~\cite{Warnke2010K4}, and can easily be verified using elementary calculus. 
Recall that $a \pm b$ denotes the interval $\{a + xb : -1 \leq x \leq 1 \}$, see Section~\ref{sec:notation}. 
\begin{lemma}%
\label{lem:pm_inequalities}%
\normalfont{\cite[Lemma~$7.1$]{Warnke2010K4}}
Suppose $0 \leq x \leq 1/2$. Then
\begin{equation}
\label{eq:pm_inverse}
(1 \pm x)^{-1} \subseteq 1 \pm 2 x    . 
\end{equation}
\end{lemma} 
\begin{lemma}%
\label{lem:pm_product_inequality_extended}%
\normalfont{\cite[Lemma~$7.2$]{Warnke2010K4}}
Suppose $x,y, f_x, f_y, g, h \geq 0$ and $g \leq 1$. 
Then $f_x + x g\leq h/2$ implies 
\begin{equation}
\label{eq:pm_product_ext1}
(1 \pm g) ( x \pm f_{x} ) \subseteq x \pm h   . 
\end{equation}
Furthermore, $x f_y + y f_x + f_x f_y + x y g \leq h/2$ implies  
\begin{equation}
\label{eq:pm_product_ext2}
(1 \pm g) ( x \pm f_{x} ) ( y \pm f_{y} )  \subseteq xy \pm h   . 
\end{equation}
\end{lemma}

\subsubsection{Triples added in one step.} 
\label{sec:proof:verification:trend:added}
In this section we verify \eqref{eq:Tj:th_goal} for $T^{+}_{\Sigma,j}(i)$.

\textbf{The case $j=0$.} 
Clearly, adding an edge to $G(i)$ can not create new open tuples in $T_{\Sigma,0}(i)$. Thus we always have $|T^+_{\Sigma,0}(i)| = 0 = x^+_0(t)$, which settles this case.

\textbf{The case $j > 0$.} 
Recall that $e_{i+1} \in O(i)$ is added to $G(i)$. 
Let $P_{\Sigma,j-1}(i)$ contain all $(v_0, \ldots, v_{\ell-2}) \in T_{\Sigma,j-1}(i)$ for which there exists a path $w_0 \ldots w_j=v_j$ with $w_0 \in A$ in $G(i)$. 
Similarly, $D_{\Sigma,j-1}(i) \subseteq T_{\Sigma,j-1}(i)$ contains all tuples with $\{f_{j+1}, \ldots, f_{\ell-2}\} \cap C_{f_{j}}(i) \neq \emptyset$, where $f_{j'} = v_{j'-1}v_{j'}$. 
With these definitions in hand, 
note that $(v_0, \ldots, v_{\ell-2}) \in T_{\Sigma,j-1}(i)$ is added to $T_{\Sigma,j}(i+1)$, i.e., is in $T_{\Sigma,j}(i+1)$, if and only if $f_j = e_{i+1}$ and $(v_0, \ldots, v_{\ell-2}) \notin P_{\Sigma,j-1}(i) \cup D_{\Sigma,j-1}(i)$, see Section~\ref{sec:main-proof:variables}. 
Since the $C_{\ell}$-free process chooses $e_{i+1}$ uniformly at random from 
$O(i)$, 
whenever $\cE_i \cap \neg\badE \cap \cH_i$ holds we have 
\begin{equation}
\label{eq:Tj:th+_0}
\EE[|T^+_{\Sigma,j}(i)| \mid \cF_i ] = \sum_{(v_0, \ldots, v_{\ell-2}) \in T_{\Sigma,j-1}(i) \setminus [P_{\Sigma,j-1}(i) \cup D_{\Sigma,j-1}(i)]} \frac{1}{|O(i)|}   . 
\end{equation}
We now bound the size of $P_{\Sigma,j-1}(i)$. 
Since $\cH_i$ implies \eqref{eq:degree-estimate}, the degree of every vertex is bounded by, say, $npn^{\epsilon}$. 
So, using $|A| = k \leq npn^{\epsilon}$, $j \leq \ell-3$, $(np)^{\ell-2}=n^{1- 1/(\ell-1)}$ and $r \geq n/\ell$, in $G(i)$ the number of $w_j$ for which there exists a path $w_0 \ldots w_j$ with $w_0 \in A$ is at most 
\begin{equation}
\label{eq:Tj:th:ineq}
|A| \cdot (npn^{\epsilon})^{j} \leq (npn^{\epsilon})^{\ell-2} \leq n^{1+\ell\epsilon - 1/(\ell-1)} \leq r  n^{-1/(2\ell)}   . 
\end{equation}
Given $w_j$, we now bound the number of $(v_0, \ldots, v_{\ell-2}) \in T_{\Sigma,j-1}(i)$ with $w_j=v_j$. 
Observe that there are at most $k (npn^{\epsilon})^{j-1}$ choices for such $v_1, \ldots, v_{j-1}$, and at most $r^{\ell-j-3}k$ choices for $v_{j+1}, \ldots, v_{\ell-2}$. 
Putting things together, we deduce that 
\begin{equation}
\label{eq:Tj:th:P_Sigma_size}
|P_{\Sigma,j-1}(i)| \leq r  n^{-1/(2\ell)} \cdot k (npn^{\epsilon})^{j-1} \cdot r^{\ell-j-3} k \leq k^2 r^{\ell-3} p^{j-1} n^{-1/(3\ell)}   . 
\end{equation}
Turning to $D_{\Sigma,j-1}(i)$, we first consider the case where $0 < j < \ell-3$. \marginal{Recheck if correct!}
Suppose that $f_h \in C_{f_{j}}(i)$. 
Depending on whether $h=j+1$ or $h > j+1$, there exists either a path $v_{j-1}=w_1 \cdots w_{\ell-1}=v_h$ with $j < h < \ell-2$, or a path $w_1 \cdots w_{\kappa}=v_{h-1}$ with $w_1 \in \{v_j,v_{j-1}\}$, $1 < \kappa \leq \ell-2$ and $j < h-1<\ell-2$, cf.\ Figure~\ref{fig:tuples:selfclosing}. 
\begin{figure}[t]
	\centering
  \setlength{\unitlength}{1bp}%
  \begin{picture}(60.00, 50.00)(0,0)
  \put(0,-2.75){\includegraphics{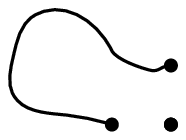}}
  \put(22.67,8.12){\fontsize{11.38}{13.66}\selectfont $v_{h}$}
  \put(41.19,36.46){\fontsize{11.38}{13.66}\selectfont $v_{j-1}$}
  \put(41.19,8.12){\fontsize{11.38}{13.66}\selectfont $v_j\!=\!v_{h-1}$}
\end{picture}%
\hspace{5.5em}
  \setlength{\unitlength}{1bp}%
  \begin{picture}(82.00, 50.00)(0,0)
  \put(0,0){\includegraphics{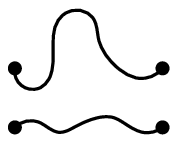}}
  \put(6.67,36.46){\fontsize{11.38}{13.66}\selectfont $v_h$}
  \put(6.67,8.12){\fontsize{11.38}{13.66}\selectfont $v_{h-1}$}
  \put(49.19,36.46){\fontsize{11.38}{13.66}\selectfont $v_{j-1}$}
  \put(49.19,8.12){\fontsize{11.38}{13.66}\selectfont $v_{j}$}
  \end{picture}%
\hspace{3.0em}
  \setlength{\unitlength}{1bp}%
  \begin{picture}(82.00, 50.00)(0,0)
  \put(0,0){\includegraphics{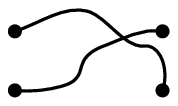}}
  \put(6.67,36.46){\fontsize{11.38}{13.66}\selectfont $v_h$}
  \put(6.67,8.12){\fontsize{11.38}{13.66}\selectfont $v_{h-1}$}
  \put(49.19,36.46){\fontsize{11.38}{13.66}\selectfont $v_{j-1}$}
  \put(49.19,8.12){\fontsize{11.38}{13.66}\selectfont $v_{j}$}
  \end{picture}%
	\caption{\label{fig:tuples:selfclosing}The solid lines represent paths such that adding both $f_{j}=v_{j-1}v_{j}$ and $f_{h}=v_{h-1}v_{h}$ completes a copy of $C_{\ell}$ consisting of those paths. In other words, adding $f_{j}$ closes $f_{h}$, i.e., $f_{h} \in C_{f_{j}}(i)$.}
\end{figure} 
So, in both cases, there exists a path $w_1 \cdots w_{\kappa}=v_{x}$ with $w_1 \in \{v_j,v_{j-1}\}$, $1 < \kappa \leq \ell-1$ and $j < x < \ell-2$. 
With this observations in hand, we are now ready to estimate the number of tuples $(v_{0}, \ldots, v_{\ell-2}) \in  D_{\Sigma,j-1}(i)$. 
Recall that by $\cH_i$ the degree of every vertex is at most $npn^{\epsilon}$. 
It follows that there are at most $k (npn^{\epsilon})^{j-1}r$ choices for $v_0, \ldots, v_{j}$, and at most $\ell^2$ choices for $h$ and $x$. 
Given $v_0, \ldots, v_{j}$ as well as $h$ and $x$, there are at most $2\ell(npn^{\epsilon})^{\ell-2} \leq r  n^{-1/(3\ell)}$ choices for $v_{x}$ by \eqref{eq:Tj:th:ineq}. 
Since we already picked $v_x$ with $j < x < \ell-2$, for the remaining vertices among $v_{j+1}, \ldots, v_{\ell-2}$ we have at most $r^{\ell-j-4}k$ choices. 
Putting things together, we see that for $0 < j < \ell-3$ we have 
\begin{equation}
\label{eq:Tj:th:D_Sigma_size}
|D_{\Sigma,j-1}(i)| \leq  k (npn^{\epsilon})^{j-1} \cdot r \cdot \ell^2 \cdot r n^{-1/(3\ell)} \cdot r^{\ell-j-4}k \leq  k^2 r^{\ell-3} p^{j-1} n^{-1/(4\ell)}   . 
\end{equation}
Now we bound $|D_{\Sigma,j-1}(i)|$ for the remaining case $j = \ell-3$. 
Recall that $f_{\ell-3}= v_{\ell-4}v_{\ell-3}$. 
If $f_{\ell-2}=v_{\ell-3}v_{\ell-2} \in C_{f_{\ell-3}}(i)$, then, with a similar reasoning as in the previous case, there exists a path $v_{\ell-4}=w_{0} \cdots  w_{\ell-2}=v_{\ell-2}$, where $v_{\ell-4} \in N^{(\ell-4)}(A,R)$ and $v_{\ell-2} \in B$.
Since $\neg\badE$ holds, by $\neg\badE[1,i]$ there are at most $k^2(np)^{\ell-4}n^{-9\epsilon}$ such pairs $(v_{\ell-2}, v_{\ell-4}) \in B \times N^{(\ell-4)}(A,R)$ in $G(i)$. 
Recall that by the extension property $\cU_{T}$ (cf.\ Lemma~\ref{lemma:extension:property}) every triple $(v_{\ell-4},v_{\ell-3}, v_{\ell-2})$ is contained in at most one tuple in $T_{\Sigma,\ell-4}(i)$. 
So, since there are at most $k^2(np)^{\ell-4}n^{-9\epsilon}$ choices for $v_{\ell-4},v_{\ell-2}$, and at most $r$ choices for $v_{\ell-3} \in V_{\ell-3}$, using $\cU_{T}$ we deduce that for $j=\ell-3$ we have 
\begin{equation}
\label{eq:Tj:th:D_Sigma_size_l3}
|D_{\Sigma,j-1}(i)| \leq k^2(np)^{\ell-4}n^{-9\epsilon} \cdot r \leq  k^2 r^{\ell-3} p^{\ell-4} n^{-8\epsilon} = k^2 r^{\ell-3} p^{j-1} n^{-8\epsilon}   . 
\end{equation}
After these preparations, we now estimate \eqref{eq:Tj:th+_0} whenever $\cE_i \cap \neg\badE \cap \cH_i$ holds. 
Observe that $\cE_i \cap \neg\badE$ implies $\cG_{i}(\Sigma)$, and so $|T_{\Sigma,j-1}(i)|$ satisfies  \eqref{eq:lem:dem:trajectories:T:simplified}. 
Furthermore, since $\cH_i$ holds, this implies that $|O(i)|$ satisfies \eqref{eq:open-estimate}. 
In addition, note that $s_{e} = n^{1/(2\ell)-\epsilon}$ and \eqref{eq:Cl-functions-estimates} imply $f(t)/s_{e} = o(1)$ and $f_{j-1}(t) \geq 1$. 
Substituting the former estimates and \eqref{eq:Tj:th:P_Sigma_size}--\eqref{eq:Tj:th:D_Sigma_size_l3} into \eqref{eq:Tj:th+_0}, using $n^{1/(3\ell)} \geq n^{8\epsilon} = \omega(s_{o})$, \eqref{eq:pm_inverse}, $x^{+}_{j}(t) = 2 x_{j-1}(t)/q(t)$ and $f_{j}(t) = f_{j-1}(t) / q(t)$, we deduce that 
\begin{equation*}
\label{eq:Tj:th+_1}
\begin{split}
\EE[|T^+_{\Sigma,j}(i)| \mid \cF_i ] &= \frac{( x_{j-1}(t) \pm f_{j-1}(t)/s_{o}) k^2 r^{\ell-3}p^{j-1} \pm 2k^2 r^{\ell-3} p^{j-1} n^{-8\epsilon}}{(1 \pm 3f(t)/s_e)q(t)n^2/2}\\
& \subseteq \frac{( x_{j-1}(t) \pm 2f_{j-1}(t)/s_{o}) k^2 r^{\ell-3}p^{j-1}}{(1 \pm 3f(t)/s_e)q(t)n^2/2}\\
& \subseteq (1 \pm 6f(t)/s_e) \cdot (x_{j}^{+}(t) \pm 4f_{j}(t)/s_{o}) \cdot k^2 r^{\ell-3}p^{j}/(n^2p)   . 
\end{split}
\end{equation*}
Therefore the desired bound, i.e., \eqref{eq:Tj:th_goal} for $T^{+}_{\Sigma,j}(i)$, follows if 
\begin{equation}
\label{eq:Tj:th+_sufficient_goal}
(1 \pm 6f(t)/s_e) \cdot (x_{j}^{+}(t) \pm 4f_{j}(t)/s_{o}) \subseteq x_{j}^{+}(t) \pm h_j(t)/s_{o}   .
\end{equation}
Now, using $f(t)=o(s_e)$ and Lemma~\ref{lem:pm_product_inequality_extended}, by writing down the assumptions of \eqref{eq:pm_product_ext1} and  multiplying both sides with $2 s_{o}$, observe that \eqref{eq:Tj:th+_sufficient_goal} follows from 
\begin{equation*}
\label{eq:Tj:th+__sufficient}
8 f_{j}(t) + 12 x_{j}^{+}(t)f(t) s_{o}/s_{e} \leq h_j(t)   . 
\end{equation*}
Using \eqref{eq:Cl-functions-estimates} and \eqref{def:xpj_fj} we see that the second term on the left hand side is $o(1)$. So, it suffices if 
\[
8 f_j(t) + 1 \leq h_j(t)   ,
\]
which is easily seen to be true, since $h_j(t) \geq W/4 \cdot (f_j(t) + 1)$ and $W \geq 50$ by \eqref{eq:Cl-constants:Wepsmu}, \eqref{eq:Cl-functions-estimates} and \eqref{def:xmj_hj}.

\subsubsection{Triples removed in one step} 
\label{sec:proof:verification:trend:removed}
Next, we prove \eqref{eq:Tj:th_goal} for $T^{-}_{\Sigma,j}(i)$.  
Since the rules for removing tuples from $T_{\Sigma,j}(i)$ are different for $j < \ell-3$ and $j = \ell-3$, we use a case distinction.

\textbf{The case $j < \ell-3$.} 
Recall that a tuple $(v_0, \ldots, v_{\ell-2}) \in T_{\Sigma,j}(i)$ is removed, i.e., not in $T_{\Sigma,j}(i+1)$, if $e_{i+1} \in \{f_{j+1}, \ldots, f_{\ell-2}\}$ or $e_{i+1} \in C_{f_{j+1}}(i) \cup \cdots \cup C_{f_{\ell-2}}(i)$. 
Since the edge $e_{i+1}$ is chosen uniformly at random from $O(i)$,  whenever $\cE_i \cap \neg\badE \cap \cH_i$ holds, using $|\{f_{j+1}, \ldots, f_{\ell-2}\}| \leq \ell$ we have 
\begin{equation}
\label{eq:Tj:th-_0}
\EE[|T^-_{\Sigma,j}(i)| \mid \cF_i ] = \sum_{(v_0, \ldots, v_{\ell-2}) \in T_{\Sigma,j}(i)} \frac{|C_{f_{j+1}}(i) \cup \cdots \cup C_{f_{\ell-2}}(i)| \pm \ell}{|O(i)|}    .
\end{equation}
Note that $\cH_i$ implies that the inequalities \eqref{eq:open-estimate}, \eqref{eq:closed-estimate} and \eqref{eq:closed-intersection-estimate} hold. 
In particular, using $n^{1/\ell} = \omega(s_e)$, $n^{-1/\ell} p^{-1} = \omega(1)$ and $f(t) \geq 1$, this yields 
\begin{equation}
\label{eq:Tj:th-closed}
\begin{split}
|C_{f_{j+1}}(i) \cup \cdots \cup C_{f_{\ell-2}}(i)| \pm \ell &\subseteq (\ell-j-2)[(\ell-1)(2t)^{\ell-2} q(t) \pm 7\ell f(t)/s_e]p^{-1} \pm \ell^2 n^{-1/\ell} p^{-1} \pm \ell\\
& \subseteq (\ell-j-2)[(\ell-1)(2t)^{\ell-2}q(t) \pm 9\ell f(t)/s_e]p^{-1}   . 
\end{split}
\end{equation} 
Since $\cE_i \cap \neg\badE$ implies $\cG_{i}(\Sigma)$, it follows that $|T_{\Sigma,j}(i)|$ satisfies \eqref{eq:lem:dem:trajectories:T:simplified}. 
In addition, as in Section~\ref{sec:proof:verification:trend:added}, $f(t)/s_e = o(1)$ holds and $|O(i)|$ satisfies \eqref{eq:open-estimate} by $\cH_i$. 
Substituting the former estimates into \eqref{eq:Tj:th-_0}, and using \eqref{eq:pm_inverse} as well as $x^{-}_{j}(t)/x_{j}(t) = 2 (\ell-j-2) (\ell-1) (2t)^{\ell-2}$, we obtain
\begin{equation*}
\label{eq:Tj:th-_1}
\begin{split}
& \EE[|T^-_{\Sigma,j}(i)| \mid \cF_i ] = \frac{( x_{j}(t)	 \pm f_j(t)/s_{o}) k^2 r^{\ell-3}p^{j} \cdot  (\ell-j-2)[ (\ell-1) (2t)^{\ell-2} q(t) \pm 9\ell f(t)/s_e]p^{-1}}{(1 \pm 3f(t)/s_e)q(t)n^2/2}\\
& \qquad \subseteq  (1 \pm 6f(t)/s_e ) \cdot ( x_{j}(t) \pm f_j(t)/s_{o} ) \cdot [x^{-}_{j}(t)/x_{j}(t) \pm 20 \ell^2 f(t)/({q(t)}s_e) ] \cdot k^2 r^{\ell-3}p^{j}/(n^2p)   . 
\end{split}
\end{equation*}
Therefore the desired bound, i.e., \eqref{eq:Tj:th_goal} for $T^{-}_{\Sigma,j}(i)$, follows if 
\begin{equation}
\label{eq:open_triples:th_sufficient_goal}
(1 \pm 6f(t)/s_e ) \cdot ( x_{j}(t) \pm f_j(t)/s_{o} ) \cdot [ x^{-}_{j}(t)/x_{j}(t) \pm 20 \ell^2 f(t)/({q(t)}s_e) ] \subseteq x_j^-(t) \pm h_j(t)/s_{o}   .
\end{equation}
We now show \eqref{eq:open_triples:th_sufficient_goal} using Lemma~\ref{lem:pm_product_inequality_extended}. 
Similar as for the added tuples, by writing down the assumptions of \eqref{eq:pm_product_ext2}, multiplying with $2 s_{o}$ and then noticing that all terms containing $s_{e}$ contribute $o(1)$, we see that 
it suffices if 
\[
(\ell-2-j)(\ell-1)2^{\ell} t^{\ell-2} f_j(t) + 1 \leq h_j(t) ,
\]
which is easily seen to be true, since $h_j(t) \geq W/2 \cdot ( t^{\ell-2} f_j(t) + 1)$ and $W/2 \geq \ell^2 2^{\ell}$ by \eqref{eq:Cl-constants:Wepsmu} and \eqref{def:xmj_hj}.

\textbf{The case $j = \ell-3$.} 
Recall that a tuple $(v_{0}, \ldots, v_{\ell-2}) \in T_{\Sigma,\ell-3}(i)$ is removed, i.e., not in $T_{\Sigma,\ell-3}(i+1)$, if $e_{i+1}=f_{\ell-2}$, or in addition to $e_{i+1} \in C_{f_{\ell-2}}(i)$ it is not ignored.  
A moment's thought reveals that for every $(v_{0}, \ldots, v_{\ell-2}) \in T_{\Sigma,\ell-3}(i)$ with $e_{i+1} \in C_{f_{\ell-2}}(i)$, if $e_{i+1} \notin L_{\Sigma}(i)$ then (R2) holds, where $L_{\Sigma}(i)$ is as in the definition of $\badE[2,i]$. 
In other words, for every $(v_{0}, \ldots, v_{\ell-2}) \in T_{\Sigma,\ell-3}(i)$ we see that $e_{i+1} \in C_{f_{\ell-2}}(i) \setminus L_{\Sigma}(i)$ is a sufficient condition for being removed. 
Clearly, a necessary condition for being removed is $e_{i+1} \in \{f_{\ell-2}\} \cup C_{f_{\ell-2}}(i)$. 
Combining our previous findings and using that $e_{i+1}$ is chosen uniformly at random from $O(i)$, 
whenever $\cE_i \cap \neg\badE \cap \cH_i$ holds we deduce that 
\begin{equation*}
\label{eq:Tl3:th-_0}
\EE[|T^-_{\Sigma,\ell-3}(i)| \mid \cF_i ] = \sum_{(v_0, \ldots, v_{\ell-2}) \in T_{\Sigma,\ell-3}(i)} \frac{|C_{f_{\ell-2}}(i)| \pm |L_{\Sigma}(i)| \pm 1}{|O(i)|}    .
\end{equation*}
Recall that $\cH_i$ implies the inequalities \eqref{eq:closed-estimate} and \eqref{eq:closed-intersection-estimate}. Furthermore, since $\neg\badE[2,i]$ holds, we have $|L_{\Sigma}(i)| \leq p^{-1}n^{-1/(2\ell)}$. 
So, similar as in the previous case, using $n^{1/(2\ell)} = \omega(s_e)$, $n^{-1/(2\ell)} p^{-1} = \omega(1)$ and $f(t) \geq 1$, we obtain  
\begin{equation*}
\label{eq:Tl3:th-closed}
\begin{split}
|C_{f_{\ell-2}}(i)| \pm |L_{\Sigma}(i)| \pm 1 &\subseteq [(\ell-1) (2t)^{\ell-2}q(t) \pm 7\ell f(t)/s_e]p^{-1} \pm p^{-1}n^{-1/(2\ell)} \pm 1\\
& \subseteq [(\ell-1) (2t)^{\ell-2}q(t) \pm 9\ell f(t)/s_e]p^{-1}   ,
\end{split}
\end{equation*}
where the final estimate equals that of \eqref{eq:Tj:th-closed} for $j=\ell-3$. 
It is not difficult to see that the remaining calculations of the case $j< \ell-3$ carry over word by word, which yields \eqref{eq:Tj:th_goal} for $T^{-}_{\Sigma,\ell-3}(i)$.  
To summarize, we have verified the trend hypothesis~\eqref{eq:Tj:th_goal}.

\subsection{Boundedness hypothesis} 
\label{sec:proof:verification:bound}
Observe that in order to verify the boundedness hypothesis \eqref{eq:lem:dem:parameter_max_change}, using \eqref{def:dem:parameters} it suffices to show that whenever $\cE_i \cap \neg\badE \cap \cH_i$ holds, for every $j \in \cV$ we have 
\begin{equation}
\label{eq:dem:Tj:max_bound}
|T^{\pm}_{\Sigma,j}(i)|  \leq  k r^{\ell-3} p^{j} n^{-20\ell\epsilon} . 
\end{equation}

\subsubsection{Triples added in one step.} 
\label{sec:proof:verification:bound:added}
In this section we verify \eqref{eq:dem:Tj:max_bound} for $T^{+}_{\Sigma,j}(i)$. 
Recall that $e_{i+1} \in O(i)$ is added to $G(i)$. 
By construction we always have $|T^{+}_{\Sigma,0}(i)| = 0$, and thus we henceforth consider the case $j > 0$.  
Note that a necessary condition for $(v_0, \ldots, v_{\ell-2}) \in T_{\Sigma,j-1}(i)$ being added to $T_{\Sigma,j}(i+1)$ is $f_j = e_{i+1}$. 
Observe that there are at most $k r^{\ell-3-j}$ choices for $(v_{j+1}, \ldots, v_{\ell-2}) \in V_{j+1} \times \cdots \times V_{\ell-3} \times B$. So, using the extension property $\cU_T$ (cf.\ Lemma~\ref{lemma:extension:property}), we deduce that for each $e_{i+1}$ there are at most $k r^{\ell-3-j}$ tuples in $T_{\Sigma,j-1}(i)$ with $f_j = e_{i+1}$.
Together with \eqref{eq:Cl-parameters}, \eqref{eq:Cl-functions-estimates}, \eqref{eq:Cl-parameters3} and $j \geq 1$ this implies 
\begin{equation}
\label{eq:dem:Tj:add_bound}
|T^{+}_{\Sigma,j}(i)| \leq k r^{\ell-3-j} = k r^{\ell-3}p^{j} \cdot (rp)^{-j} = o(k r^{\ell-3} p^{j} n^{-20\ell\epsilon})  ,
\end{equation}
as desired.

\subsubsection{Triples removed in one step} 
\label{sec:proof:verification:bound:removed}
Next we use case distinction to establish \eqref{eq:dem:Tj:max_bound} for $T^{-}_{\Sigma,j}(i)$.

\textbf{The case $j < \ell-3$.} 
We claim that whenever $\cE_i \cap \neg\badE \cap \cH_i$ holds, for all $(v_{0}, \ldots, v_{\ell-2}) \in T_{\Sigma,j}(i)$ and every $xy \in \{f_{j+1}, \ldots, f_{\ell-2}\}$,  
the number of tuples in $T_{\Sigma,j}(i)$ containing $xy$ is bounded by 
\begin{equation}
\label{eq:dem:Tj:max_contained}
kr^{\ell-4}p^{j}n^{\ell\epsilon}   .
\end{equation}

First suppose that $xy = f_{j+1}$. For $(v_{j+2}, \ldots, v_{\ell-2}) \in V_{j+2} \times \cdots \times V_{\ell-3} \times B$ there are at most $k r^{\ell-4-j} \leq k r^{\ell-4}p^{j}$ choices, 
and so \eqref{eq:dem:Tj:max_contained} follows using the extension property $\cU_T$ (cf.\ Lemma~\ref{lemma:extension:property}).

Next we consider the case $xy = f_{\ell-2}$. 
As usual, whenever $\cH_i$ holds, the degree of every vertex is bounded by, say, $npn^{\epsilon}$.
Since for every $(v_{0}, \ldots, v_{\ell-2}) \in T_{\Sigma,j}(i)$ the vertices $v_{0}, \ldots, v_{j}$ form a path starting in $A$, we deduce that there are at most $k (npn^{\epsilon})^{j}$ choices for such $v_{0}, \ldots, v_{j}$. 
Furthermore, there are  most $r^{\ell-4-j}$ choices for $(v_{j+1}, \ldots, v_{\ell-4}) \in V_{j+1} \times \cdots \times V_{\ell-4}$. 
Therefore the number of tuples in $T_{\Sigma,j}(i)$ with $xy = f_{\ell-2}$ is bounded by $k (npn^{\epsilon})^{j} \cdot r^{\ell-4-j} \leq k r^{\ell-4} p^{j} n^{\ell\epsilon}$, as claimed by \eqref{eq:dem:Tj:max_contained}.

Finally we consider the case where $xy = f_{h}$ with $j+1 < h < \ell-2$. With a similar reasoning as in the previous case, there are at most $k (npn^{\epsilon})^{j}$ choices for $v_0, \ldots, v_{j}$, at most $r^{h-j-2}$ choices for $v_{j+1}, \ldots, v_{h-2}$ and at most $k r^{\ell-h-3}$ choices for $v_{h+1}, \ldots, v_{\ell-2}$. 
To summarize, there are at most 
\[
k (npn^{\epsilon})^{j} \cdot r^{h-j-2} \cdot k r^{\ell-h-3} \leq k^2 r^{\ell-5} p^{j} n^{\ell\epsilon} \leq k r^{\ell-4} p^{j} 
\]
tuples in $T_{\Sigma,j}(i)$ with $xy = f_{h}$, which establishes \eqref{eq:dem:Tj:max_contained}, with room to spare.

With the above estimate in hand, we are now ready to bound $|T^{-}_{\Sigma,j}(i)|$. 
Recall that $(v_0, \ldots, v_{\ell-2}) \in T_{\Sigma,j}(i)$ is removed, i.e., not in $T_{\Sigma,j}(i+1)$, if $e_{i+1} \in \{f_{j+1}, \ldots, f_{\ell-2}\}$ or $e_{i+1} \in C_{f_{j+1}}(i) \cup \cdots \cup C_{f_{\ell-2}}(i)$, which is equivalent to $\{f_{j+1}, \ldots, f_{\ell-2}\} \cap C_{e_{i+1}}(i) \neq \emptyset$. 
In other words, such a tuple is removed if for some $j+1 \leq h \leq \ell-2$ we have $f_{h}=e_{i+1}$ or $f_{h} \in C_{e_{i+1}}(i)$. 
Recall that whenever $\cH_i$ holds, by \eqref{eq:closed-estimate} we have, say, $|C_{e_{i+1}}(i)| \leq p^{-1}n^{\epsilon}$. 
So, using that \eqref{eq:dem:Tj:max_contained} gives an upper bound for the number of tuples in $T_{\Sigma,j}(i)$ which contain $f_{h}$, we deduce that 
\[
|T^{-}_{\Sigma,j}(i)| \leq (\ell+|C_{e_{i+1}}(i)|) \cdot kr^{\ell-4}p^{j}n^{\ell\epsilon} \leq kr^{\ell-4}p^{j-1} n^{2\ell\epsilon} \leq kr^{\ell-3}p^{j} \cdot n^{2\ell\epsilon}/(rp)   ,
\]
which, with a similar reasoning as in \eqref{eq:dem:Tj:add_bound}, establishes \eqref{eq:dem:Tj:max_bound} for $T^{-}_{\Sigma,j}(i)$ with $j < \ell-3$.

\textbf{The case $j = \ell-3$.} 
Recall that a tuple  $(v_{0}, \ldots, v_{\ell-2}) \in T_{\Sigma,\ell-3}(i)$ is removed, i.e., not in $T_{\Sigma,\ell-3}(i+1)$, according to different rules. 
In the following we bound the total number of tuples removed in one step by each rule, which were called cases~$1$ and~$2$ in Section~\ref{sec:main-proof:variables}.  
In case $1$ we have $f_{\ell-2}=e_{i+1}$ and so, given $e_{i+1}$, using  $\cU_T$ we deduce that at most one tuple is removed under case~$1$.

Turning to case~$2$, given $e_{i+1}=xy$, note that a necessary condition for being removed by (R2) is that for some $j \in [\ell-1]$ we have $f_{\ell-2} \in C_{x,y,\Sigma}(i,j)$ or  $f_{\ell-2} \in C_{y,x,\Sigma}(i,j)$. 
Recall that by $\cU_T$ every such pair $f_{\ell-2}$ is contained in at most one tuple in $T_{\Sigma,\ell-3}(i)$. 
So, since a tuple is only removed if the corresponding $C_{x,y,\Sigma}(i,j)$ or $C_{y,x,\Sigma}(i,j)$ has size at most $p^{-1}n^{-30\ell\epsilon}$, we deduce that at most $2\ell \cdot p^{-1}n^{-30\ell\epsilon}$ tuples are removed in one step by (R2).

Putting it all together, using $p^{-1} = (np)^{\ell-2}$ and $np \leq k$, for $j=\ell-3$ we obtain 
\[
|T^{-}_{\Sigma,\ell-3}(i)| \leq 1 + 2\ell p^{-1}n^{-30\ell\epsilon} \leq  (np)^{\ell-2} n^{-25\ell\epsilon} \leq k (np)^{\ell-3} n^{-25\ell\epsilon}   ,
\]
which readily establishes the boundedness hypothesis \eqref{eq:dem:Tj:max_bound}.

\subsection{Finishing the trajectory verification} 
\label{sec:proof:verification:end}
In this section we verify the remaining conditions of the differential equation method (Lemma~\ref{lem:dem}).

\textbf{Initial conditions.} 
Using \eqref{def:xj_Sj}, for $j > 0$ we clearly have $|T_{\Sigma,j}(0)|=0=x_{j}(0)$, which settles these cases.
For the remaining case $j=0$ we crudely have
\[
|T_{\Sigma,0}(0)|=|T_{\Sigma}| = k^2(r \pm kn^{10\ell\epsilon})^{\ell-3} = (1 \pm kn^{10\ell\epsilon}/r)^{\ell-3} k^2r^{\ell-3} \subseteq (1 \pm o(1)/s_{o}) k^2r^{\ell-3}    ,
\]
which together with $x_0(0)=1$, $S_0 = k^2r^{\ell-3}$ and $\beta_{\sigma}=1$ establishes \eqref{eq:lem:dem:initial_condition}.

\textbf{Bounded number of configurations and variables.} 
Using $k=u/60$ and \eqref{def:dem:parameters} we obtain 
\begin{equation*}
\label{eq:bound:nr_configs}
|\cC| \leq n \cdot \binom{n}{u} \cdot 3^u \cdot \sum_{r \leq kn^{10\ell\epsilon}} \binom{n}{r} \leq n^{2u+kn^{10\ell\epsilon}} < e^{kn^{15\ell\epsilon}} = e^{u_{\sigma}}   , 
\end{equation*} 
which together with $|\cV| \leq \ell$ clearly establishes \eqref{eq:lem:dem:bounded_parameters}.

\textbf{Additional technical assumptions and the function $f_{\sigma}(t)$.} 
Using $s=n^2p$ as well as \eqref{eq:Cl-parameters}, \eqref{eq:Cl-parameters2} and \eqref{def:dem:parameters}, straightforward calculations show that \eqref{eq:lem:dem:technical_assumptions:sm} holds, with room to spare; we leave the details to the reader. 
Recall that by \eqref{eq:Cl-parameters} we have $t_{\max} = m/s = \Theta((\log n)^{1/(\ell-1)})$. 
Furthermore, using \eqref{def:xj_Sj}--\eqref{def:xmj_hj}, elementary calculus yields $x_{j}^{\pm }(t) = O(t_{\max}^{\ell+j-2})$ 
and $|x_{j}''(t)| = O(t_{\max}^{2\ell+j-4})$ for $t \leq t_{\max}$.  
Thus, since for all $\sigma =(\Sigma,j)\in \cC \times \cV$ we have $x_{\sigma}(t) = x_{j}(t)$ and $y_{\sigma}^{\pm }(t) = x_{j}^{\pm }(t)$, it follows that 
\[
\sup_{0 \leq t \leq m/s} y_{\sigma}^{\pm }(t) = O(\log^2 n) \leq n^\epsilon = \lambda_{\sigma}  \qquad \text{ and } \qquad \int_0^{m/s} |x_{\sigma}''(t)| \ dt = O(\log n \cdot \log^3 n) \leq \lambda_{\sigma}    .
\] 
Recall that for all $\sigma \in \cC \times \cV$ we have $h_{\sigma}(t) = f'_{\sigma}(t)/2$ and $f_{\sigma}(t) = f(t) {q(t)}^{\iota}$, where $\iota \in \{0,\ldots, \ell-3\}$. Hence, using $f_{\sigma}(0)=1=\beta_{\sigma}$, we see that 
\[
f_{\sigma}(t) = 2 \int_{0}^{t} h_{\sigma}(\tau) \ d\tau + f_{\sigma}(0) = 2 \int_{0}^{t} h_{\sigma}(\tau) \ d\tau + \beta_{\sigma}   .
\] 
Note that $h_{\sigma}(0)=O(1) \leq n^{3\epsilon} = s_{\sigma} \lambda_{\sigma}$ and $h'_{\sigma}(t) \geq 0$. 
Pick $t^{*}=t^{*}(\ell) \geq 1$ large enough such that for all $t \geq t^{*}$ we have $t^{2\ell} \leq f(t)$. 
Observe that $h'_{\sigma}(t)$ is bounded by some constant for $t \leq t^{*}$, and note that for \mbox{larger $t$} we have, say, $h'_{\sigma}(t) \leq W^3{f(t)}^2$. Putting things together, using \eqref{eq:Cl-parameters} and \eqref{eq:Cl-functions-estimates}, i.e., $m/s = O(\log n)$ and $f(t) \leq n^{\epsilon}$, we readily obtain 
\[
\int_{0}^{m/s} |h'_{\sigma}(t)| \ dt \leq \int_{0}^{t^{*}} h'_{\sigma}(t) \ dt + \int_{t^{*}}^{m/s} W^3{f(t)}^2 \ dt \leq O(1) + O(\log n \cdot  n^{2\epsilon}) \leq n^{3\epsilon} = s_{\sigma} \lambda_{\sigma}  . 
\] 
To summarize, we showed that \eqref{eq:lem:dem:derivative} as well as the additional technical assumptions \eqref{eq:lem:dem:technical_assumptions:sm}--\eqref{eq:lem:dem:technical_assumptions:h} hold, and this completes the proof of Lemma~\ref{lem:dem:trajectories}. \qed

\section{A `transfer theorem' for the $H$-free process} 
\label{sec:transfer}
In the $H$-free process there is a complicated dependency among the edges, and thus standard concentration inequalities are not directly applicable. 
In this section we show how to overcome this problem for decreasing properties by establishing a `transfer theorem'. 
Roughly speaking, this allows us to `transfer' results for decreasing properties from the binomial random graph model to the $H$-free process, at the cost of only slightly increasing the `expected' edge density. 
In our argument this will be a crucial tool for establishing Lemma~\ref{lem:dem:config}.

\subsection{Relating the $H$-free process with the uniform random graph} 
We start by relating the $H$-free process with the more familiar uniform random graph. 
In the $H$-free process the set of open pairs $O(i)$ is defined in the obvious way: it contains all pairs $xy \in \binom{[n]}{2} \setminus E(i)$ for which $G(i) \cup \{xy\}$ remains $H$-free. 
The following estimate is not best possible, but it suffices for our purposes and keeps the formulas simple. 
\begin{lemma}
\label{lem:transfer:uniform}
Suppose $\cQ$ is a decreasing graph property and that $\lambda=\lambda(n) \geq 2$ is a parameter. 
Then for every $1 \leq i \leq \binom{n}{2}/\lambda$, setting $M= i \lambda$, we have 
\begin{equation}
\label{eq:lem:transfer:uniform}
\PP[G(i) \notin \cQ \text{ and } |O(i)| \geq n^2/\lambda] \leq \PP[G_{n,M} \notin \cQ] + e^{-i/4}   ,
\end{equation}
where $G(i)$ denotes the graph produced by the $H$-free process after the first $i$ steps. 
\end{lemma}
\begin{proof} 
We sequentially generate the edges $e_{1},e_{2},\ldots$, where each  edge $e_{j+1}$ is chosen uniformly at random from $E(K_n) \setminus \{e_{1},e_{2},\ldots, e_{j}\}$. 
On the one hand, the edge-set $\{e_{1},e_{2},\ldots, e_{M}\}$ clearly gives $G_{n,M}$. 
On the other hand, we obtain the graph produced by the $H$-free process by sequentially traversing the $e_j$ and only adding those edges which do not complete a copy of $H$.  
First, for every $1 \leq j \leq M$ we define the indicator variable $X_j$ for the event that $e_{j}$ is added to the graph of the $H$-free process, and, furthermore, define the random variable 
\[
X^{j} = \sum_{1 \leq j' \leq j} X_{j'}   ,
\]
which counts the number of edges in the graph produced by the $H$-free process after traversing $e_1,\ldots,e_{j}$. 
Next, for every $1 \leq j \leq M$ we define
\[
Y_{j} =	\begin{cases}
		1, & \text{if $|O(X^{j-1})| < n^2/\lambda$,}\\
		X_{j}, & \text{otherwise,}
	\end{cases}
\qquad \text{ and } \qquad Y^j = \sum_{1 \leq j' \leq j} Y_{j'} .
\]
If $|O(X^{j-1})| \geq  n^{2}/\lambda$ holds, we have $Y_j = X_j$ by construction. In this case the next edge is added to the graph of the $H$-free process with probability at least $|O(X^{j-1})|/\binom{n}{2} \geq 2/\lambda$. 
Otherwise $Y_{j} = 1$ holds, and so we conclude that $\PP[Y_{j}=1 \mid Y_1, \ldots, Y_{j-1}] \geq 2/\lambda$, which implies that $Y^{M}$ stochastically dominates a binomial random variable with $M$ trials and success probability $2/\lambda$. 
With this in mind, standard Chernoff bounds, see e.g.~\eqref{eq:chernoff:lower} of Lemma~\ref{lem:chernoff}, give 
\begin{equation}
\label{eq:chernoff}
\PP[Y^{M} \leq 2i-t] \leq e^{-t^2/(4i)}   . 
\end{equation}

In the remainder we prove \eqref{eq:lem:transfer:uniform}. 
To this end first observe that 
\begin{equation}
\label{eq:lem:transfer:uniform:1}
\PP[G(i) \notin \cQ \text{ and } |O(i)| \geq n^{2}/\lambda] \leq \PP[G(i) \notin \cQ  \text{ and } X^{M} \geq i] + \PP[|O(i)| \geq  n^{2}/\lambda \text{ and }  X^{M} < i] . 
\end{equation}
Note that by construction $X^{M} \geq i$ implies $G(i) \subseteq G_{n,M}$, and, since $\cQ$ is a decreasing graph property, in this case $G(i) \notin \cQ$ implies $G_{n,M} \notin \cQ$. It follows that
\begin{equation*}
\label{eq:prob:subset}
\PP[G(i) \notin \cQ \text{ and } X^{M} \geq i] \leq \PP[G_{n,M} \notin \cQ] . 
\end{equation*}
Furthermore, since $O(i)$ is decreasing, if both $|O(i)| \geq n^{2}/\lambda$ and $X^{M} < i$ hold, then this implies $Y^{M} = X^{M} < i$. So, by \eqref{eq:chernoff} we have
\begin{equation*}
\label{eq:prob:badevent}
\PP[|O(i)| \geq n^{2}/\lambda \text{ and } X^{M} < i] \leq \PP[Y^{M} < i] \leq e^{-i/4} .
\end{equation*}
Substituting these bounds into \eqref{eq:lem:transfer:uniform:1} gives \eqref{eq:lem:transfer:uniform}, completing the proof. 
\end{proof} 
If we relax the additive error in Lemma~\ref{lem:transfer:uniform} to $o(1)$, then for $|O(i)| \geq \binom{n}{2}/\lambda$ a slight modification of the above proof works with $M=i\lambda +\omega(1) \lambda \sqrt{i}$; we leave these details to the interested reader.

\subsection{A `transfer theorem' for decreasing properties} 
Using Theorem~\ref{thm:BohmanKeevash2010H} and \eqref{eq:Cl-functions-estimates}, we see that $|O(m)| \geq n^{2-\epsilon/2}$ holds whp in the $C_{\ell}$-free process.  
So, setting $\lambda = \lambda(n) = n^{\epsilon/2}$ and using the `asymptotic equivalence' of the uniform and the binomial random graph for monotone graph properties (see e.g.\ Section~$1.4$ of~\cite{JLR2000RandomGraphs}), Lemma~\ref{lem:transfer:uniform} readily gives the next theorem. A similar idea is used in~\cite{Wolfovitz2010K4} for $H=K_4$. 
Observe that the edge-density of $G(m)$ is roughly $2pt_{\mathrm{max}} = \Theta(p (\log n)^{1/(\ell-1)})$ in the $C_{\ell}$-free process. 
Intuitively, the following theorem thus states that for decreasing properties, $G(m)$ is `comparable' with the binomial random graph with only slightly larger edge density $pn^{\epsilon}$. 
\begin{theorem}[`Transfer Theorem']%
\label{thm:transfer:binomial}
Define $m=m(n)$ and $p=p(n)$ as in \eqref{eq:Cl-parameters}. 
Suppose that $\epsilon$ is chosen as in \eqref{eq:Cl-constants:Wepsmu} and that $\cQ$ is a decreasing graph property. 
Then for the $C_{\ell}$-free process we have 
\begin{equation*}
\label{eq:thm:transfer:binomial}
\PP[G(m) \notin \cQ] \leq \PP[G_{n,pn^{\epsilon}} \notin \cQ] + o(1)   . 
\vspace{-2.0em}
\end{equation*}\qed
\end{theorem} 
In fact, this result also holds for the $H$-free process, where $H$ is strictly $2$-balanced, if $m$, $p$ and $\epsilon$ are chosen as in Sections~$1.2$ and~$1.3$ of~\cite{BohmanKeevash2010H}, since then $|O(m)| \geq n^{2-\epsilon/2}$, with room to spare. 
We believe that the above `transfer theorem' will significantly aid in the future analysis of the $H$-free process, since for decreasing properties it often allows us to work with the \emph{much} easier binomial random graph model, which has been extensively studied and for which e.g.\ sophisticated concentration inequalities are available.

\section{Properties of random graphs} 
\label{sec:binomial_results}
In this section we introduce several decreasing graph properties, which are key ingredients in our proof of Lemma~\ref{lem:dem:config}. 
Using the `transfer theorem' of Section~\ref{sec:transfer}, it suffices to prove that they hold whp for the binomial random graph $G_{n,p'}$ with $p' = p n^{\epsilon}$, where $p$ is defined as in \eqref{eq:Cl-parameters} and $\epsilon$ is chosen as in \eqref{eq:Cl-constants:Wepsmu}. 
We remark that essentially all results in this section are not best possible, but suffice for our purposes. 
For example, in an attempt to keep the formulas simple, we have not optimized the multiplicative $n^{\epsilon}$ factors involved (their contribution in our later arguments will be negligible).

\subsection{Basic properties} 
\label{sec:binomial_results:basic_properties}
\begin{lemma}%
\label{lem:bounded_codegree}%
Let $\cN$ denote the event that for all pairs of distinct vertices $x,y \in [n]$ we have $|\Gamma(x) \cap \Gamma(y)| \leq 9$. Then $\cN$ holds whp in $G_{n,p'}$.  
\end{lemma} 
\begin{proof}  
Using $\ell \geq 4$, \eqref{eq:Cl-constants:Wepsmu} and \eqref{eq:Cl-parameters}, i.e., $p=n^{-1+1/(\ell-1)} \leq n^{-2/3}$ and $\epsilon \leq 1/20$, we deduce that  
\[
\PP[\neg\cN] \leq \binom{n}{2} \binom{n-2}{10} (pn^{\epsilon})^{20} \leq n^2 (n p^2 n^{2\epsilon})^{10} \leq  n^2 (n^{-1/3+2\epsilon})^{10} = o(1)   , 
\]
as claimed. 
\end{proof} 
The following result states that every set of size at most $u$ contains a large independent subset. 
A similar argument was used by Bollob\'as and Riordan in~\cite{BollobasRiordan2000}.  
\begin{lemma}
\label{lem:set:independent:subset}
Let $\cI$ denote the event that for every $U \subseteq [n]$ with $|U| \leq u$ there exists an independent set $S \subseteq U$ with $|S| \geq  |U|/6$. 
Then $\cI$ holds whp in $G_{n,p'}$. 
\end{lemma} 
\begin{proof} 
Let $\cE$ denote the event that every $U \subseteq [n]$ with $|U| \leq u$ spans less than $3|U|$ edges. We have 
\begin{equation*}
\PP[\neg\cE] \leq \sum_{1 \leq x \leq u} \binom{n}{x} \binom{\binom{x}{2}}{3x} (pn^{\epsilon})^{3x} \leq \sum_{1 \leq x \leq u} \left(\frac{ne}{x}\right)^x \left(\frac{xe}{6}\right)^{3x} (pn^{\epsilon})^{3x} \leq \sum_{x \geq 1} \left(n u^2 p^3 n^{3\epsilon}\right)^x   . 
\end{equation*} 
Using $\ell \geq 4$, \eqref{eq:Cl-constants:Wepsmu}, \eqref{eq:Cl-parameters} and \eqref{eq:Cl-parameters2}, i.e., $u \leq np n^\epsilon$, $p=n^{-1+1/(\ell-1)} \leq n^{-2/3}$ and $\epsilon \leq 1/60$, we see that
\[
n u^2 p^3 n^{3\epsilon} \leq n^3 p^5 n^{5\epsilon} \leq n^{-1/3+5\epsilon} \leq n^{-1/4}   ,
\]
which implies $\PP[\neg\cE] = o(1)$. 
Suppose that $\cE$ holds. 
Then every set of at most $u$ vertices induces a graph with minimum degree less than six. 
Given $U \subseteq [n]$ with $|U| \leq u$, we set $W = U$. 
Now, by iteratively selecting a vertex $v \in W$ with at most five neighbours in $G[W]$ and removing $\{v\} \cup \Gamma(v)$ from $W$, we obtain an independent set with at least $|U|/6$ vertices, and the proof is complete. 
\end{proof}

\subsection{Bounding the numbers of certain paths}
\label{sec:binomial_results:paths}
The results in this section give estimates for the numbers of certain paths. 
Their statements will contain certain exceptions, and, as we shall see, many of these complications are in fact  necessary.

\subsubsection{Preliminaries: the size of certain neighbourhoods} 
\label{sec:binomial_results:paths:preliminaries}
The following crude upper bound on the degree of every vertex readily follows from standard Chernoff bounds (Lemma~\ref{lem:chernoff}) -- we omit the straightforward details. 
\begin{lemma}
\label{lem:set:neighbourhood:size} 
Let $\cD$ denote the event that for every $v \in [n]$ we have $|\Gamma(v)| \leq npn^{2\epsilon}$. 
Then $\cD$ holds whp in $G_{n,p'}$. \qed 
\end{lemma}
With similar reasoning it is also not difficult to see that whp for all large sets $S$, in $G_{n,p'}$ we have, say, $|\Gamma(S)| \geq |S| np$, which is much larger than $|S|$. 
Intuitively, the next lemma thus implies that for most reasonable sized $A \subseteq [n]$, 
only a small proportion of $\Gamma(S)$ is contained in $N^{(\leq \ell-3)}(A,S \cup A)$. 
\begin{lemma}
\label{lem:edges:bounded}
Let $\cM$ denote the event that for all disjoint $A,S \subseteq [n]$ with $|A|,|S| \leq kn^{5\epsilon}$ we have 
\begin{equation}
\label{eq:lem:edges:bounded}
e\bigl(S,\; N^{(\leq \ell-3)}(A,S \cup A)\bigr) \leq k n^{4\ell\epsilon}   . 
\end{equation}
Then $\cM$ holds whp in $G_{n,p'}$. 
\end{lemma}
\begin{proof} 
Let $\Psi$ contain all pairs $(A,S)$ with disjoint $A,S \subseteq [n]$ satisfying $|A|,|S| \leq kn^{5\epsilon}$.  
Given $\psi = (A,S) \in \Psi$, let $\cM_{\psi}$ denote the event that \eqref{eq:lem:edges:bounded} holds, and let $\cY_{\psi}$ contain all $Y \subseteq A \cup \bigcup_{1 \leq d \leq \ell-3} V_{d}(S \cup A)$ with $|Y| \leq (np n^{2\epsilon})^{\ell-2}n^{5\epsilon}$. 
Given $\psi = (A,S) \in \Psi$ and $Y \in \cY_{\psi}$, let $\cN_{\psi,Y}$ denote the event that $N^{(\leq \ell-3)}(A,S \cup A) = Y$. 
Using $k \leq npn^{\epsilon}$, it is not difficult to see that whenever $\cD$ holds, then for every $\psi \in \Psi$ some $\cN_{\psi,Y}$ with $Y \in \cY_{\psi}$ holds. 
Furthermore, $\neg\cM$ clearly implies that some $\cM_{\psi}$ with $\psi \in \Psi$ fails. 
So, we obtain 
\begin{equation*}
\PP[\neg \cM] \leq \PP[\neg \cD] + \sum_{\psi = (A,S) \in \Psi} \ \sum_{Y \in \cY_{\psi}} \PP[\neg \cM_{\psi} \cap  \cN_{\psi,Y}]   . 
\end{equation*} 
Note that for every $\psi = (A,S) \in \Psi$ the events $\cN_{\psi,Y}$ are mutually exclusive. 
So, using $|\Psi| \leq n^{2kn^{5\epsilon}}$ and that $\cD$ holds whp by Lemma~\ref{lem:set:neighbourhood:size}, to finish the proof it is enough to show that for every $\psi = (A,S) \in \Psi$ and $Y \in \cY_{\psi}$ we have 
\begin{equation}
\label{eq:lem:edges:bounded:prob}
\PP[\neg \cM_{\psi} \mid \cN_{\psi,Y}] \leq n^{-\omega(kn^{5\epsilon})}   .
\end{equation} 
Observe that we can find $Y=N^{(\leq \ell-3)}(A,S \cup A)$ by starting with $N^{(0)}(A,S \cup A) = A$, and then iteratively testing vertices in $V_{d}(S \cup A)$ to see whether they are adjacent to $N^{(d-1)}(A,S \cup A)$, up to $d=\ell-3$. 
Since $S$ is disjoint from $A$ and all $V_{d}(S \cup A)$ with $1 \leq d \leq \ell-3$, this exploration has not revealed any pairs between $S$ and $Y$. 
We deduce that, conditioned on $\cN_{\psi,Y}$, all edges between $S$ and $Y=N^{(\leq \ell-3)}(A,S \cup A)$ are included independently with probability $p'=pn^{\epsilon}$. 
Now, using $(np)^{\ell-2}=p^{-1}$ and $\ell \geq 4$, the expected number of these edges is bounded by 
\[
|S| \cdot |Y| \cdot p' \leq kn^{5\epsilon} \cdot (np n^{2\epsilon})^{\ell-2}n^{5\epsilon} \cdot p n^{\epsilon} = k n^{(2\ell+7)\epsilon} \leq k n^{(4\ell-1)\epsilon}   . 
\] 
Thus standard Chernoff bounds, see e.g.\ \eqref{eq:chernoff:upper:simple} of Lemma~\ref{lem:chernoff}, imply \eqref{eq:lem:edges:bounded:prob}, completing the proof. 
\end{proof}

\subsubsection{Paths ending in the neighbourhood of another set} 
\label{sec:binomial_results:paths:neighbourhood}
We start with a technical lemma, which will be used in the subsequent proofs of Lemmas~\ref{lem:path:endpoints} and~\ref{lem:path:endpoints:pairs}. 
\begin{lemma}
\label{lem:path:endpoints:shortest}
Let $\cQ_1$ denote the event that for all $v \in [n]$ and $A,X \subseteq [n]$ with $A \subseteq X$ and $|A|,|X| \leq k n^{5\ell\epsilon}$, for every $2 \leq j \leq \ell-1$ and $0 \leq d \leq \ell-3$ there are at most at most $(np)^{j-1} n^{9 \ell\epsilon}$ vertices $w \in N^{(\leq d)}(A,X)$ for which there exists a path 
\begin{equation}
\label{eq:lem:path:endpoints:shortest}
 v=w_0 \cdots w_{j}=w \qquad \text{ with } \qquad \{w_{0},\ldots,w_{j-1}\} \cap N^{(\leq d)}(A,X) = \emptyset   . 
\end{equation}
Then $\cQ_1$ holds whp in $G_{n,p'}$. 
\end{lemma}
\begin{proof}
Let $\Psi$ contain all tuples $(v,A,X,j,d)$ with $v \in [n]$, $A,X \subseteq [n]$, $2 \leq j \leq \ell-1$ and  $0 \leq d \leq \ell-3$ satisfying $A \subseteq X$ and $|A|,|X| \leq k n^{5\ell\epsilon}$.  
Given $\psi = (v,A,X,j,d) \in \Psi$, by $\cQ_{\psi}$ we denote the event that there are at most $(np)^{j-1} n^{9\ell\epsilon}$ vertices $w \in N^{(\leq d)}(A,X)$ for which there exists a path satisfying \eqref{eq:lem:path:endpoints:shortest}. 
Clearly, $\neg\cQ_1$ implies that some $\cQ_{\psi}$ with $\psi \in \Psi$ fails.

Next, given $\psi = (v,A,X,j,d) \in \Psi$, we denote by $\cY_\psi$  the set of pairs $(Y,Z)$ with $Y \subseteq A \cup \bigcup_{1 \leq d' \leq d} V_{d'}(X)$ and $Z \subseteq [n] \setminus Y$ satisfying $|Y| \leq (np n^{2\epsilon})^{d+1}n^{5\ell\epsilon}$ and $|Z| \leq (npn^{2\epsilon})^{j-1}$.  
Furthermore, for every $Y \subseteq [n]$ and $v \in [n]$ we inductively define 
\begin{equation}
\label{eq:def:gamma:vY}
\Gamma^{(0)}(v,Y) = \{v\} \setminus Y \qquad \text{ and } \qquad  \Gamma^{(i+1)}(v,Y) = \Gamma(\Gamma^{(i)}(v,Y))\setminus Y   .
\end{equation} 
Given $\psi = (v,A,X,j,d) \in \Psi$ and $\phi = (Y,Z) \in \cY_\psi$, let $\cN_{\psi,\phi}$ be the event that $N^{(\leq d)}(A,X) = Y$ and $\Gamma^{(j-1)}(v,Y) = Z$. 
Whenever $\cD$ holds, using $k \leq npn^{\epsilon}$ it is easy to see that for every $\psi \in \Psi$ some $\cN_{\psi,\phi}$ with $\phi \in \cY_\psi$ holds. 
Putting things together, we obtain 
\begin{equation*}
\PP[\neg\cQ_1] \leq \PP[\neg \cD] + \sum_{\psi = (v,A,X,j,d) \in \Psi} \ \sum_{\phi=(Y,Z) \in \cY_\psi} \PP[\neg \cQ_{\psi} \cap \cN_{\psi,\phi}]   . 
\end{equation*}
Since $\cD$ holds whp by Lemma~\ref{lem:set:neighbourhood:size}, using $|\Psi| \leq n^{3kn^{5\ell\epsilon}}$ and that for every $\psi \in \Psi$ the events $\cN_{\psi,\phi}$ are mutually exclusive, to complete the proof it suffices to show that for every $\psi = (v,A,X,j,d) \in \Psi$ and $\phi= (Y,Z) \in \cY_\psi$ we have 
\begin{equation}
\label{eq:lem:path:endpoints:shortest:prob}
\PP[\neg \cQ_{\psi} \mid \cN_{\psi,\phi}] \leq n^{-\omega(kn^{5\ell\epsilon})}   .
\end{equation} 
Recall that on $\cN_{\psi,\phi}$ we have $Y=N^{(\leq d)}(A,X)$ and $Z=\Gamma^{(j-1)}(v,Y)$. 
Every $w \in Y$ for which there exists a path satisfying \eqref{eq:lem:path:endpoints:shortest} is contained in $\Gamma(Z)$, and so whenever $\cQ_{\psi}$ fails we deduce $|\Gamma(Z) \cap Y| \geq (np)^{j-1} n^{9\ell\epsilon}$, which in turn implies 
\begin{equation}
\label{eq:lem:path:endpoints:shortest:prob:badevent}
e(Y,Z) \geq (np)^{j-1} n^{9\ell\epsilon}   . 
\end{equation} 
Next we analyse the distribution of the edges between $Y$ and $Z$ conditional on $\cN_{\psi,\phi}$. 
We can iteratively determine $Y=N^{(\leq d)}(A,X)$ as in the proof of Lemma~\ref{lem:edges:bounded}. 
Then, given $Y$, we can similarly find $Z = \Gamma^{(j-1)}(v,Y)$; by \eqref{eq:def:gamma:vY} this can clearly be done without testing any pairs between $Y$ and $Z$. 
It certainly can happen that during the first exploration, i.e., when determining $Y$, we have already revealed some pairs between $Y$ and $Z$, consider e.g.\ the case where $Z \cap V_1(X) \neq \emptyset$. However, by construction all such pairs are \emph{non-edges}. 
Therefore the number of edges between $Y$ and $Z$ is stochastically dominated by a binomial distribution with $|Y| \cdot |Z|$ trials and success probability $p'=pn^{\epsilon}$. 
Using $d \leq \ell-3$ and $j \leq \ell-1$ as well as $(np)^{d+1} \leq (np)^{\ell-2} = p^{-1}$, the expected value of the corresponding binomial random variable is at most 
\[
|Y| \cdot |Z| \cdot p' \leq (np n^{2\epsilon})^{d+1}n^{5\ell\epsilon} \cdot (np n^{2\epsilon})^{j-1} \cdot p n^{\epsilon} 
\leq (np)^{j-1} n^{(5\ell+2d+2j+1)\epsilon} \leq (np)^{j-1} n^{(9\ell-1)\epsilon}   . 
\]
So, since $j \geq 2$ and $k \leq npn^{\epsilon}$, standard Chernoff bounds show that \eqref{eq:lem:path:endpoints:shortest:prob:badevent} holds with probability at most $e^{-kn^{8\ell\epsilon}}$, see e.g.\ \eqref{eq:chernoff:upper:simple} of Lemma~\ref{lem:chernoff}. 
This establishes \eqref{eq:lem:path:endpoints:shortest:prob} and thus completes the proof. 
\end{proof}

Given a vertex $v \in [n]$, we expect that roughly $(np')^{\ell-2}$ vertices $w \in [n]$ are endpoints of a path $v=w_0 \cdots w_{\ell-2}=w$. 
Loosely speaking, the next lemma states that there are significantly fewer such vertices $w$ if we only count endpoints in a certain restricted set and forbid some exceptional paths. 
For the argument of Section~\ref{sec:good_configurations_exist} it is important to observe that $\cP_1$ is monotone decreasing. 
\begin{lemma}
\label{lem:path:endpoints}
Let $\cP_1$ denote the event that for all disjoint $A,S \subseteq [n]$ with $|A|,|S|\leq k$ there exists $X \subseteq [n]$ with $|X| \leq k n^{5\ell\epsilon}$, 
such that for every $v \in S$ there are at most $(np)^{\ell-3} n^{15\ell\epsilon}$ vertices $w \in N^{(\ell-3)}(A,X)$ for which there exists a path 
\begin{equation}
\label{eq:lem:path:endpoints}
 v=w_0 \cdots w_{\ell-2}=w \qquad \text{ with } \qquad w_1 \notin A   .
\end{equation}
Then $\cP_1$ holds whp in $G_{n,p'}$. 
\end{lemma}  
\begin{proof} 
By Lemmas~\ref{lem:set:neighbourhood:size}, \ref{lem:edges:bounded} and~\ref{lem:path:endpoints:shortest}  the event $\cD \cap \cM \cap \cQ_1$ holds whp. In the following we are going to argue that for every graph $G$ satisfying those properties, $\cP_1$ holds as well. 
As this claim is purely deterministic, it suffices to prove it for fixed disjoint $A,S \subseteq [n]$ with $|A|,|S|\leq k$. 
By $\cM$ there are at most $kn^{4\ell\epsilon}$ edges between $S$ and $N^{(\leq \ell-3)}(A,S \cup A)$. 
Let $V_{S,A}$ contain the endpoints of those edges and define 
\begin{equation}
\label{eq:lem:path:endpoints:def:X}
X= A \cup S \cup V_{S,A}   .
\end{equation}
Note that $|X| \leq kn^{5\ell\epsilon}$. 
Given $v \in [n]$, by $W_{v}$ we denote the set of $w \in N^{(\ell-3)}(A,X)$ for which there exists a path satisfying \eqref{eq:lem:path:endpoints}. 
To finish the proof, it suffices to show that for every $v \in S$ we have 
\begin{equation}
\label{eq:lem:path:endpoints:goal}
|W_{v}| \leq (np)^{\ell-3} n^{10\ell\epsilon}   .
\end{equation} 
Fix $v \in S \subseteq X$. 
Since $S \cap A = \emptyset$, for every path $v=w_0 \cdots w_{\ell-2}=w$ with  $w \in N^{(\ell-3)}(A,X)$ there exists $1 \leq j \leq \ell-2$ such that 
\begin{equation}
\label{eq:lem:path:endpoints:condition}
\{w_{0},\ldots,w_{j-1}\} \cap N^{(\leq \ell-3)}(A,X) = \emptyset \qquad \text{ and } \qquad w_j \in N^{(\leq \ell-3)}(A,X)   . 
\end{equation} 
Recall that by assumption $w_1 \notin A$. 
So, by \eqref{eq:neighbourhood:monotone} and \eqref{eq:lem:path:endpoints:def:X} we may restrict our attention to the case $j \geq 2$, since $S$ has no neighbours in $N^{(\leq \ell-3)}(A,X) \setminus A$. 
Now, as $\cQ_1$ holds, considering $d \leftarrow \ell-3$, for every $2 \leq j \leq \ell-2$ we deduce that there are at most $(np)^{j-1}n^{9\ell\epsilon}$ vertices $w_j \in N^{(\leq \ell-3)}(A,X)$ for which there exists a path  $v=w_0 \cdots w_j$ satisfying  \eqref{eq:lem:path:endpoints:condition}. 
Recall that the degree of every vertex is at most $npn^{2\epsilon}$ by $\cD$. 
So, given $ w_j$, there are at most $(npn^{2\epsilon})^{\ell-j-2}$ vertices $w \in N^{(\ell-3)}(A,X)$ for which there exists a path $w_j \cdots w_{\ell-2}=w$. 
Putting things together, we deduce that 
\[
|W_{v}| \leq \sum_{2 \leq j \leq \ell-2} (np)^{j-1} n^{9\ell\epsilon} \cdot (npn^{2\epsilon})^{\ell-j-2} \leq (np)^{\ell-3} n^{15 \ell\epsilon}    . 
\]
As explained, this implies $\cP_1$, and the proof is complete.  
\end{proof}
Note that in Lemma~\ref{lem:path:endpoints} a condition of the form $w_1 \notin A$ is necessary.   
Indeed, standard Chernoff bounds imply that whp every vertex has degree $\Omega(np')$. 
Furthermore, e.g.\ with a similar argument as in the proof of Lemma~$10.6$ in~\cite{Bollobas2001RandomGraphs}, one can show that whp for all choices of $A,S,X$, for all $Z \subseteq A$ with $|Z|\geq np$ we have, say, $|N^{(\ell-3)}(Z,X)| \geq  |Z| (np)^{\ell-3} \geq (np)^{\ell-2}$. 
So, by picking $A \in \binom{[n]}{k}$ such that it contains at least $np = o(k)$ neighbours of some vertex $v^*$, we have at least $(np)^{\ell-2}$ vertices $w \in N^{(\ell-3)}(A,X)$ which are endpoints of paths $v^*=w_0 \cdots w_{\ell-2}=w$ with $w_1 \in A$, violating the claimed bound.

\subsubsection{Paths connecting two sets} 
\label{sec:binomial_results:paths:sets}
Given $A,B,X \subseteq [n]$, for every $j \geq 1$ and $0 \leq d \leq \ell-3$, we say that $w_0 \cdots w_j = v_d \cdots v_0$ is a  \emph{$(j,d)$-path} wrt.\ $(A,B,X)$ if $v_0 \in A$, $w_0 \in B$ and $v_{d'} \in V_{d'}(X)$ for all $1 \leq d' \leq d$, cf.\ Figure~\ref{fig:jdpaths}. 
\begin{figure}[t]
	\centering
  \setlength{\unitlength}{1bp}%
  \begin{picture}(154.02, 88.51)(0,0)
  \put(0,0){\includegraphics{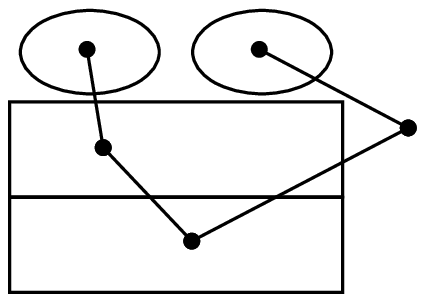}}
  \put(42.35,71.87){\fontsize{11.38}{13.66}\selectfont \makebox[0pt]{$A$}}
  \put(93.56,50.52){\fontsize{11.38}{13.66}\selectfont \makebox[0pt]{$V_1$}}
  \put(20.14,74.04){\fontsize{11.38}{13.66}\selectfont \makebox[0pt]{$v_0$}}
  \put(68.60,74.04){\fontsize{11.38}{13.66}\selectfont \makebox[0pt]{$w_0$}}
  \put(25.20,46.02){\fontsize{11.38}{13.66}\selectfont \makebox[0pt]{$v_1$}}
  \put(59.22,12.29){\fontsize{11.38}{13.66}\selectfont \makebox[0pt]{$v_2\!=\!w_2$}}
  \put(91.39,71.87){\fontsize{11.38}{13.66}\selectfont \makebox[0pt]{$B$}}
  \put(93.56,22.74){\fontsize{11.38}{13.66}\selectfont \makebox[0pt]{$V_2$}}
  \put(130.00,48.86){\fontsize{11.38}{13.66}\selectfont \makebox[0pt]{$w_1$}}
  \end{picture}%
\hspace{3.0em}
  \setlength{\unitlength}{1bp}%
  \begin{picture}(140.02, 88.51)(0,0)
  \put(0,0){\includegraphics{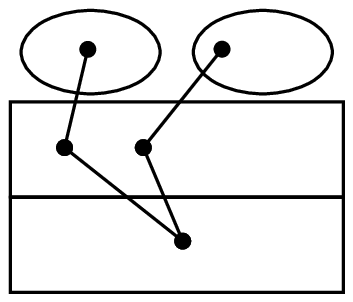}}
  \put(48.33,71.87){\fontsize{11.38}{13.66}\selectfont \makebox[0pt]{$A$}}
  \put(99.54,50.52){\fontsize{11.38}{13.66}\selectfont \makebox[0pt]{$V_1$}}
  \put(26.12,75.17){\fontsize{11.38}{13.66}\selectfont \makebox[0pt]{$v_0$}}
  \put(81.67,75.17){\fontsize{11.38}{13.66}\selectfont \makebox[0pt]{$w_0$}}
  \put(19.84,43.75){\fontsize{11.38}{13.66}\selectfont \makebox[0pt]{$v_1$}}
  \put(62.50,12.29){\fontsize{11.38}{13.66}\selectfont \makebox[0pt]{$v_2\!=\!w_2$}}
  \put(97.37,71.87){\fontsize{11.38}{13.66}\selectfont \makebox[0pt]{$B$}}
  \put(99.54,22.74){\fontsize{11.38}{13.66}\selectfont \makebox[0pt]{$V_2$}}
  \put(53.86,43.75){\fontsize{11.38}{13.66}\selectfont $w_1$}
  \end{picture}%
	\caption{\label{fig:jdpaths}Examples of $(2,2)$-paths for $\ell=5$. As usual, solid lines represent edges; for the other pairs there are no restrictions. Note that $w_1$ may be in $A \cup B$ or the vertex classes $V_1 \cup V_2$.} 
\end{figure}%
Intuitively, the next technical result states that the number of $(j,d)$-paths is not `too large' if we allow for deleting a few edges. 
\begin{lemma}
\label{lem:path:endpoints:pairs:deletion}
Let $\cQ_2$ denote the event that for all $A,B \subseteq [n]$ with $|A|,|B|\leq k$ there exists $F \subseteq \binom{[n]}{2}$ with $|F| \leq k n^{2\epsilon}$, such that for every $1 \leq j \leq \ell-1$ and $0 \leq d \leq \ell-4$ the number of $(j,d)$-paths wrt.\ $(A,B,A \cup B)$ that are edge disjoint from $F$ is bounded by $k^2(np)^{j-3} n^{4 \ell\epsilon}$. 
Then $\cQ_2$ holds whp in $G_{n,p'}$. 
\end{lemma}
\begin{proof} 
Fix $A,B \subseteq [n]$ with $|A|,|B|\leq k$. 
Given $j$ and $d$, we denote by $\cS_{j,d}=\cS_{j,d}(A,B)$ the family of edge-sets of all possible $(j,d)$-paths wrt.\ $(A,B,A \cup B)$. 
Clearly, $|V_{d'}(A \cup B)| \leq n$ for all $1 \leq d' \leq d$. 
So, using $p = (np)^{-(\ell-2)}$, $j \leq \ell-1$ and $d \leq \ell-4$, the expected number $\mu_{j,d}$ of such  $(j,d)$-paths satisfies  
\[
\mu_{j,d} \leq k^2 n^{j+d-1}(pn^{\epsilon})^{d+j} \leq k^2 (np)^{j+d-1} p n^{2\ell\epsilon} = k^2 (np)^{j+d+1-\ell} n^{2\ell\epsilon} \leq k^2 (np)^{j-3} n^{2\ell\epsilon}   . 
\]
Set $\kappa_j = k^2 (np)^{j-3} n^{3\ell\epsilon}$ and $b=kn^{\epsilon}$. 
Using the Deletion Lemma (cf.\ Lemma~\ref{lem:deletion_lemma}) the probability that $\cD\cL(b,\kappa_j,\cS_{j,d})$ fails for some $1 \leq j \leq \ell-1$ and $0 \leq d \leq \ell-4$ is bounded by 
\[
\sum_{1 \leq j \leq \ell} \ \sum_{0 \leq d \leq \ell-4} (1+\kappa_j/\mu_{j,d})^{-b} \leq \ell^2 \cdot n^{-\ell \epsilon b} = n^{-\omega(k)}   , 
\]
with room to spare. 
Whenever $\cD\cL(b,\kappa_j,\cS_{j,d})$ holds, we denote by $F_{j,d}$ the corresponding `ignored' edge set $E_0$ as in Lemma~\ref{lem:deletion_lemma}. 
If all $\cD\cL(b,\kappa_j,\cS_{j,d})$ with $1 \leq j \leq \ell-1$ and  $0 \leq d \leq \ell-4$ hold simultaneously, then defining $F$ as the union of all edge sets $F_{j,d}$ has the required properties. 
Finally, taking the union bound over all choices of $A$ and $B$ completes the proof. 
\end{proof}

For most large sets $B$ and $W$, we expect that the number of $(b,w) \in B \times W$ for which there exists a path $b=w_0 \cdots w_{\ell-2}=w$ should be roughly $|B||W|n^{\ell-3}p'^{\ell-2} = |B||W|n^{(\ell-2)\epsilon}/(np)$. 
Loosely speaking, the next lemma suggests that for most reasonable sized $A,B \subseteq [n]$, this upper bound holds for $W=N^{(\ell-4)}(A,X)$ if we forbid certain exceptional paths, as in this case $|W| \approx |A| (np')^{\ell-4}$. 
\begin{lemma}
\label{lem:path:endpoints:pairs}
Let $\cP_2$ denote the event that for all disjoint $A,B \subseteq [n]$ with $|A|,|B|\leq k$ there exists $X \subseteq [n]$ and $F \subseteq \binom{[n]}{2}$ with $|X| \leq k n^{5\ell\epsilon}$ and $|F| \leq k n^{2\epsilon}$, such that the number of pairs $(b,w) \in B \times N^{(\ell-4)}(A, X)$ for which there exists a path $b=w_0 \cdots w_{\ell-2}=w$ with 
\begin{equation}
\label{eq:lem:path:endpoints:pairs:condition} 
w_1 \notin A \qquad \text{ and } \qquad \bigl(w_2 \not\in A \ \text{ or } \ \{w_0w_1,w_1w_2\} \cap F = \emptyset\bigr) 
\end{equation} 
is at most $k^2(np)^{\ell-5} n^{15 \ell\epsilon}$. 
Then $\cP_2$ holds whp in $G_{n,p'}$. 
\end{lemma}
Before turning to the proof, note that $\cP_2$ is monotone decreasing. \marginal{Most technical proof of paper?}
\begin{proof}[Proof of Lemma~\ref{lem:path:endpoints:pairs}] 
By Lemmas~\ref{lem:set:neighbourhood:size}, \ref{lem:edges:bounded}, \ref{lem:path:endpoints:shortest} and~\ref{lem:path:endpoints:pairs:deletion} it is enough to show that $\cP_2$ holds for every graph $G$ satisfying $\cD \cap \cM \cap \cQ_1 \cap \cQ_2$. 
As this claim is purely deterministic, it suffices to prove it for fixed disjoint $A,B \subseteq [n]$ with $|A|,|B|\leq k$. 
Given $X \subseteq [n]$ and $F \subseteq \binom{[n]}{2}$, we denote by $P_{j,d}(X,F)$ the set of $(j,d)$-paths wrt.\ $(A,B,X)$ that are edge disjoint from $F$. 
By $\cQ_2$ there exists $F \subseteq \binom{[n]}{2}$ with $|F| \leq k n^{2\epsilon}$ such that for all $1 \leq j \leq \ell-2$ and $0 \leq d \leq \ell-4$ we have 
\begin{equation}
\label{eq:lem:path:endpoints:pairs:bound:jr:paths}
|P_{j,d}(A \cup B,F)| \leq k^2(np)^{j-3} n^{4 \ell\epsilon}   . 
\end{equation} 
Let $V_F$ contain all vertices outside $A$ that are endpoints of edges in $F$. 
Note that $|V_F| \leq 2 k n^{2\epsilon}$. 
Considering $S \leftarrow B \cup V_F$, by $\cM$ there are at most $kn^{4\ell\epsilon}$ edges between $B \cup V_F$ and $N^{(\leq \ell-3)}(A,B \cup V_F \cup A)$. 
Let $V_{B,F}$ contain the endpoints of all those edges and set 
\begin{equation}
\label{eq:lem:path:endpoints:pairs:def:X}
X=  A \cup B \cup V_F \cup V_{B,F}   . 
\end{equation} 
Observe that, say, $|X| \leq kn^{5\ell\epsilon}$. 
Furthermore, using \eqref{eq:neighbourhood:monotone} we see that 
\begin{equation}
\label{eq:lem:path:endpoints:pairs:properties:removal}
V_F \cap \bigcup_{1 \leq \kappa \leq \ell-4} V_{\kappa}(X) = \emptyset \qquad \text{ and } \qquad \Gamma\big(V_F \cup B\big) \cap \bigl(N^{(\leq \ell-4)}(A,X) \setminus A\bigr) = \emptyset   . 
\end{equation} 
For every $1 \leq j \leq \ell-2$ we define $W_j$ as the set of all pairs $(b,y) \in B \times N^{(\leq \ell-4)}(A,X)$ for which there exists a path $b=w_0 \cdots w_{j}=y$ satisfying \eqref{eq:lem:path:endpoints:pairs:condition} and  
\begin{equation}
\label{eq:lem:path:endpoints:pairs:condition:j}
\{w_{0},\ldots, w_{j-1}\} \cap N^{(\leq \ell-4)}(A,X) = \emptyset \qquad \text{ and } \qquad w_j \in N^{(\leq \ell-4)}(A,X)   . 
\end{equation}
We claim that in order to complete the proof, it suffices to show that for all $1 \leq j \leq \ell-2$ we have 
\begin{equation}
\label{eq:lem:path:endpoints:pairs:size:Wj}
|W_j| \leq k^2(np)^{j-3} n^{10\ell\epsilon}   . 
\end{equation}
Indeed, let $W$ contain all pairs $(b,w) \in B \times N^{(\ell-4)}(A,X)$ for which there exists a path $b=w_0 \cdots  w_{\ell-2}=w$ satisfying \eqref{eq:lem:path:endpoints:pairs:condition}. 
Note that for every such $b=w_0 \cdots w_{\ell-2}=w$ there exists $1 \leq j \leq \ell-2$ such that $b=w_0 \cdots w_{j}$ satisfies \eqref{eq:lem:path:endpoints:pairs:condition:j}. 
Recall that by $\cD$ the degree is bounded by $npn^{2\epsilon}$. 
So, given $w_j$, there are at most $(npn^{2\epsilon})^{\ell-j-2}$ vertices $w \in N^{(\ell-4)}(A,X)$ for which there exists a path $w_j \cdots w_{\ell-2}=w$.   
Putting things together, assuming \eqref{eq:lem:path:endpoints:pairs:size:Wj} we obtain  
\begin{equation*}
\label{eq:lem:path:endpoints:pairs:size:W}
|W| \leq \sum_{1 \leq j \leq \ell-2} |W_j| \cdot (npn^{2\epsilon})^{\ell-j-2} 
\leq k^2 (np)^{\ell-5} n^{15\ell\epsilon}   ,
\end{equation*}
and so $\cP_2$ holds, as claimed.

We shall now prove \eqref{eq:lem:path:endpoints:pairs:size:Wj}. 
Observe that for $j=1$ we need to consider paths $w_0w_1$ with $w_0 \in B$ and $w_1 \in N^{(\leq \ell-4)}(A,X) \setminus A$. 
Now, using the second part of \eqref{eq:lem:path:endpoints:pairs:properties:removal} we see that $w_1 \in \Gamma(w_0) \cap (N^{(\leq \ell-4)}(A,X) \setminus A)$ is impossible. 
This implies $|W_1|=0$, which clearly establishes  \eqref{eq:lem:path:endpoints:pairs:size:Wj} for $j=1$.

For $j \geq 2$ we first consider $W_{j,F} \subseteq W_j$, which contains all pairs $(b,y) \in W_j$ for which there exists a path $b=w_0 \cdots w_j=y$ satisfying \eqref{eq:lem:path:endpoints:pairs:condition:j} and 
\begin{equation}
\label{eq:lem:path:endpoints:pairs:condition:F:disjoint}
\{w_0w_1,\ldots, w_{j-1}w_{j}\} \cap F = \emptyset   . 
\end{equation}
Clearly, for every $(b,y) \in W_{j,F}$ there exists $0 \leq d \leq \ell-4$ such that at least one $(j,d)$-path wrt.\ $(A,B,X)$ with $b=w_0$ and $w_{j}=y$ satisfies \eqref{eq:lem:path:endpoints:pairs:condition:F:disjoint}. 
We claim that the corresponding $(j,d)$-path $w_0 \cdots w_j = v_d \cdots v_0$ is edge-disjoint from $F$, i.e., contained in $P_{j,d}(X,F)$. 
To see this, observe that every $f \in \{v_{d}v_{d-1}, \cdots, v_1v_0\} \cap F$ has at least one vertex outside of $A$, say $v_{\kappa} \in V_{\kappa}(X)$ with $1 \leq \kappa \leq d$, which contradicts \eqref{eq:lem:path:endpoints:pairs:properties:removal}, since by construction $v_{\kappa} \in V_F$. 
In addition, by \eqref{eq:neighbourhood:monotone} and \eqref{eq:lem:path:endpoints:pairs:def:X} we see that $P_{j,d}(X,F) \subseteq P_{j,d}(A \cup B,F)$. 
Putting things together, using \eqref{eq:lem:path:endpoints:pairs:bound:jr:paths} our discussion yields
\begin{equation}
\label{eq:lem:path:endpoints:pairs:bound:WjF:disjoint}
|W_{j,F}| \leq \sum_{0 \leq d \leq \ell-4} |P_{j,d}(X,F)| \leq \sum_{0 \leq d \leq \ell-4} |P_{j,d}(A \cup B,F)| \leq k^2(np)^{j-3} n^{5 \ell\epsilon}   . 
\end{equation}

It remains to estimate the number of pairs in $W^{*}_{j,F} = W_{j} \setminus W_{j,F}$, where the corresponding paths intersect with $F$. 
We start with the special case $j=2$, i.e., paths $b=w_0w_1w_2=y$ with $(b,y) \in W^{*}_{2,F}$ satisfying \eqref{eq:lem:path:endpoints:pairs:condition}. 
Observe that every $f \in \{w_0w_1,w_1w_2\} \cap F$ contains $w_1 \in V_F$, since  $w_1 \notin A$ by \eqref{eq:lem:path:endpoints:pairs:condition}. 
Note that $w_2 \in A$ contradicts the second part of \eqref{eq:lem:path:endpoints:pairs:condition}, and that $w_2 \in \Gamma(w_1) \cap (N^{(\leq \ell-4)}(A,X) \setminus A)$ is impossible by \eqref{eq:lem:path:endpoints:pairs:properties:removal}. 
To sum up, $|W^{*}_{2,F}|=0$, which together with \eqref{eq:lem:path:endpoints:pairs:bound:WjF:disjoint} implies \eqref{eq:lem:path:endpoints:pairs:size:Wj} for $j=2$.

Turning to $j \geq 3$, for every $1 \leq \varsigma \leq j$ we denote by $W^{*}_{j,F,\varsigma} \subseteq W^{*}_{j,F}$ the set of pairs $(b,y) \in W^{*}_{j,F}$ with $y \notin A$ where the corresponding path $b=w_0 \cdots w_j =y$ satisfies $w_{\varsigma-1}w_{\varsigma} \in F$ and \eqref{eq:lem:path:endpoints:pairs:condition:j}. 
We claim that it is enough to show that for every $1 \leq \varsigma \leq j$ we have 
\begin{equation}
\label{eq:lem:path:endpoints:pairs:bound:WjF:intersection:notinA}
|W^{*}_{j,F,\varsigma}| \leq k^2(np)^{j-3}n^{8\ell\epsilon}   . 
\end{equation} 
Indeed, since there are at most $|B| \cdot |A| \leq k^2 \leq k^2(np)^{j-3}$ pairs $(b,y) \in W^{*}_{j,F}$ with $y \in A$, we obtain  
\[
|W^{*}_{j,F}| \leq k^2(np)^{j-3} + \sum_{1 \leq \varsigma \leq j} |W^{*}_{j,F,\varsigma}|  \leq k^2(np)^{j-3}n^{9\ell\epsilon}   , 
\]
which together with \eqref{eq:lem:path:endpoints:pairs:bound:WjF:disjoint} establishes \eqref{eq:lem:path:endpoints:pairs:size:Wj}, as claimed.

In the following we verify \eqref{eq:lem:path:endpoints:pairs:bound:WjF:intersection:notinA}. 
First we show that $|W^{*}_{j,F,\varsigma}| = 0$ for $\varsigma \in \{j-1,j\}$. 
If $w_{j-1}w_{j} \in F$, then $w_j \notin A$ implies $w_j \in V_F$, but the remaining possibility $w_j \in N^{(\leq \ell-4)}(A,X) \setminus A$ contradicts \eqref{eq:lem:path:endpoints:pairs:properties:removal}. 
If $w_{j-2}w_{j-1} \in F$, then by \eqref{eq:lem:path:endpoints:pairs:condition:j} we have $w_{j-1} \notin N^{(\leq \ell-4)}(A,X)$ and so $w_{j-1} \in V_F$. 
Since by assumption $w_j \notin A$ we must have $w_j \in \Gamma(w_{j-1}) \cap (N^{(\leq \ell-4)}(A,X) \setminus A)$, which is impossible by \eqref{eq:lem:path:endpoints:pairs:properties:removal}.

Now, suppose that $w_{\varsigma-1}w_{\varsigma} \in F$ with $1 \leq \varsigma \leq j-2$.  
Considering $v \leftarrow w_{\varsigma}$ and $d \leftarrow \ell-4$, by $\cQ_1$ there are at most $(np)^{j-\varsigma-1}n^{6\ell\epsilon}$ vertices $w_j \in N^{(\leq \ell-4)}(A,X)$ for which there exists a path $w_{\varsigma}=w'_0 \cdots w'_{j-\varsigma}=w_j$ with $\{w_{\varsigma}, \ldots, w_{j-1}\} \cap N^{(\leq \ell-4)}(A,X) = \emptyset$. 
So, using $|F| \leq kn^{2\epsilon}$, since there are at most $|B|=k$ choices for $b \in B$, for $\varsigma \geq 2$ we deduce that  
\[
|W^{*}_{j,F,\varsigma}| \leq |B| \cdot 2|F| \cdot (np)^{j-\varsigma-1}n^{6\ell\epsilon} \leq k^2(np)^{j-\varsigma-1}n^{(6\ell+3)\epsilon} \leq k^2(np)^{j-3}n^{8\ell\epsilon}   , 
\]
as claimed. 
Note that for the remaining case $\varsigma=1$ each (ordered) edge $w_0w_1 \in F$ also determines the vertex $b=w_0 \in B$. So, compared to the estimate above we win a factor of $|B|$, and a virtually identical calculation yields that  \eqref{eq:lem:path:endpoints:pairs:bound:WjF:intersection:notinA} also holds in this case, which completes the proof. 
\end{proof} 
With very similar reasoning as for Lemma~\ref{lem:path:endpoints}, one can argue that an extra condition for the case $w_2 \in A$ is needed in Lemma~\ref{lem:path:endpoints:pairs}: this time we can otherwise violate the claimed bound whp by fixing some vertex $v^*$ and then choosing disjoint $A,B \subseteq [n]$ such that each contains at least $np$ vertices from $\Gamma(v^*)$; we leave the  details to the interested reader.

\section{Very good configurations exist} 
\label{sec:good_configurations_exist}
In this section we prove Lemma~\ref{lem:dem:config}. 
Given a graph property $\cY$, let $\cY_i$ denote the event $G(i) \in \cY$, i.e., that $G(i)$ satisfies $\cY$. 
Now, for every $0 \leq i \leq m$ we set 
\[
\cW_i = \cI_i \cap \cK_i \cap \cL_i \cap \cN_i \cap \cP_{1,i} \cap  \cP_{2,i} \cap \cT_i,
\]
where $\cK_i$, $\cL_i$, $\cT_i$ are defined as in Theorem~\ref{thm:BohmanKeevash2010H} and Lemma~\ref{lem:edges_bounded_and_large_degree_bounded}, and $\cI$, $\cN$, $\cP_{1}$, $\cP_{2}$ are defined as in Lemmas~\ref{lem:bounded_codegree}, \ref{lem:set:independent:subset}, \ref{lem:path:endpoints} and~\ref{lem:path:endpoints:pairs}. 
It is not difficult to see that $\cW_i$ is monotone decreasing and, using the `transfer theorem' (Theorem~\ref{thm:transfer:binomial}), that $\cW_m$ holds whp. 
Observe that by monotonicity $\cW_m$ implies $\cW_i$ for every $i \leq m$, and that $\neg\badE$ implies $\neg\badE[i-1]$. 
So, to complete the proof it suffices to consider fixed $G(i)$ satisfying $\cW_i$ and show that for every $(\tilde{v},U)$ with $U \in \binom{[n] \setminus \{\tilde{v}\}}{u}$ there exists $\Sigma^*=(\tilde{v},U,A,B,R) \in \cC$ satisfying $\neg\badEC$ and \eqref{eq:lem:dem:config:ignored}. 
In fact, since the above claim is purely deterministic, it is enough to also consider fixed $(\tilde{v},U)$. 
Our proof proceeds in several steps and we tacitly assume that $n$ is sufficiently large whenever necessary. 
First, in Section~\ref{sec:good_configurations_exist:find_config} we choose a `special' configuration $\Sigma^*=(\tilde{v},U,A,B,R)$ and collect some of its basic properties. 
In the remaining sections we verify that $\Sigma^*$ has the properties claimed by Lemma~\ref{lem:dem:config}. 
More precisely, in Section~\ref{sec:good_configuration} we show that $\neg\badEC$ holds, and in Section~\ref{sec:good_configuration:few_ignored} we establish \eqref{eq:lem:dem:config:ignored}.

\subsection{Finding $\Sigma^*=(\tilde{v},U,A,B,R)$} 
\label{sec:good_configurations_exist:find_config}
In the following we show how we pick $\Sigma^*=(\tilde{v},U,A,B,R)$. 
Along the way, we furthermore collect some immediate properties of the resulting $\Sigma^*$. 
We set 
\begin{equation}
\label{eq:Cl-parameters4}
\tau =  40\ell \qquad \text{ and } \qquad \vartheta = 20 \ell \tau = 800 \ell^2  . 
\end{equation} 
For the main steps of our argument it is useful to keep in mind that $\vartheta \gg \tau \gg \ell$ and $\vartheta\epsilon \ll 1/\ell$. 
First, we choose $S \subseteq U$ such that
\begin{equation}
\label{eq:size:S}
\text{$S$ is an independent set} \qquad \text{ and } \qquad |S| \geq u/6   , 
\end{equation} 
which is possible since $\cI_i$ holds. 
Henceforth we assume that $v_1, \ldots, v_n \in [n]$ are ordered so that 
\begin{equation}
\label{eq:vertex_ordering}
|\Gamma(v_1) \cap S| \geq |\Gamma(v_2) \cap S| \geq \cdots \geq  |\Gamma(v_j) \cap S| \geq \cdots \geq  |\Gamma(v_n) \cap S|   .
\end{equation}
We greedily choose first $\ell_A$, and afterwards $\ell_B$, such that they are the smallest indices for which 
\[
N_A = \bigcup_{1 \leq j \leq \ell_A}  \big(\Gamma(v_j) \cap S\big) \qquad \text{ and } \qquad   N_B  = \bigcup_{\ell_A < j \leq \ell_B}  \big(\Gamma(v_j) \cap S\big) \setminus N_A
\] 
each have cardinality at least $2k$, where we set the corresponding index to $\infty$ if this is not possible. 
Recall that $k=\gamma/60 \cdot np t_{\max}$ by \eqref{eq:Cl-parameters3} and $\gamma \geq 180$ by \eqref{eq:Cl-parameters2}.  
So, since $\cT_i$ holds, by \eqref{eq:degree-estimate} the maximum degree is at most $3 np t_{\max} \leq k$. 
Using $k=u/60$, we deduce that  
\begin{equation}
\label{eq:size_neighbourhoods}
|N_A \cup N_B| \leq 6k \leq u/10   .
\end{equation}

\subsubsection{Picking $A,B$} 
If $\ell_B=\infty$ or $\ell_B > n^{2\vartheta\epsilon}$, we choose arbitrary disjoint sets, each of size $k=u/60$, satisfying  
\[
A,B \subseteq S \setminus (N_A \cup N_B)   ,
\]
which is possible by \eqref{eq:size:S} and \eqref{eq:size_neighbourhoods}. 
For later usage, we furthermore set $I_A = \emptyset$ and $I_B = \emptyset$.

If $\ell_B \leq n^{2\vartheta\epsilon} = o(k)$, we set $I_A = \{v_1, \ldots, v_{\ell_A}\}$ and $I_B = \{v_{\ell_A+1}, \ldots, v_{\ell_B}\}$. 
Since $G(i)$ satisfies $\cN_i$, the codegrees are all bounded by nine, and thus 
\begin{equation}
\label{eq:size_common_neighbourhoods}
|\Gamma(I_B) \cap N_A| \leq |\Gamma(I_B) \cap \Gamma(I_A)| \leq 9 \cdot \ell_B \cdot \ell_A \leq 9 n^{4\vartheta\epsilon} = o(k)   .
\end{equation} 
Now we choose arbitrary sets, each of size $k$, satisfying 
\[
A \subseteq N_A \setminus \bigl(I_B \cup \Gamma(I_B)\bigr) \qquad \text{ and } \qquad B \subseteq N_B   ,
\]
which is possible by \eqref{eq:size_common_neighbourhoods}. 
Clearly, $A$ and $B \cup I_B$ are disjoint.

Next we estimate the size of certain neighbourhoods. A similar argument can be found in~\cite{Warnke2010K4}. 
\begin{lemma}
\label{lem:size:neighbourhoods}
We have 
$\Gamma(I_A) \cap B = \emptyset$ and $\Gamma(I_B) \cap A = \emptyset$. 
Given $Y \in \{A,B\}$, every $v \notin I_Y$ satisfies 
\begin{equation}
\label{eq:vnotinIAB:neighbourhood:bound}
|\Gamma(v) \cap Y| \leq np n^{-\vartheta\epsilon}   .
\end{equation}
\end{lemma}
\begin{proof}
If $\ell_B=\infty$, then all vertices $v \in [n]$ satisfy the stronger bound $|\Gamma(v) \cap (A \cup B)| = 0$.

Next, we consider the case $n^{2\vartheta \epsilon} < \ell_B < \infty$, where $I_A = I_B = \emptyset$.  
Since all vertices $v \in \{v_{1}, \ldots, v_{\ell_B}\}$ satisfy $|\Gamma(v) \cap (A \cup B)| = 0$, using \eqref{eq:vertex_ordering} it is not difficult to see that in order to prove \eqref{eq:vnotinIAB:neighbourhood:bound}, it suffices to show $|\Gamma(v_{x}) \cap S| \leq np n^{-\vartheta\epsilon}$ for $x = n^{2\vartheta \epsilon}$.  
Set $H = \{v_{1}, \ldots, v_{x}\}$. On the one hand, using \eqref{eq:vertex_ordering} we have $2 e(H,S) \geq x |\Gamma(v_{x}) \cap S|$. 
On the other hand, since $G(i)$ satisfies $\cK_i$, using $|H| = n^{2\vartheta \epsilon}$ and $|S|\leq npn^{\epsilon}$, we have, say, $e(H,S) \leq npn^{2\epsilon}$. 
So, we deduce $|\Gamma(v_{x}) \cap S| \leq np n^{-\vartheta\epsilon}$, as claimed.

Finally, suppose that $\ell_B \leq n^{2\vartheta \epsilon}$.
Observe that $\Gamma(I_A) \cap B = \emptyset$ and $\Gamma(I_B) \cap A = \emptyset$ hold by construction. 
Fix $Y \in \{A,B\}$. 
Since by $\cN_i$ all codegrees are at most nine, for  every $v \notin I_Y$ we have $|\Gamma(v) \cap Y| \leq |\Gamma(v) \cap \Gamma(I_Y)| \leq 9 \ell_B$, which readily establishes \eqref{eq:vnotinIAB:neighbourhood:bound}, and thus completes the proof. 
\end{proof}

\subsubsection{Choosing $R$} 
Observe that $|I_B| \leq n^{2\vartheta\epsilon}$. 
Considering $A$ and $S \leftarrow I_B$, we denote by $X_1$ the set $X$ whose existence is guaranteed by $\cP_{1,i}$.  
Similarly, let $X_2$ and $F$ denote the sets $X$ and $F$ whose existence is guaranteed by $\cP_{2,i}$ when considering $A$ and $B$. 
We have $|X_1|, |X_2| \leq k n^{5\ell\epsilon}$ and $|F| \leq k n^{2\epsilon}$. 
Now we set 
\begin{equation}
\label{eq:def:R} 
R = \{\tilde{v}\} \cup U \cup X_1 \cup X_2   .
\end{equation} 
Clearly, $|R| \leq k n^{10\ell\epsilon}$ holds, with room to spare. 
Next we collect several structural properties. 
By \eqref{eq:neighbourhood:monotone} and \eqref{eq:def:R} and have $N^{(j)}(A,R) \subseteq N^{(j)}(A,X_1) \cap N^{(j)}(A,X_2)$. 
So, using $(\Gamma(I_B) \cup I_B) \cap A = \emptyset$, we immediately obtain the following statement: 
\begin{lemma}
\label{lem:bound_paths_IB}
We have $|I_B| \leq n^{2\vartheta\epsilon}$, and for every $v \in I_B$ there are at most $(np)^{\ell-3} n^{15 \ell\epsilon}$ vertices $w \in N^{(\ell-3)}(A,R)$ for which there exists a path $v=w_0 \cdots w_{\ell-2}=w$. \qed 
\end{lemma}
In addition, using that $A \cup B$ is an independent set, we readily deduce the following result: 
\begin{lemma}
\label{lem:bound_pairs:BW}
We have $|F| \leq k n^{2\epsilon}$, and there are at most $k^2(np)^{\ell-5} n^{15 \ell\epsilon}$ pairs $(b,w) \in B \times N^{(\ell-4)}(A,R)$ for which there exists a path $b=w_0 \cdots w_{\ell-2}=w$ satisfying $w_2 \not\in A$ or $\{w_0w_1,w_1w_2\} \cap F = \emptyset$. \qed 
\end{lemma} 
In the subsequent sections, the construction of $A$ and $B$ is irrelevant; all that we use is that $A$, $B$ are disjoint subsets of $U$ with size $k$,  and there are sets $F$, $I_A$, $I_B$, $R$ such that the conclusions of Lemmas~\ref{lem:size:neighbourhoods}--\ref{lem:bound_pairs:BW} hold in $G(i)$.

\subsection{The configuration $\Sigma^*$ is good} 
\label{sec:good_configuration} 
In this section we show that $\neg\badEC = \neg\badEC[1,i] \cap \neg\badEC[2,i]$ holds.

\subsubsection{The bad event $\badEC[1,i]$} 
\label{sec:good_configuration:bad:2} 
In order to prove that $\badEC[1,i]$ fails, using Lemma~\ref{lem:bound_pairs:BW} it suffices to show that there are at most $k^2(np)^{\ell-4}n^{-10\epsilon}$ paths $w_0 \cdots w_{\ell-2}$ with $(w_0,w_2) \in B \times A$ satisfying $w_0w_1 \in F$ or $w_1w_2 \in F$. 
Let $P_{\Sigma^*}$ denote all such paths. 
For every $w_0w_1 \in F \cap E(i)$ with $w_0 \in B$, using Lemma~\ref{lem:size:neighbourhoods} we see that $w_1 \notin I_A$, which by \eqref{eq:vnotinIAB:neighbourhood:bound} implies that there are at most $np n^{-\vartheta\epsilon}$ choices for $w_2 \in \Gamma(w_1) \cap A$. 
With a similar argument, for every $w_1w_2 \in F \cap E(i)$ with $w_2 \in A$ we have at most $np n^{-\vartheta\epsilon}$ choices for $w_0 \in \Gamma(w_1) \cap B$. 
Furthermore, since the degree is bounded by $npn^{\epsilon}$, given $w_2 \in A$ there are at most $(npn^{\epsilon})^{\ell-4}$ paths $w_2 \cdots w_{\ell-2}$. 
So, using $np \leq k$, $|F| \leq k n^{2\epsilon}$ and \eqref{eq:Cl-parameters4}, i.e., $\vartheta \geq 20 \ell$, we deduce that  
\[
|P_{\Sigma^*}| \leq np n^{-\vartheta\epsilon} \cdot 2|F| \cdot (npn^{\epsilon})^{\ell-4} \leq k^2 (np)^{\ell-4} n^{(\ell-\vartheta)\epsilon} < k^2 (np)^{\ell-4} n^{-10\epsilon}   ,
\] 
which, as explained, establishes $\neg\badEC[1,i]$.

\subsubsection{The bad event $\badEC[2,i]$} 
\label{sec:good_configuration:bad:1}  
In anticipation of the estimates in Section~\ref{sec:good_configuration:few_ignored}, here we analyse the combinatorial structure of $L_{\Sigma^*}(i)$ much more precisely than needed. 
To this end we introduce the sets $L_{\Sigma^*}(i,j)$, where for every $j \in [\ell-1]$ we denote by $L_{\Sigma^*}(i,j)$ the set of all \emph{ordered} pairs $xy$ with distinct $x,y \in [n]$ such that $|C_{x,y,\Sigma^*}(i,j)| \geq p^{-1}n^{-30\ell\epsilon}$.  
We start by showing that we may restrict our attention to the case $j \in \{1,2\}$. 
Recall that $C_{x,y,\Sigma^*}(i,j)$ contains all pairs $bw \in B \times N^{(\ell-3)}(A,R)$ for which there exist disjoint paths $b=w_1 \cdots w_{j}=x$ and $y=w_{j+1} \cdots w_{\ell}=w$ in $G(i)$.  
Fix $x \neq y$. 
Since the degree is at most $npn^{\epsilon}$ by \eqref{eq:degree-estimate}, for $j \geq 3$ the number of choices for $w$ is at most $(npn^{\epsilon})^{\ell-j-1} \leq (npn^{\epsilon})^{\ell-4}$. 
Now, as there are at most $|B| \leq k \leq npn^{\epsilon}$ ways to pick $b \in B$, using $(np)^{\ell-2}=p^{-1}$ we crudely have 
\begin{equation}
\label{eq:closed:xySigma:small:jl}
|C_{x,y,\Sigma^*}(i,j)| \leq npn^{\epsilon} \cdot (npn^{\epsilon})^{\ell-4} \leq p^{-1} n^{\ell\epsilon}/(np) < p^{-1} n^{-30\ell\epsilon}   ,
\end{equation}
which implies $xy \notin L_{\Sigma^*}(i,j)$. Therefore $L_{\Sigma^*}(i,j) = \emptyset$ for $j \geq 3$, so 
\begin{equation}
\label{eq:closed:xySigma:inequality}
|L_{\Sigma^*}(i)| \leq |L_{\Sigma^*}(i,1)| + |L_{\Sigma^*}(i,2)|   .
\end{equation}
With foresight, for all $j \geq 1$ we define $M^{(j)}(A)$ as the set of $v \in [n]$ with $|W^{(j)}(v,A)| \geq (np)^{j} n^{-\tau\epsilon}$, where $W^{(j)}(v,A)$ contains all vertices $w \in N^{(\ell-3)}(A,R)$ for which there exists a path $v=w_0 \cdots w_{j}=w$ in $G(i)$. 
Now we claim that 
\begin{equation}
\label{eq:closed:xySigma:i2:inclusion}
L_{\Sigma^*}(i,2) \subseteq \left\{xy \;:\; x \in I_B  \; \wedge \;  y \in M^{(\ell-3)}(A) \right\}   . 
\end{equation}
Note that $C_{x,y,\Sigma^*}(i,2)$ contains only pairs $bw \in B \times N^{(\ell-3)}(A,R)$ for which there exists paths $b=w_1w_{2}=x$ and $y=w_{3} \cdots w_{\ell}=w$ in $G(i)$. 
First suppose that $x \notin I_B$. 
Using Lemma~\ref{lem:size:neighbourhoods}, by \eqref{eq:vnotinIAB:neighbourhood:bound} we have at most $npn^{-\vartheta\epsilon}$ choices for $b \in \Gamma(x) \cap B$. 
Since the degree is at most $npn^{\epsilon}$, we have at most $(npn^{\epsilon})^{\ell-3}$ choices for $w$. 
So, using $(np)^{\ell-2}=p^{-1}$ and \eqref{eq:Cl-parameters4}, i.e., $\vartheta \geq 40 \ell$, we deduce  that 
\[
|C_{x,y,\Sigma^*}(i,2)| \leq npn^{-\vartheta\epsilon} \cdot (npn^{\epsilon})^{\ell-3} \leq p^{-1} n^{(\ell-\vartheta)\epsilon} < p^{-1} n^{-30\ell\epsilon}  ,
\]
which implies $xy \notin L_{\Sigma^*}(i,2)$. 
Next, we consider the case where $y \notin M^{(\ell-3)}(A)$. 
With a very similar reasoning as above, this time using $|W^{(\ell-3)}(y,A)| \leq (np)^{\ell-3} n^{-\tau\epsilon}$ and 
\eqref{eq:Cl-parameters4}, i.e., $\tau = 40 \ell$,  we obtain 
\[
|C_{x,y,\Sigma^*}(i,2)| \leq npn^{\epsilon} \cdot (np)^{\ell-3} n^{-\tau\epsilon} \leq p^{-1} n^{(1-\tau)\epsilon} < p^{-1} n^{-30\ell\epsilon}   ,
\]
which implies $xy \notin L_{\Sigma^*}(i,2)$. This completes the proof of \eqref{eq:closed:xySigma:i2:inclusion}.

By a similar but simpler argument we furthermore see that 
\begin{equation}
\label{eq:closed:xySigma:i1:inclusion}
L_{\Sigma^*}(i,1) \subseteq \left\{xy \;:\; x \in B  \; \wedge \;  y \in M^{(\ell-2)}(A) \right\}   . 
\end{equation}

Next we estimate the cardinality of $M^{(j)}(A)$. A similar argument is implicit in~\cite{BohmanKeevash2010H}. 
\begin{lemma}
\label{lem:bound:manypaths}
For every $1 \leq j \leq \ell-2$ we have $|M^{(j)}(A)| \leq (np)^{\ell-2-j}n^{2\ell\tau\epsilon}$. 
\end{lemma}
\begin{proof}
Set $H^{(0)}(A) = N^{(\ell-3)}(A,R)$, and for every $j \geq 1$ we let $H^{(j)}(A)$ contain all $v \in [n]$ with $|\Gamma(v) \cap H^{(j-1)}(A) | \geq np n^{-2\tau\epsilon}$. 
First, we claim that for all $1 \leq j \leq \ell-2$ we have
\begin{equation}
\label{eq:lem:bound:manypaths:HP:inclusion}
M^{(j)}(A) \subseteq H^{(j)}(A)   . 
\end{equation}
Since $\tau \geq 2\ell$ by \eqref{eq:Cl-parameters4}, it clearly suffices to show that for all $1 \leq j \leq \ell-2$, for every $v \notin H^{(j)}(A)$ we have $|W^{(j)}(v,A)| \leq j (npn^{\epsilon})^{j} n^{-2\tau\epsilon}$. 
We proceed by induction on $j$. For the base case $j=1$ the claim is trivial, since $H^{(1)}(A)$ contains all vertices $v \in [n]$ with $|\Gamma(v) \cap N^{(\ell-3)}(A,R)| \geq np n^{-2\tau\epsilon}$. 
Turning to $j \geq 2$, fix $v \notin H^{(j)}(A)$. 
By distinguishing between the neighbours of $v$ inside and outside of $H^{(j-1)}(A)$, using the induction hypothesis and that the degree is bounded by $npn^{\epsilon}$, we obtain  
\[
|W^{(j)}(v,A)| \leq np n^{-2\tau\epsilon} \cdot (npn^{\epsilon})^{j-1} + np n^{\epsilon} \cdot (j-1) (npn^{\epsilon})^{j-1} n^{-2\tau\epsilon} \leq j (npn^{\epsilon})^{j} n^{-2\tau\epsilon}   , 
\]
which, as explained, establishes \eqref{eq:lem:bound:manypaths:HP:inclusion}.

To finish the proof, again using $\tau \geq 2\ell$, it suffices to show that for all $0 \leq j \leq \ell-2$ we have 
\begin{equation}
\label{eq:lem:bound:manypaths:H:bound}
|H^{(j)}(A)| \leq (np)^{\ell-2-j}n^{(2j\tau+\ell+j)\epsilon}  . 
\end{equation} 
As before, we proceed by induction on $j$. 
Using $|A| \leq k \leq npn^{\epsilon}$ and that the degree is bounded by $npn^{\epsilon}$, we establish the base case $j=0$ by observing that $|H^{(0)}(A)| \leq |\Gamma^{(\ell-3)}(A)| \leq (npn^{\epsilon})^{\ell-2}$. 
Suppose $j \geq 1$. Recall that $(np)^{\ell-2} = p^{-1}$. 
Since $\cL_i$ holds, using the induction hypothesis we obtain 
\[
|H^{(j)}(A)| \leq 16\epsilon^{-1} (np)^{\ell-2-j}n^{(2j\tau+\ell+j-1)\epsilon} \leq (np)^{\ell-2-j}n^{(2j\tau+\ell+j)\epsilon}   ,
\]
completing the proof. 
\end{proof} 
With Lemma~\ref{lem:bound:manypaths} in hand, combing \eqref{eq:closed:xySigma:inequality}--\eqref{eq:closed:xySigma:i1:inclusion} with $|B| = k \leq npn^{\epsilon}$ as well as $|I_B| \leq n^{2\vartheta\epsilon}$, and then using \eqref{eq:Cl-constants:Wepsmu}, \eqref{eq:Cl-parameters4} as well as $\ell \geq 4$, $np=n^{1/(\ell-1)}$ and $(np)^2 \leq (np)^{\ell-2} = p^{-1}$, we deduce that  
\[
|L_{\Sigma^*}(i)| \leq npn^{\epsilon} \cdot n^{2\ell\tau\epsilon} + n^{2\vartheta\epsilon} \cdot npn^{2\ell\tau\epsilon}  \leq npn^{5\vartheta\epsilon} < (np)^{2} n^{-1/(2\ell)} \leq p^{-1} n^{-1/(2\ell)}   , 
\]
which establishes $\neg\badEC[2,i]$.

\subsection{Few tuples are ignored for $\Sigma^*$} 
\label{sec:good_configuration:few_ignored} 
In this section we estimate the size of $T_{\Sigma^*,\ell-3}(i) \setminus Z_{\Sigma^*,\ell-3}(i)$. 
Let $Q_{\Sigma^*}(i)$ contain all pairs $(w_1, w_\ell) \in B \times N^{(\ell-3)}(A,R)$ for which there exists a path $w_1 \cdots w_{\ell}$ with $w_2 \in I_B \cup M^{(\ell-2)}(A)$. 
We claim that 
\begin{equation}
\label{eq:bad:inequality}
|T_{\Sigma^*,\ell-3}(i) \setminus Z_{\Sigma^*,\ell-3}(i)| \leq |Q_{\Sigma^*}(i)|   . 
\end{equation} 
Every tuple $(v_0, \ldots, v_{\ell-2}) \in T_{\Sigma^*,\ell-3}(i) \setminus Z_{\Sigma^*,\ell-3}(i)$ was ignored in one of the first $i$ steps because (R2) failed. 
Recall that $C_{x,y,\Sigma^*}(i,j)$ contains all pairs $bw \in B \times N^{(\ell-3)}(A,R)$ for which there exist disjoint paths $b=w_1 \cdots w_{j}=x$ and $y=w_{j+1} \cdots w_{\ell}=w$ in $G(i)$. 
Observe that for every ignored tuple there exists $i' < i$, distinct $x,y \in [n]$ and $j \in [\ell-1]$ with $e_{i'+1} =xy$, $f_{\ell-2} \in C_{x,y,\Sigma}(i',j)$ and $|C_{x,y,\Sigma}(i',j)| > p^{-1} n^{-30\ell\epsilon}$. 
So, since $e_{i'+1} = xy$ was added, for every such tuple there exists a path $v_{\ell-2}=w_1 \cdots w_{j}w_{j+1} \cdots w_{\ell}=v_{\ell-3}$ with $w_{j}=x$ and $w_{j+1}=y$ in $G(i'+1) \subseteq G(i)$. 
Note that by monotonicity we have $C_{x,y,\Sigma^*}(i',j) \subseteq C_{x,y,\Sigma^*}(i,j)$, and therefore all such `bad' pairs $xy$ satisfy $|C_{x,y,\Sigma^*}(i,j)| > p^{-1} n^{-30\ell\epsilon}$. 
By the findings of Section~\ref{sec:good_configuration:bad:1} it thus suffices to consider $C_{x,y,\Sigma^*}(i,j)$ for $xy \in L_{\Sigma^*}(i,j)$ with $j \in \{1,2\}$, since for all others \eqref{eq:closed:xySigma:small:jl} holds. 
Now, using \eqref{eq:closed:xySigma:i2:inclusion} and \eqref{eq:closed:xySigma:i1:inclusion}, it is not difficult to see that the corresponding paths $v_{\ell-2}=w_1 \cdots w_{\ell}=v_{\ell-3}$ satisfy $w_1 \in B$, $w_2 \in I_B \cup M^{(\ell-2)}(A)$ and $w_{\ell} \in N^{(\ell-3)}(A,R)$. 
Putting things together,  the extension property $\cU_T$ (cf.\ Lemma~\ref{lemma:extension:property}) implies \eqref{eq:bad:inequality}, since every $(v_0, \ldots, v_{\ell-2}) \in T_{\Sigma^*,\ell-3}(i) \setminus Z_{\Sigma^*,\ell-3}(i)$ is uniquely determined by the pair $f_{\ell-2}=v_{\ell-3}v_{\ell-2}$.

Let $Q_{\Sigma^*,I}(i)$ and  $Q_{\Sigma^*,M}(i)$ contain all pairs $(w_1,w_\ell) \in Q_{\Sigma^*}(i)$ where at least one corresponding path $w_1 \cdots w_{\ell}$ satisfies $w_2 \in I_B$ and $w_2 \in M^{(\ell-2)}(A) \setminus I_B$, respectively. 
Now, using \eqref{eq:Cl-parameters3} and \eqref{eq:bad:inequality}, to establish \eqref{eq:lem:dem:config:ignored}, it  suffices to prove, say,
\begin{equation}
\label{eq:bad:paths:estimate}
\max\{|Q_{\Sigma^*,I}(i)|,|Q_{\Sigma^*,M}(i)|\} \leq (np)^{\ell-1}n^{-15\epsilon}   . 
\end{equation} 
Using Lemma~\ref{lem:bound_paths_IB}, $|I_B| \leq n^{2\vartheta\epsilon}$ and that the degree is at most $npn^{\epsilon}$, we obtain, with room to spare, 
\begin{equation*}
\label{eq:bad:paths:estimate:small:Q1}
|Q_{\Sigma^*,I}(i)| \leq npn^{\epsilon} \cdot |I_B| \cdot (np)^{\ell-3} n^{15 \ell\epsilon} \leq (np)^{\ell-2} n^{(15 \ell + 2\vartheta + 1)\epsilon} \leq (np)^{\ell-1}n^{-15\epsilon}   . 
\end{equation*} 
Turning to $Q_{\Sigma^*,M}(i)$, note that for every $w_2 \in M^{(\ell-2)}(A) \setminus I_B$ we have $|\Gamma(w_2) \cap B| \leq npn^{-\vartheta\epsilon}$ by \eqref{eq:vnotinIAB:neighbourhood:bound}. 
With a similar argument as above, using Lemma~\ref{lem:bound:manypaths}, i.e., $|M^{(\ell-2)}(A)| \leq n^{2\ell\tau\epsilon}$, we see that
\begin{equation*}
\label{eq:bad:paths:estimate:small:Q2}
|Q_{\Sigma^*,M}(i)| \leq npn^{-\vartheta\epsilon} \cdot |M^{(\ell-2)}(A)| \cdot (npn^{\epsilon})^{\ell-2} \leq (np)^{\ell-1} n^{(2\ell\tau+\ell-\vartheta)\epsilon} \leq (np)^{\ell-1}n^{-15\epsilon}  , 
\end{equation*} 
where the last inequality follows from \eqref{eq:Cl-parameters4}, i.e., $\vartheta = 20\ell\tau$. 
This establishes \eqref{eq:bad:paths:estimate}, which, as explained,  completes the proof of Lemma~\ref{lem:dem:config}. \qed

\bigskip{\bf Acknowledgements.} 
I would like to thank my supervisor Oliver Riordan for many stimulating discussions and helpful comments on an earlier version of this paper. 
Part of this research was done while visiting the University of Memphis, and I am grateful for the hospitality and great working conditions. 
Finally, I would also like to thank the anonymous referees for several suggestions improving the presentation of the paper.

\small\begin{spacing}{0.4}
\bibliographystyle{plain}

\end{spacing}

\end{document}